\documentclass{amsart}

\usepackage{amssymb,amsmath,stmaryrd,mathrsfs}
\def\definetac{\newif\iftac}    
\ifx\tactrue\undefined
  \definetac
  \ifx\state\undefined\tacfalse\else\tactrue\fi
\fi
\iftac\else\usepackage{amsthm}\fi
\usepackage[all,2cell]{xy}
\UseAllTwocells
\usepackage{enumitem}
\usepackage{xcolor}
\definecolor{darkgreen}{rgb}{0,0.45,0} 
\usepackage[pagebackref,colorlinks,citecolor=darkgreen,linkcolor=darkgreen]{hyperref}
\usepackage{mathtools}          
\usepackage{tikz}
\usepackage{braket}             

\usepackage{url}                
\usepackage{xspace}             

\makeatletter
\let\ea\expandafter

\def\mdef#1#2{\ea\ea\ea\gdef\ea\ea\noexpand#1\ea{\ea\ensuremath\ea{#2}\xspace}}
\def\alwaysmath#1{\ea\ea\ea\global\ea\ea\ea\let\ea\ea\csname your@#1\endcsname\csname #1\endcsname
  \ea\def\csname #1\endcsname{\ensuremath{\csname your@#1\endcsname}\xspace}}

\DeclareRobustCommand\widecheck[1]{{\mathpalette\@widecheck{#1}}}
\def\@widecheck#1#2{%
    \setbox\z@\hbox{\m@th$#1#2$}%
    \setbox\tw@\hbox{\m@th$#1%
       \widehat{%
          \vrule\@width\z@\@height\ht\z@
          \vrule\@height\z@\@width\wd\z@}$}%
    \dp\tw@-\ht\z@
    \@tempdima\ht\z@ \advance\@tempdima2\ht\tw@ \divide\@tempdima\thr@@
    \setbox\tw@\hbox{%
       \raise\@tempdima\hbox{\scalebox{1}[-1]{\lower\@tempdima\box
\tw@}}}%
    {\ooalign{\box\tw@ \cr \box\z@}}}

\newcount\foreachcount

\def\foreachletter#1#2#3{\foreachcount=#1
  \ea\loop\ea\ea\ea#3\@alph\foreachcount
  \advance\foreachcount by 1
  \ifnum\foreachcount<#2\repeat}

\def\foreachLetter#1#2#3{\foreachcount=#1
  \ea\loop\ea\ea\ea#3\@Alph\foreachcount
  \advance\foreachcount by 1
  \ifnum\foreachcount<#2\repeat}

\def\definescr#1{\ea\gdef\csname s#1\endcsname{\ensuremath{\mathscr{#1}}\xspace}}
\foreachLetter{1}{27}{\definescr}
\def\definecal#1{\ea\gdef\csname c#1\endcsname{\ensuremath{\mathcal{#1}}\xspace}}
\foreachLetter{1}{27}{\definecal}
\def\definebold#1{\ea\gdef\csname b#1\endcsname{\ensuremath{\mathbf{#1}}\xspace}}
\foreachLetter{1}{27}{\definebold}
\def\definebb#1{\ea\gdef\csname l#1\endcsname{\ensuremath{\mathbb{#1}}\xspace}}
\foreachLetter{1}{27}{\definebb}
\def\definefrak#1{\ea\gdef\csname f#1\endcsname{\ensuremath{\mathfrak{#1}}\xspace}}
\foreachletter{1}{9}{\definefrak} 
\foreachletter{10}{27}{\definefrak}
\def\definebar#1{\ea\gdef\csname #1bar\endcsname{\ensuremath{\overline{#1}}\xspace}}
\foreachLetter{1}{27}{\definebar}
\foreachletter{1}{8}{\definebar} 
\foreachletter{9}{15}{\definebar} 
\foreachletter{16}{27}{\definebar}
\def\definetil#1{\ea\gdef\csname #1til\endcsname{\ensuremath{\widetilde{#1}}\xspace}}
\foreachLetter{1}{27}{\definetil}
\foreachletter{1}{27}{\definetil}
\def\definehat#1{\ea\gdef\csname #1hat\endcsname{\ensuremath{\widehat{#1}}\xspace}}
\foreachLetter{1}{27}{\definehat}
\foreachletter{1}{27}{\definehat}
\def\definechk#1{\ea\gdef\csname #1chk\endcsname{\ensuremath{\widecheck{#1}}\xspace}}
\foreachLetter{1}{27}{\definechk}
\foreachletter{1}{27}{\definechk}
\def\defineul#1{\ea\gdef\csname u#1\endcsname{\ensuremath{\underline{#1}}\xspace}}
\foreachLetter{1}{27}{\defineul}
\foreachletter{1}{27}{\defineul}

\def\autofmt@n#1\autofmt@end{\mathrm{#1}}
\def\autofmt@b#1\autofmt@end{\mathbf{#1}}
\def\autofmt@l#1#2\autofmt@end{\mathbb{#1}\mathsf{#2}}
\def\autofmt@c#1#2\autofmt@end{\mathcal{#1}\mathit{#2}}
\def\autofmt@s#1#2\autofmt@end{\mathscr{#1}\mathit{#2}}
\def\autofmt@f#1\autofmt@end{\mathsf{#1}}
\def\autofmt@u#1\autofmt@end{\underline{\smash{\mathsf{#1}}}}
\def\autofmt@U#1\autofmt@end{\underline{\underline{\smash{\mathsf{#1}}}}}
\def\autofmt@h#1\autofmt@end{\widehat{#1}}
\def\autofmt@r#1\autofmt@end{\overline{#1}}
\def\autofmt@t#1\autofmt@end{\widetilde{#1}}
\def\autofmt@k#1\autofmt@end{\check{#1}}

\def\auto@drop#1{}
\def\autodef#1{\ea\ea\ea\@autodef\ea\ea\ea#1\ea\auto@drop\string#1\autodef@end}
\def\@autodef#1#2#3\autodef@end{%
  \ea\def\ea#1\ea{\ea\ensuremath\ea{\csname autofmt@#2\endcsname#3\autofmt@end}\xspace}}
\def\autodefs@end{blarg!}
\def\autodefs#1{\@autodefs#1\autodefs@end}
\def\@autodefs#1{\ifx#1\autodefs@end%
  \def\autodefs@next{}%
  \else%
  \def\autodefs@next{\autodef#1\@autodefs}%
  \fi\autodefs@next}

\DeclareSymbolFont{bbold}{U}{bbold}{m}{n}
\DeclareSymbolFontAlphabet{\mathbbb}{bbold}

\newcommand{\bbone}{\ensuremath{\mathbbb{1}}\xspace}

\mdef\delbar{\overline{\partial}}

\mdef\hf{\textstyle\frac12 }
\mdef\thrd{\textstyle\frac13 }
\mdef\qtr{\textstyle\frac14 }

\newcommand{\op}{^{\mathrm{op}}}

\SelectTips{cm}{}
\newdir{ >}{{}*!/-10pt/@{>}}    
\newcommand{\pushoutcorner}[1][dr]{\save*!/#1+1.2pc/#1:(1,-1)@^{|-}\restore}

\mdef\Id{\mathrm{Id}}
\mdef\id{\mathrm{id}}
\alwaysmath{ell}
\alwaysmath{infty}
\alwaysmath{odot}
\def\frc#1/#2.{\frac{#1}{#2}}   
\mdef\ten{\mathrel{\otimes}}

\mdef\sqten{\mathrel{\boxtimes}}

\DeclareRobustCommand\widecheck[1]{{\mathpalette\@widecheck{#1}}}
\def\@widecheck#1#2{%
    \setbox\z@\hbox{\m@th$#1#2$}%
    \setbox\tw@\hbox{\m@th$#1%
       \widehat{%
          \vrule\@width\z@\@height\ht\z@
          \vrule\@height\z@\@width\wd\z@}$}%
    \dp\tw@-\ht\z@
    \@tempdima\ht\z@ \advance\@tempdima2\ht\tw@ \divide\@tempdima\thr@@
    \setbox\tw@\hbox{%
       \raise\@tempdima\hbox{\scalebox{1}[-1]{\lower\@tempdima\box
\tw@}}}%
    {\ooalign{\box\tw@ \cr \box\z@}}}

\DeclareMathOperator\colim{colim}

\DeclareMathOperator\im{im}

\DeclareMathOperator\Ho{Ho}

\newcommand{\ot}{\ensuremath{\leftarrow}}

\mdef\we{\overset{\sim}{\longrightarrow}}
\mdef\leftwe{\overset{\sim}{\longleftarrow}}

\let\xto\xrightarrow

\def\rightarrowtailfill@{\arrowfill@{\Yright\joinrel\relbar}\relbar\rightarrow}
\newcommand\xrightarrowtail[2][]{\ext@arrow 0055{\rightarrowtailfill@}{#1}{#2}}

\def\twoheadrightarrowfill@{\arrowfill@{\relbar\joinrel\relbar}\relbar\twoheadrightarrow}
\newcommand\xtwoheadrightarrow[2][]{\ext@arrow 0055{\twoheadrightarrowfill@}{#1}{#2}}

\def\slashedarrowfill@#1#2#3#4#5{%
  $\m@th\thickmuskip0mu\medmuskip\thickmuskip\thinmuskip\thickmuskip
   \relax#5#1\mkern-7mu%
   \cleaders\hbox{$#5\mkern-2mu#2\mkern-2mu$}\hfill
   \mathclap{#3}\mathclap{#2}%
   \cleaders\hbox{$#5\mkern-2mu#2\mkern-2mu$}\hfill
   \mkern-7mu#4$%
}
\def\rightslashedarrowfill@{%
  \slashedarrowfill@\relbar\relbar\mapstochar\rightarrow}
\newcommand\xslashedrightarrow[2][]{%
  \ext@arrow 0055{\rightslashedarrowfill@}{#1}{#2}}
\mdef\hto{\xslashedrightarrow{}}
\mdef\htoo{\xslashedrightarrow{\quad}}

\long\def\my@drawfill#1#2;{%
\@skipfalse
\fill[#1,draw=none] #2;
\@skiptrue
\draw[#1,fill=none] #2;
}
\newif\if@skip
\newcommand{\skipit}[1]{\if@skip\else#1\fi}
\newcommand{\drawfill}[1][]{\my@drawfill{#1}}

\newif\ifhyperref
\@ifpackageloaded{hyperref}{\hyperreftrue}{\hyperreffalse}
\iftac
  \let\your@state\state
  \def\state#1{\gdef\currthmtype{#1}\your@state{#1}}
  \let\your@staterm\staterm
  \def\staterm#1{\gdef\currthmtype{#1}\your@staterm{#1}}
  \let\defthm\newtheorem
  \def\currthmtype{}
  \ifhyperref
    \def\autoref#1{\ref*{label@name@#1}~\ref{#1}}
  \else
    \def\autoref#1{\ref{label@name@#1}~\ref{#1}}
  \fi
  \AtBeginDocument{%
    \let\old@label\label%
    \def\label#1{%
      {\let\your@currentlabel\@currentlabel%
        \edef\@currentlabel{\currthmtype}%
        \old@label{label@name@#1}}%
      \old@label{#1}}
  }
\else
  \ifhyperref
    \def\defthm#1#2{%
      \newtheorem{#1}{#2}[section]%
      \expandafter\def\csname #1autorefname\endcsname{#2}%
      \expandafter\let\csname c@#1\endcsname\c@thm}
  \else
    \def\defthm#1#2{\newtheorem{#1}[thm]{#2}}
    \ifx\SK@label\undefined\let\SK@label\label\fi
    \let\old@label\label
    \let\your@thm\@thm
    \def\@thm#1#2#3{\gdef\currthmtype{#3}\your@thm{#1}{#2}{#3}}
    \def\currthmtype{}
    \def\label#1{{\let\your@currentlabel\@currentlabel\def\@currentlabel%
        {\currthmtype~\your@currentlabel}%
        \SK@label{#1@}}\old@label{#1}}
    \def\autoref#1{\ref{#1@}}
  \fi
\fi

\newtheorem{thm}{Theorem}[section]

\defthm{cor}{Corollary}
\defthm{prop}{Proposition}
\defthm{lem}{Lemma}
\defthm{sch}{Scholium}
\defthm{assume}{Assumption}
\defthm{claim}{Claim}
\defthm{conj}{Conjecture}
\defthm{hyp}{Hypothesis}
\defthm{fact}{Fact}
\iftac\theoremstyle{plain}\else\theoremstyle{definition}\fi
\defthm{defn}{Definition}
\defthm{notn}{Notation}
\iftac\theoremstyle{plain}\else\theoremstyle{remark}\fi
\defthm{rmk}{Remark}
\defthm{eg}{Example}
\defthm{egs}{Examples}
\defthm{ex}{Exercise}
\defthm{ceg}{Counterexample}
\defthm{warn}{Warning}
\defthm{con}{Construction}

\def\thmqedhere{\expandafter\csname\csname @currenvir\endcsname @qed\endcsname}

\setitemize[1]{leftmargin=2em}
\setenumerate[1]{leftmargin=*}

\iftac
  \let\c@equation\c@subsection
\else
  \let\c@equation\c@thm
\fi
\numberwithin{equation}{section}

\@ifpackageloaded{mathtools}{\mathtoolsset{showonlyrefs,showmanualtags}}{}

\alwaysmath{alpha}
\alwaysmath{beta}
\alwaysmath{gamma}
\alwaysmath{Gamma}
\alwaysmath{delta}
\alwaysmath{Delta}
\alwaysmath{epsilon}
\mdef\ep{\varepsilon}
\alwaysmath{zeta}
\alwaysmath{eta}
\alwaysmath{theta}
\alwaysmath{Theta}
\alwaysmath{iota}
\alwaysmath{kappa}
\alwaysmath{lambda}
\alwaysmath{Lambda}
\alwaysmath{mu}
\alwaysmath{nu}
\alwaysmath{xi}
\alwaysmath{pi}
\alwaysmath{rho}
\alwaysmath{sigma}
\alwaysmath{Sigma}
\alwaysmath{tau}
\alwaysmath{upsilon}
\alwaysmath{Upsilon}
\alwaysmath{phi}
\alwaysmath{Pi}
\alwaysmath{Phi}
\mdef\ph{\varphi}
\alwaysmath{chi}
\alwaysmath{psi}
\alwaysmath{Psi}
\alwaysmath{omega}
\alwaysmath{Omega}

\makeatother

\usetikzlibrary{calc}	

\pgfdeclarelayer{bg}    
\pgfsetlayers{bg,main}  

\newcommand{\SqrtThree}{1.732}

\newcommand{\TikzCplusminus}[3]{
  \draw (3.1+#2, 1.0+#3) ellipse (1.5 and 1.0);
  \draw (4.1+#2,1.6+#3) node[anchor=south west] {$C$};

  \draw[white,double,ultra thick] (2.4+#2, 1.5+#3) .. controls (1.6+#2, 2.0+#3) and (0.8+#2, 2.0+#3) .. (0.0+#2, 1.1+#3);
  \draw[white,double,ultra thick] (2.4+#2, 0.6+#3) .. controls (1.6+#2, 1.1+#3) and (0.8+#2, 1.1+#3) .. (0.1+#2, 1.0+#3);
  \draw[white,double,ultra thick] (2.4+#2, 0.5+#3) .. controls (1.6+#2, 0.0+#3) and (0.8+#2, 0.0+#3) .. (0.0+#2, 0.9+#3);

  \draw[#1] (2.4+#2, 1.5+#3) .. controls (1.6+#2, 2.0+#3) and (0.8+#2, 2.0+#3) .. (0.0+#2, 1.1+#3);
  \draw[#1] (2.4+#2, 0.6+#3) .. controls (1.6+#2, 1.1+#3) and (0.8+#2, 1.1+#3) .. (0.1+#2, 1.0+#3);
  \draw[#1] (2.4+#2, 0.5+#3) .. controls (1.6+#2, 0.0+#3) and (0.8+#2, 0.0+#3) .. (0.0+#2, 0.9+#3);

  \draw (0.15+#2 ,1.0+#3) node[anchor=east] {$v\;\cdot$};
  \draw (2.3+#2 ,1.4+#3) node[anchor=west] {$\cdot\;y_1$};
  \draw (2.3+#2 ,0.5+#3) node[anchor=west] {$\cdot\;y_2=y_3$};
}

\newcommand{\TikzCplusminusInj}[3]{
  \draw (3.1+#2, 1.0+#3) ellipse (1.5 and 1.0);
  \draw (4.1+#2,1.6+#3) node[anchor=south west] {$C$};

  \draw[white,double,ultra thick] (2.4+#2, 1.5+#3) .. controls (1.6+#2, 2.0+#3) and (0.8+#2, 2.0+#3) .. (0.0+#2, 1.1+#3);
  \draw[white,double,ultra thick] (2.4+#2, 1.0+#3) -- (0.1+#2, 1.0+#3);
  \draw[white,double,ultra thick] (2.4+#2, 0.5+#3) .. controls (1.6+#2, 0.0+#3) and (0.8+#2, 0.0+#3) .. (0.0+#2, 0.9+#3);

  \draw[#1] (2.4+#2, 1.5+#3) .. controls (1.6+#2, 2.0+#3) and (0.8+#2, 2.0+#3) .. (0.0+#2, 1.1+#3);
  \draw[#1] (2.4+#2, 1.0+#3) -- (0.1+#2, 1.0+#3);
  \draw[#1] (2.4+#2, 0.5+#3) .. controls (1.6+#2, 0.0+#3) and (0.8+#2, 0.0+#3) .. (0.0+#2, 0.9+#3);

  \draw (0.15+#2, 1.0+#3) node[anchor=east] {$v\;\cdot$};
  \draw (2.3+#2, 1.4+#3) node[anchor=west] {$\cdot\;y_1$};
  \draw (2.3+#2, 0.95+#3) node[anchor=west] {$\cdot\;y_2$};
  \draw (2.3+#2, 0.5+#3) node[anchor=west] {$\cdot\;y_3$};
}

\newcommand{\TikzDplusminus}[3]{
  \draw (4.1+#2, 1.0+#3) ellipse (1.5 and 1.0);
  \draw (5.1+#2,1.6+#3) node[anchor=south west] {$C$};

  \draw[white,double,ultra thick] (2.2+#2, 1.8+#3) .. controls (2.6+#2, 1.8+#3) and (3.0+#2, 1.7+#3) .. (3.4+#2, 1.5+#3);
  \draw[white,double,ultra thick] (2.2+#2, 1.0+#3) .. controls (2.6+#2, 1.0+#3) and (3.0+#2, 0.9+#3) .. (3.4+#2, 0.6+#3);
  \draw[white,double,ultra thick] (2.2+#2, 0.2+#3) .. controls (2.6+#2, 0.2+#3) and (3.0+#2, 0.3+#3) .. (3.4+#2, 0.5+#3);

  \draw[#1] (2.0+#2, 1.8+#3) -- (0.0+#2, 1.1+#3);
  \draw[#1] (2.0+#2, 1.0+#3) -- (0.1+#2, 1.0+#3);
  \draw[#1] (2.0+#2, 0.2+#3) -- (0.0+#2, 0.9+#3);

  \draw[->] (2.2+#2, 1.8+#3) .. controls (2.6+#2, 1.8+#3) and (3.0+#2, 1.7+#3) .. (3.4+#2, 1.5+#3);
  \draw[->] (2.2+#2, 1.0+#3) .. controls (2.6+#2, 1.0+#3) and (3.0+#2, 0.9+#3) .. (3.4+#2, 0.6+#3);
  \draw[->] (2.2+#2, 0.2+#3) .. controls (2.6+#2, 0.2+#3) and (3.0+#2, 0.3+#3) .. (3.4+#2, 0.5+#3);

  \draw (0.15+#2 ,1.0+#3) node[anchor=east] {$v\;\cdot$};

  \draw (2.1+#2, 1.8+#3) node[anchor=center] {$\cdot$};
  \draw (2.1+#2, 1.8+#3) node[anchor=south] {$x_1$};
  \draw (2.1+#2, 1.0+#3) node[anchor=center] {$\cdot$};
  \draw (2.1+#2, 1.0+#3) node[anchor=south] {$x_2$};
  \draw (2.1+#2, 0.2+#3) node[anchor=center] {$\cdot$};
  \draw (2.1+#2, 0.2+#3) node[anchor=north] {$x_3$};

  \draw (3.3+#2 ,1.4+#3) node[anchor=west] {$\cdot\;y_1$};
  \draw (3.3+#2 ,0.5+#3) node[anchor=west] {$\cdot\;y_2=y_3$};
}

\newcommand{\TikzDplusminusInj}[3]{
  \draw (4.1+#2, 1.0+#3) ellipse (1.5 and 1.0);
  \draw (5.1+#2,1.6+#3) node[anchor=south west] {$C$};

  \draw[white,double,ultra thick] (2.2+#2, 1.8+#3) .. controls (2.6+#2, 1.8+#3) and (3.0+#2, 1.7+#3) .. (3.4+#2, 1.5+#3);
  \draw[white,double,ultra thick] (2.2+#2, 1.0+#3) -- (3.4+#2, 1.0+#3);
  \draw[white,double,ultra thick] (2.2+#2, 0.2+#3) .. controls (2.6+#2, 0.2+#3) and (3.0+#2, 0.3+#3) .. (3.4+#2, 0.5+#3);

  \draw[#1] (2.0+#2, 1.8+#3) -- (0.0+#2, 1.1+#3);
  \draw[#1] (2.0+#2, 1.0+#3) -- (0.1+#2, 1.0+#3);
  \draw[#1] (2.0+#2, 0.2+#3) -- (0.0+#2, 0.9+#3);

  \draw[->] (2.2+#2, 1.8+#3) .. controls (2.6+#2, 1.8+#3) and (3.0+#2, 1.7+#3) .. (3.4+#2, 1.5+#3);
  \draw[->] (2.2+#2, 1.0+#3) -- (3.4+#2, 1.0+#3);
  \draw[->] (2.2+#2, 0.2+#3) .. controls (2.6+#2, 0.2+#3) and (3.0+#2, 0.3+#3) .. (3.4+#2, 0.5+#3);

  \draw (0.15+#2 ,1.0+#3) node[anchor=east] {$v\;\cdot$};

  \draw (2.1+#2, 1.8+#3) node[anchor=center] {$\cdot$};
  \draw (2.1+#2, 1.8+#3) node[anchor=south] {$x_1$};
  \draw (2.1+#2, 1.0+#3) node[anchor=center] {$\cdot$};
  \draw (2.1+#2, 1.0+#3) node[anchor=south] {$x_2$};
  \draw (2.1+#2, 0.2+#3) node[anchor=center] {$\cdot$};
  \draw (2.1+#2, 0.2+#3) node[anchor=north] {$x_3$};

  \draw (3.3+#2 ,1.4+#3) node[anchor=west] {$\cdot\;y_1$};
  \draw (3.3+#2, 0.95+#3) node[anchor=west] {$\cdot\;y_2$};
  \draw (3.3+#2, 0.5+#3) node[anchor=west] {$\cdot\;y_3$};
}

\newcommand{\TikzBiprodCubeCoordinates}[5]{
  \coordinate (vertex1) at ( #1 - #3/2 + #3*#5,         #2 + #3*\SqrtThree/2    );
  \coordinate (vertex2) at ( #1 + #3/2 - #3*#4 + #3*#5, #2 + #3*\SqrtThree/2    );
  \coordinate (vertex3) at ( #1 + #3 - #3*#4/2 + #3*#5, #2 - #3*#4*\SqrtThree/2 );
  \coordinate (vertex4) at ( #1 + #3/2 - #3*#5,         #2 - #3*\SqrtThree/2    );
  \coordinate (vertex5) at ( #1 - #3/2 + #3*#4 - #3*#5, #2 - #3*\SqrtThree/2    );
  \coordinate (vertex6) at ( #1 - #3 + #3*#4/2 - #3*#5, #2 + #3*#4*\SqrtThree/2 );

  \coordinate (vertex-front) at ( #1 + #3*#4/2 + #3*#5, #2 - #3*#4*\SqrtThree/2 );
  \coordinate (vertex-back)  at ( #1 - #3*#4/2 - #3*#5, #2 + #3*#4*\SqrtThree/2 );

  \coordinate (vertex12) at ( #1 - #3*#4/2 + #3*#5,            #2 + #3*\SqrtThree/2                      );
  \coordinate (vertex23) at ( #1 + #3*3/4 - #3*#4*3/4 + #3*#5, #2 + #3*\SqrtThree/4 - #3*#4*\SqrtThree/4 );
  \coordinate (vertex34) at ( #1 + #3*3/4 - #3*#4/4,           #2 - #3*\SqrtThree/4 - #3*#4*\SqrtThree/4 );
  \coordinate (vertex45) at ( #1 + #3*#4/2 - #3*#5,            #2 - #3*\SqrtThree/2                      );
  \coordinate (vertex56) at ( #1 - #3*3/4 + #3*#4*3/4 - #3*#5, #2 - #3*\SqrtThree/4 + #3*#4*\SqrtThree/4 );
  \coordinate (vertex61) at ( #1 - #3*3/4 + #3*#4/4,           #2 + #3*\SqrtThree/4 + #3*#4*\SqrtThree/4 );

  \coordinate (vertex-6-to-back)  at ( #1 - #3/2 - #3*#5, #2 + #3*#4*\SqrtThree/2 );
  \coordinate (vertex-front-to-3) at ( #1 + #3/2 + #3*#5, #2 - #3*#4*\SqrtThree/2 );

  \coordinate (vertex-center) at ( #1, #2 );
}

\newcommand{\TikzDrawBiproductCube}[3]{
  \draw[#3] ($(vertex6)!0.05!(vertex1)$) -- ($(vertex6)!0.95!(vertex1)$);
  \draw[#3] ($(vertex6)!0.05!(vertex5)$) -- ($(vertex6)!0.95!(vertex5)$);
  \draw[#3] ($(vertex6)!0.05!(vertex-back)$) -- ($(vertex6)!0.95!(vertex-back)$);

  \draw[#3] ($(vertex-back)!0.05!(vertex2)$) -- ($(vertex-back)!0.95!(vertex2)$);
  \draw[#3] ($(vertex-back)!0.05!(vertex4)$) -- ($(vertex-back)!0.95!(vertex4)$);

  \draw[#3] ($(vertex1)!0.05!(vertex2)$) -- ($(vertex1)!0.95!(vertex2)$);
  \draw[#3] ($(vertex2)!0.1!(vertex23)$) -- ($(vertex2)!0.9!(vertex23)$);
  \draw[#3] ($(vertex23)!0.1!(vertex3)$) -- ($(vertex23)!0.9!(vertex3)$);

  \draw[#3] ($(vertex5)!0.05!(vertex4)$) -- ($(vertex5)!0.95!(vertex4)$);
  \draw[#3] ($(vertex4)!0.1!(vertex34)$) -- ($(vertex4)!0.9!(vertex34)$);
  \draw[#3] ($(vertex34)!0.1!(vertex3)$) -- ($(vertex34)!0.9!(vertex3)$);

  \begin{scope}
    \clip (vertex61) rectangle (vertex34);
    \draw[white,double,ultra thick] (vertex1) -- (vertex-front);
    \draw[white,double,ultra thick] (vertex5) -- (vertex-front);
    \draw[white,double,ultra thick] (vertex-front) -- (vertex3);
  \end{scope}

  \draw[#3] ($(vertex1)!0.05!(vertex-front)$) -- ($(vertex1)!0.95!(vertex-front)$);
  \draw[#3] ($(vertex5)!0.05!(vertex-front)$) -- ($(vertex5)!0.95!(vertex-front)$);
  \draw[#3] ($(vertex-front)!0.1!(vertex-front-to-3)$) -- ($(vertex-front)!0.9!(vertex-front-to-3)$);
  \draw[#3] ($(vertex-front-to-3)!0.1!(vertex3)$) -- ($(vertex-front-to-3)!0.9!(vertex3)$);

  \draw (vertex-center) node {$\cdot b$};

  \draw (vertex23) node {$\cdot$};
  \draw ($(vertex23)+(-0.1, 0.1)$) node[anchor=south west] {$x_1$};
  \draw (vertex-front-to-3) node {$\cdot$};
  \draw ($(vertex-front-to-3) + (-0.1, 0.0)$) node[anchor=south] {$x_2$};
  \draw (vertex34) node {$\cdot$};
  \draw ($(vertex34)+(-0.1, -0.1)$) node[anchor=north west] {$x_3$};

  \draw (vertex1) node {$\cdot$};
  \draw (vertex2) node {$\cdot$};
  \draw (vertex3) node {$\cdot$};
  \draw (vertex4) node {$\cdot$};
  \draw (vertex5) node {$\cdot$};
  \draw (vertex6) node {$\cdot$};

  \draw (vertex-front) node {$\cdot$};
  \draw (vertex-back) node {$\cdot$};
}

\newcommand{\TikzEplusminusCoordinates}[2]{
  \TikzBiprodCubeCoordinates{#1}{#2}{1.5}{.08}{.22}

  \coordinate (source-sink)        at ($(vertex56) - (1.5, 0.3)$);
  \coordinate (v-to-b-arrow-start) at ($(source-sink)!0.05!(vertex-center)$);
  \coordinate (v-to-b-arrow-end)   at ($(source-sink)!0.95!(vertex-center)$);
  \coordinate (ellipse-center)     at (4.1+#1,      #2);
  \coordinate (obj-y1)             at (3.3+#1,  0.4+#2);
  \coordinate (obj-y23)            at (3.3+#1, -0.5+#2);
}

\newcommand{\TikzDrawEplusminus}[4]{
  \draw (ellipse-center) ellipse (1.5 and 1.0);
  \draw ($(ellipse-center) + (1.0, 0.6)$) node[anchor=south west] {$C$};
  \draw (obj-y1) node[anchor=west] {$\cdot\;y_1$};
  \draw (obj-y23) node[anchor=west] {$\cdot\;y_2=y_3$};

  \draw (source-sink) node {$\cdot$};
  \draw (source-sink) node[anchor=east] {$v$};
  \draw[#1] (v-to-b-arrow-start) -- (v-to-b-arrow-end);

  \begin{scope}
    \clip (vertex6) rectangle (vertex4);
    \draw[white,double,ultra thick] (vertex5) -- (vertex6);
  \end{scope}

  \TikzDrawBiproductCube{#3}{#4}{#2}

  \begin{scope}
    \clip (vertex-back) rectangle (vertex-front);
    \draw[white,double,ultra thick] (v-to-b-arrow-start) -- (v-to-b-arrow-end);
    \draw[#1] (v-to-b-arrow-start) -- (v-to-b-arrow-end);
  \end{scope}

  \draw[white,double,ultra thick] ($(vertex23) + (.1,.1)$) .. controls ($(vertex23)!0.3!(obj-y1) + (0,.3)$) and ($(vertex23)!0.7!(obj-y1) + (0,.5)$) .. ($(obj-y1) + (0.05,.1)$);
  \draw[->] ($(vertex23) + (.1,.1)$) .. controls ($(vertex23)!0.3!(obj-y1) + (0,.3)$) and ($(vertex23)!0.7!(obj-y1) + (0,.5)$) .. ($(obj-y1) + (0.05,.1)$);

  \draw[white,double,ultra thick] ($(vertex-front-to-3) + (.1,.1)$) .. controls ($(vertex-front-to-3)!0.3!(obj-y23) + (0,.5)$) and ($(vertex-front-to-3)!0.7!(obj-y23) + (0,.7)$) .. ($(obj-y23) + (0.05,.1)$);
  \draw[->] ($(vertex-front-to-3) + (.1,.1)$) .. controls ($(vertex-front-to-3)!0.3!(obj-y23) + (0,.5)$) and ($(vertex-front-to-3)!0.7!(obj-y23) + (0,.7)$) .. ($(obj-y23) + (0.05,.1)$);

  \draw[white,double,ultra thick] ($(vertex34) + (.1,-.05)$) .. controls ($(vertex34)!0.3!(obj-y23) - (0,.3)$) and ($(vertex34)!0.7!(obj-y23) - (0,.5)$) .. ($(obj-y23) + (0.05,0)$);
  \draw[->] ($(vertex34) + (.1,-.05)$) .. controls ($(vertex34)!0.3!(obj-y23) - (0,.3)$) and ($(vertex34)!0.7!(obj-y23) - (0,.5)$) .. ($(obj-y23) + (0.05,0)$);

  \draw (vertex1) node[anchor=south east] {$0$};
  \draw (vertex2) node[anchor=south west] {$0$};
  \draw (vertex3) node[anchor=north west] {$0$};
  \draw (vertex4) node[anchor=north west] {$0$};
  \draw (vertex5) node[anchor=north east] {$0$};
  \draw (vertex6) node[anchor=east] {$0$};
}

\newcommand{\TikzEplusminus}[4]{
  \TikzEplusminusCoordinates{#3}{#4}

  \TikzDrawEplusminus{#1}{#2}{#3}{#4}
}

\newcommand{\TikzCategoryFCoordinates}[2]{
  \TikzEplusminusCoordinates{#1}{#2}

  \coordinate (source)              at (source-sink);
  \coordinate (sink)                at ($(vertex34) + (0.4, -1.0)$);
  \coordinate (contractible-corner) at ($(source)+(sink)-(vertex-center)$);

  \coordinate (b-to-v'-arrow-start) at ($(vertex-center)!0.08!(sink)$);
  \coordinate (b-to-v'-arrow-end)   at ($(vertex-center)!0.95!(sink)$);

  \coordinate (v-to-0-arrow-start) at ($(source)!0.05!(contractible-corner)$);
  \coordinate (v-to-0-arrow-end)   at ($(source)!0.95!(contractible-corner)$);

  \coordinate (0-to-v'-arrow-start) at ($(contractible-corner)!0.05!(sink)$);
  \coordinate (0-to-v'-arrow-end)   at ($(contractible-corner)!0.95!(sink)$);
}

\newcommand{\TikzDrawCategoryF}[4]{
  \draw (sink) node {$\cdot$};
  \draw (sink) node[anchor=north west] {$v'$};
  \draw (contractible-corner) node {$\cdot$};
  \draw (contractible-corner) node[anchor=north] {$0$};

  \draw[#1] (b-to-v'-arrow-start) -- (b-to-v'-arrow-end);
  \draw[#1] (v-to-0-arrow-start) -- (v-to-0-arrow-end);
  \draw[#1] (0-to-v'-arrow-start) -- (0-to-v'-arrow-end);

  \draw[white,double,ultra thick] (vertex3) -- (vertex4);

  \TikzDrawEplusminus{#1}{#2}{#3}{#4}
}

\newcommand{\TikzCategoryF}[4]{
  \TikzCategoryFCoordinates{#3}{#4}

  \TikzDrawCategoryF{#1}{#2}{#3}{#4}
}

\newcommand{\TikzFigCplusminus}{
\begin{tikzpicture}
  \TikzCplusminus{<-}{0.0}{0.3}
  \draw (2.3, 0.0) node[anchor=north] {$\textrm{The category } C^-$};

  \TikzCplusminus{->}{6.5}{0.3}
  \draw (8.8, 0.0) node[anchor=north] {$\textrm{The category } C^+$};
\end{tikzpicture}
}

\newcommand{\TikzFigGlobalStrategy}{
\begin{tikzpicture}
  \begin{scope}[scale=.85]   
      \TikzCplusminus{<-,thick}{1.0}{4.2}
      \draw (3.3, 6.5) node[anchor=south] {$\textrm{The category } C^-$};

      \draw[<->,dashed] (4.1, 3.9) -- (4.1, 2.6);
      \draw (3.9, 3.25) node[anchor=east] {inflate/deflate};

      \TikzDplusminus{<-,thick}{0.0}{0.3}
      \draw (3.3, 0.0) node[anchor=north] {$\textrm{The category } D^-$};

      \draw[<->,dashed] (6.0, 0.0) .. controls (6.5, -.5) and (7.5, -.5) .. (8.0, 0.0);
      \draw (7.0, -.5) node[anchor=north] {reflect};

      \TikzDplusminus{->,thick}{7.5}{0.3}
      \draw (10.8, 0.0) node[anchor=north] {$\textrm{The category } D^+$};

      \draw[<->,dashed] (11.6, 3.9) -- (11.6, 2.6);
      \draw (11.4, 3.25) node[anchor=east] {inflate/deflate};

      \TikzCplusminus{->,thick}{8.5}{4.2}
      \draw (10.8, 6.5) node[anchor=south] {$\textrm{The category } C^+$};
  \end{scope}
\end{tikzpicture}
}

\newcommand{\TikzFigReflectionUnraveled}{
\begin{tikzpicture}
  \begin{scope}[scale=.85]   
    \TikzDplusminus{<-,thick}{-1.1}{3.3}
    \draw (-4.5, 4.7) node[anchor=north] {$D^-:$};

    \draw[->,dashed] (2.0, 2.7) -- (2.0, 1.7);
    \draw (2.3, 2.2) node[anchor=west] {insert a biproduct $n$-cube};

    \TikzEplusminus{->,thick}{->}{0}{0}
    \draw (-4.5, 0.5) node[anchor=north] {$E_1^-:$};

    \draw[->,dashed] (2.0, -2.0) -- (2.0, -3.0);
    \draw (2.3, -2.5) node[anchor=west] {make the $n$-cube invertible};

    \TikzEplusminus{->}{<->,thick}{0}{-5}
    \draw (-4.5, -4.5) node[anchor=north] {$E_2^-:$};

    \draw[->,dashed] (2.0, -7.0) -- (2.0, -8.0);
    \draw (2.3, -7.5) node[anchor=west] {reflect $v$ to $v'$};

    \TikzCategoryF{->,thick}{<->}{0}{-10}
    \draw (-4.5, -9.5) node[anchor=north] {$F:$};
  \end{scope}
\end{tikzpicture}
}

\newcommand{\TikzFigCompatibleGlueings}{
\begin{tikzpicture}
  \begin{scope}[scale=.85]
     \draw (2.0, 6.0) ellipse (2.0 and 1.5);
     \draw ($(2.0, 6.0) + (-1.5, 0.8)$) node[anchor=south east] {$A_1$};
     \draw (7.0, 6.0) ellipse (2.0 and 1.5);
     \draw ($(7.0, 6.0) + (1.5, 0.8)$) node[anchor=south west] {$A_2$};

     \draw (2.0, 1.5) ellipse (2.0 and 1.5);
     \draw ($(2.0, 1.5) + (-1.5, 0.8)$) node[anchor=south east] {$A'_1$};
     \draw (7.0, 1.5) ellipse (2.0 and 1.5);
     \draw ($(7.0, 1.5) + (1.5, 0.8)$) node[anchor=south west] {$A_2$};

     \draw[white,double,ultra thick] (3.3, 6.75) -- (6.0, 6.75);
     \draw[white,double,ultra thick] (3.3, 6.65) -- (6.0, 6.05);
     \draw[white,double,ultra thick] (3.3, 6.0 ) -- (6.0, 5.45);
     \draw[white,double,ultra thick] (3.3, 5.35) -- (6.0, 5.35);

     \draw[->] (3.3, 6.75) -- (6.0, 6.75);
     \draw[->] (3.3, 6.65) -- (6.0, 6.05);
     \draw[->] (3.3, 6.0 ) -- (6.0, 5.45);
     \draw[->] (3.3, 5.35) -- (6.0, 5.35);

     \draw[white,double,ultra thick] (3.3, 2.0 ) -- (6.0, 2.25);
     \draw[white,double,ultra thick] (3.3, 1.9 ) -- (6.0, 1.55);
     \draw[white,double,ultra thick] (3.3, 1.8 ) -- (6.0, 0.9 );
     \draw[white,double,ultra thick] (3.3, 1.0 ) -- (6.0, 0.8 );

     \draw[->] (3.3, 2.0 ) -- (6.0, 2.25);
     \draw[->] (3.3, 1.9 ) -- (6.0, 1.55);
     \draw[->] (3.3, 1.8 ) -- (6.0, 0.9 );
     \draw[->] (3.3, 1.0 ) -- (6.0, 0.8 );

     \draw (3.3, 6.7) node[anchor=east] {$s_1=s_2\;\cdot$};
     \draw (3.3, 6.0) node[anchor=east] {$s_3\;\cdot$};
     \draw (3.3, 5.3) node[anchor=east] {$s_4\;\cdot$};

     \draw (6.0, 6.7) node[anchor=west] {$\cdot\;t_1$};
     \draw (6.0, 6.0) node[anchor=west] {$\cdot\;t_2$};
     \draw (6.0, 5.3) node[anchor=west] {$\cdot\;t_3=t_4$};

     \draw (3.3, 1.9) node[anchor=east] {$s'_1=s'_2=s'_3\;\cdot$};
     \draw (3.3, 1.0) node[anchor=east] {$s'_4\;\cdot$};

     \draw (6.0, 2.2) node[anchor=west] {$\cdot\;t_1$};
     \draw (6.0, 1.5) node[anchor=west] {$\cdot\;t_2$};
     \draw (6.0, 0.8) node[anchor=west] {$\cdot\;t_3=t_4$};

     \draw[->] (2.0, 4.3) to node[left,midway] {$u_1$} (2.0, 3.2);
     \draw[->] (7.0, 4.3) to node[right,midway] {$\id$} (7.0, 3.2);
  \end{scope}
\end{tikzpicture}
}

\newcommand{\TikzFigSeparateSource}{
\begin{tikzpicture}
  \begin{scope}[scale=.85]   
      \TikzCplusminusInj{<-}{9.0}{0.3}
      \draw (11.8, 0.0) node[anchor=north] {$\textrm{The category } C^-$};

      \TikzDplusminusInj{<-}{1.0}{0.3}
      \draw (3.8, 0.0) node[anchor=north] {$\textrm{The category } D^-$};

      \draw (7.6, 0.3) node[anchor=center] {$\underset{u^-}\longrightarrow$};
  \end{scope}
\end{tikzpicture}
}

\newcommand{\TikzFigSeparateSink}{
\begin{tikzpicture}
  \begin{scope}[scale=.85]   
      \TikzCplusminusInj{->}{9.0}{0.3}
      \draw (11.8, 0.0) node[anchor=north] {$\textrm{The category } C^+$};

      \TikzDplusminusInj{->}{1.0}{0.3}
      \draw (3.8, 0.0) node[anchor=north] {$\textrm{The category } D^+$};

      \draw (7.6, 0.3) node[anchor=center] {$\underset{u^+}\longrightarrow$};
  \end{scope}
\end{tikzpicture}
}

\tikzset{lab/.style={auto,font=\scriptsize}} 
\usetikzlibrary{arrows}

\usepackage{xcolor}
\usepackage{fixme}
\fxsetup{
    status=draft,
    author=,
    layout=margin,
    theme=color
}

\definecolor{fxnote}{rgb}{1.0000,0.0000,0.0000}
\colorlet{fxnotebg}{yellow}

\newcommand{\D}{\sD}
\newcommand{\E}{\sE}

\autodefs{\cDER\cCat\cGrp\cCAT\cMONCAT\bSet\cProf\bCat
\cPro\cMONDER\cBICAT\ncomp\nid\ncoker\niso\nIm\sProf\ncyl\fC\fZ\fT\nhocolim\nholim\cPsNat\ncoll\nCh\bsSet\cAlg\cIm\cI\cP}

\def\cPDER{\ensuremath{\mathcal{PD}\mathit{ER}}\xspace}

\def\ho{\mathscr{H}\!\mathit{o}\xspace}

\let\oldboxtimes\boxtimes
\def\boxtimes{\mathrel{\oldboxtimes}}

\newcommand{\fib}{\mathsf{fib}}
\newcommand{\cof}{\mathsf{cof}}

\def\ccsub{_{\mathrm{cc}}}
\def\pdh(#1,#2){\llbracket #1,#2\rrbracket}
\def\ldh(#1,#2){\llbracket #1,#2\rrbracket\ccsub}
\def\pend(#1){\pdh(#1,#1)}
\def\lend(#1){\ldh(#1,#1)}

\def\DTl#1#2#3#4#5#6#7{%
  \xymatrix@C=3pc{{#1} \ar[r]^-{#2} &
    {#3} \ar[r]^-{#4} &
    {#5} \ar[r]^-{#6} &
    {#7}
  }}

\newsavebox{\tvabox}
\savebox\tvabox{\hspace{1mm}\begin{tikzpicture}[>=latex',baseline={(0,-.18)}]
  \draw[->] (0,.1) -- +(1,0);
  \node at (.5,0) {$\scriptscriptstyle\bot$};
  \draw[->] (1,-.1) -- +(-1,0);
  \draw[->] (1,-.2) -- +(-1,0);
\end{tikzpicture}\hspace{1mm}}

\newcommand{\exx}{\mathrm{ex}}

\newcommand{\modr}{\mathrm{mod}\,}

\newcommand{\sink}[1]{({#1}\cdot\bbone)^\rhd}
\newcommand{\source}[1]{({#1}\cdot\bbone)^\lhd}

\newcommand{\sse}{\stackrel{\mathrm{s}}{\sim}}

\title{Abstract tilting theory for quivers and related categories}

\author{Moritz Groth}
\address{Mathematisches Institut, Rheinische Friedrich-Wilhelms-Universit\"at Bonn, Endenicher Allee 60, D-53115 Bonn, Germany}
\email{mgroth@math.uni-bonn.de}

\author{Jan \v{S}\v{t}ov\'{\i}\v{c}ek}
\address{Department of Algebra, Charles University in Prague, Sokolovska 83, 186 75 Praha~8, Czech Republic}
\email{stovicek@karlin.mff.cuni.cz}


\date{\today}

\thanks{The second named author was supported by Neuron Fund for Support of Science.}

\begin{document}

\begin{abstract}
We generalize the construction of reflection functors from classical representation theory of quivers to arbitrary small categories with freely attached sinks or sources. These reflection morphisms are shown to induce equivalences between the corresponding representation theories with values in arbitrary stable homotopy theories, including representations over fields, rings or schemes as well as differential-graded and spectral representations. 

Specializing to representations over a field and to specific shapes, this recovers derived equivalences of Happel for finite, acyclic quivers. However, even over a field our main result leads to new derived equivalences for example for not necessarily finite or acyclic quivers. 

The results obtained here rely on a careful analysis of the compatibility of gluing constructions for small categories with homotopy Kan extensions and homotopical epimorphisms, as well as on a study of the combinatorics of amalgamations of categories. 
\end{abstract}

\maketitle

\tableofcontents

\section{Introduction}
\label{sec:intro}

In \cite{happel:fd-algebra} Happel considered derived categories of finite-dimensional algebras over fields. Interesting special cases of such algebras are path algebras of finite and acyclic quivers. Let us recall that a quiver is simply an oriented graph and that a quiver is acyclic if it admits no non-trivial oriented cycles.  Given such an acyclic quiver $Q$ and a source $q_0\in Q$ (no edge ends at $q_0$) there is the reflected quiver $Q'$ obtained by turning the source into a sink. Bern{\v{s}}te{\u\i}n, Gel$'$fand, and Ponomarev \cite{BGP:Coxeter} showed that the corresponding abelian categories of representations are related by reflection functors. If one works with representations of a \emph{finite, acyclic} quiver over a field, then Happel proved in \cite{happel:fd-algebra} that derived reflection functors yield exact equivalences between the corresponding bounded derived categories of the path algebras. 

The main goal of this paper is to generalize this result in two different directions. First, we show that one obtains similar equivalences if one drops the assumption of working over a field. More precisely, we construct such exact equivalences of derived or homotopy categories of representations over a ring, of representations in quasi-coherent modules on arbitrary schemes, of differential-graded representations over differential-graded algebras, and of spectral representations. In fact, we obtain equivalences of homotopy theories of representations and we show that the existence of such equivalences is a formal consequence of \emph{stability} only. Hence there are many additional variants for representations with values in other \emph{stable homotopy theories} arising in algebra, geometry, and topology (for more details about what we mean by a stable homotopy theory see further below).

Second, we generalize this result in that we obtain such equivalences for a significantly larger class of shapes. Given an \emph{arbitrary small category} $C$ and a finite string  $y_1,y_2,\ldots,y_n$ of objects in $C$, then we can form new categories $C^-$ and $C^+$ by freely adjoining a source or a sink to these objects in $C$. The string of objects is allowed to have some repetition, so that the generic picture to have in mind is as in \autoref{fig:adjoining}. In this situation we show that the categories $C^-$ and $C^+$ have equivalent homotopy theories of representations with values in arbitrary stable homotopy theories, i.e., that $C^-$ and $C^+$ are \emph{strongly stably equivalent} in a sense made precise in \eqref{eq:sse-intro}. 

To illustrate this abstract statement let us turn to some special cases which we explore further in \cite{gst:acyclic-2}. As a first example, if we specialize to a finite, acyclic quiver and consider representations over a field, then we recover the derived equivalences of Happel~\cite{happel:fd-algebra} (actually also a version for unbounded chain complexes). However, even for representations over a field and of quivers, the main result leads to new classes of derived equivalences. 
\begin{enumerate}
\item For example, dropping the finiteness assumption, we see that reflection functors induce derived equivalences between the infinite-dimensional (possibly non-unital) algebras associated to infinite, acyclic quivers. 
\item Alternatively, we can drop the acyclicity assumption. As long as there are sources or sinks in a finite quiver, corresponding reflection functors yield derived equivalences between infinite-dimensional path algebras.
\item Combining these two, we can also drop both the finiteness and the acyclicity assumption. As soon as an arbitrary quiver has sources or sinks, there are associated derived equivalences given by reflection functors. 
\end{enumerate}
Choosing other examples of stable homotopy theories, we see that all these equivalences also have variants if we do not work over a field but with more general abstract representations. As a further specialization we deduce that finite oriented trees can be reoriented arbitrarily without affecting the abstract representation theory, thereby reproducing the main result of \cite{gst:tree}. To mention an additional instance, if one considers representations of a poset in Grothendieck abelian categories, then our main result reestablishes a special case of a result of Ladkani~\cite{ladkani:posets}, but also extends it for example to differential-graded and spectral representations. And there are additional such statements starting with more general small categories instead.

These abstract equivalences are realized by general reflection morphisms between homotopy theories of representations. The arguments involved in their construction are rather formal as they rely only on the existence of a well-behaved calculus of restrictions and (homotopy) Kan extensions of diagrams in stable homotopy theories. Besides being fairly transparent, there are two additional advantages of this method of construction. 
\begin{enumerate}
\item First, this leads to equivalences of \emph{homotopy theories} of abstract representations as opposed to mere equivalences of \emph{homotopy categories} of representations. Since equivalences of homotopy theories \emph{are} exact, the corresponding functors between derived categories or homotopy categories can be turned into exact equivalences with respect to classical triangulations \cite{groth:ptstab}. However, in general, the existence of exact equivalences of triangulated categories of representations does not imply that there are equivalences of homotopy theories in the background. While this is by \cite{dugger-shipley:k-theory} the case for representations over rings, as soon as one passes to differential-graded or spectral representations it is in general a stronger result to have equivalences of homotopy theories.
\item Second, this way the equivalences of homotopy theories of representations with values in stable homotopy theories are seen to be compatible with exact morphisms of stable homotopy theories. In particular, these equivalences hence interact nicely with restriction and (co)induction of scalar morphisms, with localizations and colocalizations, with derived tensor and hom morphisms, and more general exact morphisms. 
\end{enumerate} 

\begin{figure}
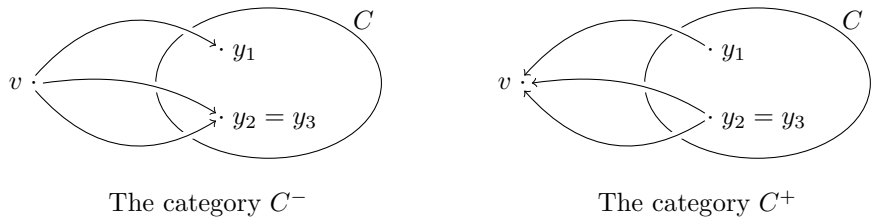

\centering
\TikzFigCplusminus
\caption{Adjoining a source and a sink to $C\in\cCat$.}
\label{fig:adjoining}
\end{figure}

Let us now be more specific about what we mean by abstract (stable) homotopy theories. By now there are various ways of axiomatizing (stable) homotopy theories, including Quillen model categories (\cite{quillen:ha} or \cite{hovey:model}), quasi-categories or $\infty$-categories (\cite{HTT,HA} or \cite{groth:scinfinity}), derivators (\cite{grothendieck:derivators,heller:htpythies,franke:adams}), as well as the more classical triangulated categories. In this paper we use the language of derivators which by definition can be thought of as minimal, purely categorical extensions of the more classical derived or homotopy categories to a framework with a well behaved calculus of homotopy (co)limits and homotopy Kan extensions. In this approach to abstract homotopy theory, homotopy (co)limits and homotopy Kan extensions are defined and characterized by ordinary universal properties, thereby making their calculus accessible by elementary categorical techniques.

The basic idea about derivators is as follows. Given an abelian category $\cA$, the derived category $D(\cA)$ is rather ill-behaved. In particular, the calculus of derived (co)limits and derived Kan extensions is not visible to $D(\cA)$ \emph{alone}. Hence, if one agrees on this calculus to be relevant  (and some evidence for this is for example provided by the observation that classical triangulations simply encode certain shadows of iterated derived cokernel constructions), why not simply encode derived categories of diagram categories $D(\cA^B)$ for various small categories $B$ together with restriction functors between them? Pursuing this more systematically one is lead to consider the \emph{derivator} of $\cA$, a certain $2$-functor
\[
\D_\cA\colon B\mapsto\D_{\cA}(B)=D(\cA^B),
\]
and derived Kan extensions now are merely adjoints to (derived) restriction functors. The values of $\D_{\cA}$ are considered as plain categories, but \emph{exactness properties} of the derivator can be used to construct canonical triangulations and canonical higher triangulations in the sense of Maltsiniotis \cite{maltsiniotis:higher}. In fact, this holds more generally for strong, stable derivators  (see \cite{franke:adams,maltsiniotis:seminar,groth:ptstab} and \cite{gst:Dynkin-A}), such as homotopy derivators of stable model categories or stable $\infty$-categories. Let us recall that a derivator is stable if it admits a zero object and if a square is cartesian if and only if it is cocartesian (see \cite{gps:mayer,gst:basic} for alternative characterizations). While stability is invisible to ordinary category theory, there is a ubiquity of stable derivators arising in algebra, geometry, and topology (\cite[\S5]{gst:basic}).

Now, the connection to abstract representation theory or abstract tilting theory is provided by the following observation. Given a derivator \D and a small category $B$, there is the derivator $\D^B$ of coherent diagrams of shape $B$ in $\D$. This exponentiation is compatible with the formation of exponentials at the level of abelian categories, (nice) model categories, and $\infty$-categories. For example, given a Grothendieck abelian category $\cA$ and a small category $B$ there is an equivalence of derivators
\[
\D_{\cA}^B\simeq\D_{\cA^B}.
\]
Specializing further this shows that the passage to category algebras (like path algebras, incidence algebras, and group algebras) can be modeled by this shifting operation at the level of derivators. 

To state the main result of this paper more precisely, let $\cDER_{\mathrm{St,ex}}$ be the $2$-category of stable derivators, exact morphisms, and all natural transformations. For every small category $B$, exponentiation by $B$ defines a $2$-functor
\[
(-)^B\colon\cDER_{\mathrm{St,ex}}\to\cDER\colon\D\mapsto\D^B,
\]
where $\cDER$ is the $2$-category of derivators. Denoting again by $C$ an arbitrary small category and by $C^-,C^+$ the categories obtained from $C$ by freely attaching a source or a sink to a prescribed string of objects (see again \autoref{fig:adjoining}), we show that these two categories are  \emph{strongly stably equivalent} in the sense of \cite{gst:basic}. Thus, we show that there is a pseudo-natural equivalence of $2$-functors
\begin{equation}\label{eq:sse-intro}
\Phi\colon(-)^{C^-}\simeq(-)^{C^+}\colon\cDER_{\mathrm{St,ex}}\to\cDER,
\end{equation}
and in this precise sense $C^-,C^+$ have equivalent abstract representation theories. 

In the sequel \cite{gst:acyclic-2} we study these general reflection morphisms further. We will show that unrelated reflections commute, leading to abstract Coxeter morphisms for finite, acyclic quivers. Moreover, the reflections are shown to be realized by explicitly constructed invertible spectral bimodules, and this yields non-trivial elements in spectral Picard groupoids. We will also obtain a spectral Serre duality result for acyclic quivers and, more generally, strongly homotopy finite categories.

While here and in the sequel we state and prove the above results using the language of derivators, it is completely formal to also deduce implications for model categories and $\infty$-categories of abstract representations. For concreteness, given a stable, combinatorial model category~$\cM$, the existence of the strong stable equivalence \eqref{eq:sse-intro} implies by \cite{renaudin} that the model categories $\cM^{C^-}$ and $\cM^{C^+}$ are connected by a zigzag of Quillen equivalences. And, similarly, there is a variant for stable, presentable $\infty$-categories of representations. 

This paper belong to a series of papers on abstract representation theory and abstract tilting theory, and they can be considered as sequels to \cite{gst:basic},\cite{gst:tree}, and \cite{gst:Dynkin-A}. This project relies both on a basic formal understanding of stability \cite{groth:ptstab,gps:mayer} as well as on a basic understanding of the interaction of monoidality and stability \cite{gps:additivity,ps:linearity}. We intend to come back to further applications to abstract representation theory elsewhere.

The content of the sections is as follows. In \S\S\ref{sec:review}-\ref{sec:review-htpy} we recall some basics concerning derivators. In \S\ref{sec:guide} we outline the strategy of the construction of the general reflection morphisms leading to the desired strong stable equivalence. In \S\S\ref{sec:glue}-\ref{sec:glue-epi} we introduce free oriented gluing constructions of small categories and study their compatibility with Kan extensions and homotopical epimorphisms. This allows us in \S\ref{sec:reflection-sep} to construct reflection equivalences in the special case of separated sources and sinks. In \S\ref{sec:detect} we establish two simple detection criteria for homotopical epimorphisms, which we use in \S\ref{sec:reflection} to conclude the construction of reflection equivalences in the general case. In \S\ref{sec:applications} we deduce some consequences of our abstract tilting result. Finally, in \S\ref{sec:amalgamation} we collect some results concerning the combinatorics of amalgamations of small categories which are useful in \S\ref{sec:reflection}.

\section{Review of stable derivators and strong stable equivalences}
\label{sec:review}

In this section we include a short review of stable derivators. For more details we refer the reader to \cite{groth:ptstab,gps:mayer}. The key idea behind a derivator is that they enhance the more classical derived categories of abelian categories and homotopy categories of model categories by also keeping track of homotopy categories of diagram categories together with the calculus of homotopy Kan extensions. Like stable model categories and stable $\infty$-categories, stable derivators provide an enhancement of triangulated categories.

To make this precise, let $\cCat$ be the $2$-category of small categories and $\cCAT$ the $2$-category of not necessarily small categories. We refer the reader to \cite{borceux:1} for basic $2$-categorical terminology.

\begin{defn}
A \textbf{prederivator} is a $2$-functor $\D\colon\cCat\op\to\cCAT$. \textbf{Morphisms} of prederivators are pseudo-natural transformations and \textbf{transformations} between morphisms of prederivators are modifications, yielding the $2$-category~$\cPDER$ of prederivators.
\end{defn}

Given a prederivator~\D we refer to objects in $\D(A)$ as \textbf{coherent diagrams (of shape~$A$)}. For every functor $u\colon A\to B$ there is a \textbf{restriction functor} $u^\ast\colon\D(B)\to\D(A)$. In the special case that $A=\bbone$ is the terminal category and $u=b\colon\bbone\to B$ hence classifies an object~$b\in B$, we refer to $b^\ast\colon\D(B)\to\D(\bbone)$ as an \textbf{evaluation functor}. Evaluating a morphism $f\colon X\to Y$ in $\D(B)$ we obtain induced morphisms $f_b\colon X_b\to Y_b,b\in B,$ in the \textbf{underlying category} $\D(\bbone)$.

If a restriction functor $u^\ast\colon\D(B)\to\D(A)$ admits a left adjoint, then we refer to it as a \textbf{left Kan extension functor} and denote it by  $u_!\colon\D(A)\to\D(B)$. In the special case that $u=\pi_A\colon A\to\bbone$ collapses $A$ to a point, such a left adjoint is also denoted by $(\pi_A)_!=\colim_A\colon\D(A)\to\D(\bbone)$ and referred to as a \textbf{colimit functor}. Dually, we speak of \textbf{right Kan extension functors} $u_\ast\colon\D(A)\to\D(B)$ and \textbf{limit functors} $(\pi_A)_\ast=\lim_A\colon\D(A)\to\D(\bbone)$.

For \emph{derivators} we ask for the existence of such Kan extension functors and that they can be calculated \emph{pointwise} (see \cite[X.3.1]{maclane} for the classical context of ordinary categories). To express this purely $2$-categorically, we consider the \textbf{slice squares}
\begin{equation}
\vcenter{
\xymatrix{
(u/b)\ar[r]^-p\ar[d]_-{\pi_{(u/b)}}\drtwocell\omit{}&A\ar[d]^-u&&(b/u)\ar[r]^-q\ar[d]_-{\pi_{(b/u)}}&A\ar[d]^-u\\
\bbone\ar[r]_-b&B,&&\bbone\ar[r]_-b&B,\ultwocell\omit{}
}
}
\label{eq:Der4}
\end{equation}
coming with transformations $u\circ p\to b\circ\pi$ and $b\circ\pi\to u\circ q$, respectively. Here, objects in the \textbf{slice category} $(u/b)$ are pairs $(a,f)$ consisting of an object $a\in A$ and a morphism $f\colon u(a)\to b$ in $B$. A morphism $(a,f)\to(a',f')$ is a map $a\to a'$ in $A$ making the  obvious triangles commute. The functor $p\colon (u/b)\to A$ is the obvious projection and the component of the transformation $u\circ p\to b\circ\pi$ at $(a,f)$ is~$f$. The square on the right in \eqref{eq:Der4} is defined dually.

\begin{defn}\label{defn:derivator}
  A prederivator $\D\colon\cCat\op\to\cCAT$ is a \textbf{derivator}\footnote{We emphasize that $\cCat\op$ is obtained from $\cCat$ by changing the orientation of functors but not of natural transformations. Thus, following Heller~\cite{heller:htpythies} and Franke~\cite{franke:adams}, our convention for derivators is based on \emph{diagrams}. There is an equivalent approach using \emph{presheaves}, i.e., contravariant functors; see for example \cite{grothendieck:derivators,cisinski:direct}.} if the following properties are satisfied.
  \begin{itemize}[leftmargin=4em]
  \item[(Der1)] $\D\colon \cCat\op\to\cCAT$ takes coproducts to products, i.e., the canonical map $\D(\coprod A_i)\to\prod\D(A_i)$ is an equivalence.  In particular, $\D(\emptyset)$ is equivalent to the terminal category.
  \item[(Der2)] For any $A\in\cCat$, a morphism $f\colon X\to Y$ in $\D(A)$ is an isomorphism if and only if the morphisms $f_a\colon X_a\to Y_a, a\in A,$ are isomorphisms in $\D(\bbone).$
  \item[(Der3)] Each functor $u^*\colon \D(B) \to\D(A)$ has both a left adjoint $u_!$ and a right adjoint $u_*$.
  \item[(Der4)] For any functor $u\colon A\to B$ and any $b\in B$ the canonical transformations
\begin{gather}
  \pi_! p^* \stackrel{\eta}{\to} \pi_! p^* u^* u_! \to \pi_! \pi^* b^* u_! \stackrel{\epsilon}{\to} b^* u_!  \mathrlap{\qquad\text{and}}\label{eq:Der4!}\\
  b^* u_* \stackrel{\eta}{\to} \pi_* \pi^* b^* u_* \to \pi_* q^* u^* u_* \stackrel{\epsilon}{\to} \pi_* q^*\label{eq:Der4*}
\end{gather}
associated to the slice squares \eqref{eq:Der4} are isomorphisms.
  \end{itemize}
\end{defn}

Axiom (Der4) hence says that for $u\colon A\to B,b\in B,$ and $X\in\D(A)$ certain canonical maps
\[
\colim_{(u/b)}p^\ast X\to u_!(X)_b\qquad\text{and}\qquad u_\ast(X)_b\to\lim_{(b/u)}q^\ast X
\]
are isomorphisms. We say a bit more about the formalism related to (Der4) in \S\ref{sec:review-htpy}. 

\textbf{Morphisms} and \textbf{transformations} of derivators are morphisms and transformations of underlying prederivators, respectively, yielding the sub-$2$-category $\cDER\subseteq\cPDER$ of derivators. Given a (pre)derivator, we often write $X\in\D$ if there is a small category~$A$ such that $X\in\D(A)$.

\begin{egs}\label{egs:derivator}
~
\begin{enumerate}
\item Let~\cC be an ordinary category. The $2$-functor
\[
y_\cC\colon\cCat\op\to\cCAT\colon A\mapsto \cC^A
\]
is a derivator if and only if $\cC$ is complete and cocomplete. Kan extension functors in such a \textbf{represented derivator} are \emph{ordinary Kan extensions} from classical category theory. The underlying category of $y_\cC$ is isomorphic to~$\cC$.
\item Let~$\cA$ be a Grothendieck abelian category and let $\nCh(\cA)$ be the category of unbounded chain complexes in~$\cA$. For every $A\in\cCat$ we denote by $W^A$ the class of levelwise quasi-isomorphisms in $\nCh(\cA)^A$. The $2$-functor
\[
\D_\cA\colon\cCat\op\to\cCAT\colon A\mapsto \nCh(\cA)^A[(W^A)^{-1}]
\]
is a derivator. Kan extension functors in $\D_{\cA}$ are \emph{derived Kan extensions} in the sense of homological algebra. The underlying category of $\D_\cA$ is isomorphic to the \emph{derived category} $D(\cA)$ of~$\cA$. As interesting examples we obtain derivators associated to fields, rings, and schemes.
\item Let $\cM$ be a Quillen model category \cite{quillen:ha,hovey:model} with weak equivalences $W$. Denoting by $W^A$ the levelwise weak equivalences in $\cM^A$, there is an associated \textbf{homotopy derivator}
\[
\ho_\cM\colon\cCat\op\to\cCAT\colon A\mapsto \cM^A[(W^A)^{-1}];
\]
see \cite{cisinski:direct} for the general case and \cite[Prop.~1.30]{groth:ptstab} for an easy proof in the case of combinatorial model categories. Kan extension functors in $\ho_\cM$ are \emph{homotopy Kan extensions}. The underlying category of $\ho_\cM$ is isomorphic to the \emph{homotopy category} $\Ho(\cM)$. Similarly, there are homotopy derivators associated to complete and cocomplete $\infty$-categories or quasi-categories (\cite{joyal:I-II,joyal:barca,HTT,groth:scinfinity}); see \cite{gps:mayer} for a sketch proof. These two classes give rise to a plethora of additional examples of derivators.
\end{enumerate}
\end{egs}

Thus, derivators encode key formal properties of the calculus of Kan extensions, derived Kan extensions, and homotopy Kan extensions, as it is available in typical situations arising in nature. It turns out that many constructions are combinations of such Kan extensions, including the general reflection functors we construct in this paper; see \S\ref{sec:guide}, \S\ref{sec:reflection-sep}, and \S\ref{sec:reflection}.

Let $[1]$ be the poset $(0<1)$ considered as a category and let $\square=[1]\times[1]$ be the commutative square. We denote by $i_\ulcorner\colon\ulcorner\to\square$, $i_\lrcorner\colon\lrcorner\to\square$ the full subcategories obtained by removing the final and initial object, respectively. A square $X\in\D(\square)$ is \textbf{cartesian} if it lies in the essential image of $(i_\lrcorner)_\ast\colon\D(\lrcorner)\to\D(\square)$. Dually, we define \textbf{cocartesian squares}.

\begin{defn}
A derivator is \textbf{pointed} if the underlying category has a zero object. A pointed derivator is \textbf{stable} if a square is cartesian if and only if it is cocartesian.
\end{defn}

\begin{egs}
~
\begin{enumerate}
\item The derivator of a Grothendieck abelian category is stable. In particular, fields, rings, and schemes have associated stable derivators.
\item Homotopy derivators of stable model categories and stable $\infty$-categories are stable.
\item The derivator of differential graded modules over a differential graded algebra is stable.
\item The derivator of module spectra over a symmetric ring spectrum is stable. In particular, the derivator of spectra itself is stable.
\end{enumerate}
\end{egs}

We refer the reader to \cite[Examples~5.5]{gst:basic} for many additional examples of stable derivators arising in algebra, geometry, and topology. It can be shown that the values of a (strong) stable derivators are canonically triangulated categories (\cite{franke:adams,maltsiniotis:seminar} or \cite[Thm.~4.16 and Cor.~4.19]{groth:ptstab}) and even higher triangulated categories (\cite[Thm.~13.6, Cor.~13.11, and Rmk.~13.12]{gst:Dynkin-A}) in the sense of Maltsiniotis~\cite{maltsiniotis:higher}.

In a pointed derivator~\D one can define \textbf{suspensions, loops, cofibers, and fibers} (see~\cite[\S3]{groth:ptstab}), yielding adjunctions
\[
(\Sigma,\Omega)\colon\D(\bbone)\rightleftarrows\D(\bbone)\qquad\text{and}\qquad (\cof,\fib)\colon\D([1])\rightleftarrows\D([1]).
\]
We recall from \cite[\S8]{gst:basic} some basic notation and terminology related to $n$-cubes $[1]^n=[1]\times\ldots\times[1]$. The poset $[1]^n$ is isomorphic to the power set of $\{1,\ldots,n\}$, and this isomorphism is used implicitly in what follows. We denote by $i_{\geq k}\colon[1]^n_{\geq k}\to[1]^n, 0\leq k\leq n,$ the full subcategory spanned by all subsets of cardinality at least $k$. This notation has obvious variants, for example, the full subcategory $i_{=n-1}\colon[1]^n_{=n-1}\to[1]^n$ is the discrete category $n\cdot\bbone=\bbone\sqcup\ldots\sqcup\bbone$ on~$n$ objects.

\begin{defn}
Let \D be a derivator. An $n$-cube $X\in\D([1]^n)$ is \textbf{strongly cartesian} if it lies in the essential image of $(i_{\geq n-1})_\ast\colon\D([1]_{\geq n-1}^n)\to\D([1]^n)$. An $n$-cube $X\in\D([1]^n)$ is \textbf{cartesian} if it lies in the essential image of $(i_{\geq 1})_\ast$.
\end{defn}

Dually, one defines \textbf{(strongly) cocartesian $n$-cubes}. Following ideas of Goodwillie \cite{goodwillie:II}, one shows the following.

\begin{thm}[{\cite[Thm.~8.4]{gst:basic},\cite[Cor.~8.12]{gst:basic}}]
\label{thm:strongly-cocart}
An $n$-cube, $n \ge 2$, in a derivator is strongly cartesian if and only if all subcubes are cartesian if and only if all subsquares are cartesian.
\end{thm}

Stable derivators admit the following different characterizations.

\begin{thm}[{\cite[Thm.~7.1]{gps:mayer},\cite[Cor.~8.13]{gst:basic}}]
The following are equivalent for a pointed derivator \D.
\begin{enumerate}
\item The adjunction $(\Sigma,\Omega)\colon\D(\bbone)\to\D(\bbone)$ is an equivalence.
\item The adjunction $(\cof,\fib)\colon\D([1])\to\D([1])$ is an equivalence.
\item The derivator \D is stable.
\item An $n$-cube in \D, $n\geq 2,$ is strongly cartesian if and only if it is strongly cocartesian.
\end{enumerate}
\end{thm}

An $n$-cube which is simultaneously strongly cartesian and strongly cocartesian is \textbf{strongly bicartesian}. In the case of $n=2$ this reduces to the classical notion of a \textbf{bicartesian square}. Strongly bicartesian $n$-cubes in stable derivators satisfy the 2-out-of-3 property with respect to composition and cancellation (see \cite[\S8]{gst:basic} for the case of $n$-cubes).

The natural domains for Kan extensions with parameters are given by shifted derivators in the sense of the following proposition.  This exponential construction is central to abstract representation theory.

\begin{prop}[{\cite[Thm.~1.25 and Prop.~4.3]{groth:ptstab}}]\label{prop:shifting}
Let \D be a derivator and let $B\in\cCat$. The $2$-functor
\[
\D^B\colon\cCat\op\to\cCAT\colon A\mapsto\D(B\times A)
\]
is again a derivator, the \textbf{derivator of coherent diagrams of shape~$B$}, which is pointed or stable as soon as \D is.
\end{prop}

This shifting operation also applies to morphisms and natural transformations in either variable, thereby defining a $2$-functor
\[
\cCat\op\times\cDER\to\cDER\colon (A,\D)\mapsto \D^A.
\]
In abstract representation theory we are interested in suitable restrictions of related $2$-functors. To begin with, as special cases of morphisms of derivators preserving certain (co)limits (\cite[\S2.2]{groth:ptstab}) there are the following definitions.

\begin{defn}
\begin{enumerate}
\item A morphism of derivators is \textbf{right exact} if it preserves initial objects and cocartesian squares.
\item A morphism of derivators is \textbf{left exact} if it preserves terminal objects and cartesian squares.
\item A morphism of derivators is \textbf{exact} if it is right exact and left exact.
\end{enumerate}
\end{defn}

A morphism between stable derivators is right exact if and only if it is left exact if and only if it is exact. In particular, adjunctions and equivalences between stable derivators give rise to exact morphisms. (Adjunctions and equivalences of derivators are defined internally to the 2-category $\cDER$; see \cite[\S2]{groth:ptstab} for details including explicit reformulations.)

Identity morphisms are exact and exact morphisms are closed under compositions, and there is thus the $2$-category $\cDER_{\mathrm{St},\mathrm{ex}}\subseteq\cDER$ of stable derivators, exact morphisms, and arbitrary natural transformations. 
Hence, for every $A\in\cCat$ we obtain an induced $2$-functor $(-)^A\colon\cDER\to\cDER$ which can be restricted to
\[
(-)^A\colon\cDER_{\mathrm{St},\mathrm{ex}}\to\cDER.
\]

\begin{defn}[{\cite[Def.~5.1]{gst:basic}}]\label{defn:sse}
Two small categories $A$ and $A'$ are \textbf{strongly stably equivalent}, in notation $A\sse A'$, if there is a pseudo-natural equivalence between the $2$-functors
\[
\Phi\colon(-)^A\simeq (-)^{A'}\colon\cDER_{\mathrm{St},\mathrm{ex}}\to\cDER.
\]
Such a pseudo-natural equivalence is called a \textbf{strong stable equivalence}.
\end{defn}

This definition makes precise the idea that the categories $A$ and $A'$ have the same representation theories in arbitrary stable derivators. More formally, a strong stable equivalence $\Phi\colon A\sse  A'$ consists of 
\begin{enumerate}
\item an equivalence of derivators $\Phi_\D\colon\D^A\simeq\D^{A'}$ for \emph{every} stable derivator \D and
\item associated to every exact morphism of stable derivators $F\colon\D\to\E$ a natural isomorphism 
$\gamma_F\colon F\circ \Phi_\D\to \Phi_E\circ F,$
\[
\xymatrix{
\D^A\ar[r]^-{\Phi_\D}_-\simeq\ar[d]_-F\drtwocell\omit{\cong}&\D^{A'}\ar[d]^-F\\
\E^A\ar[r]^-\simeq_-{\Phi_{\E}}&\E^{A'},
}
\]
\end{enumerate}
satisfying the usual coherence properties of a pseudo-natural transformation.

The motivation for this definition is the following example of the shifting operation; see~\cite[\S5]{gst:basic}. 

\begin{eg}\label{eg:grothendieck-shift}
Let \cA be a Grothendieck abelian category and let $B\in\cCat$. There is an equivalence of derivators
\[
\D_\cA^B\simeq\D_{\cA^B}
\]
\end{eg}

In particular, if $B,B'$ are strongly stably equivalent,  then there is a chain of equivalences of derivators
\[
\D_{\cA^B}\simeq\D_\cA^B\simeq\D_\cA^{B'}\simeq\D_{\cA^{B'}}.
\]
Specializing to the Grothendieck abelian category of modules over a ring $R$ and assuming that $B=Q,B'=Q'$ are quivers with finitely many vertices, we obtain equivalences
\[
\D_{RQ}\simeq\D_{RQ'}
\]
of the derivators of the respective path algebras. Since equivalences of derivators are exact, this yields exact equivalences of derived categories
\[
D(RQ)\stackrel{\Delta}{\simeq}D(RQ'),
\]
showing that strongly stably equivalent quivers are derived equivalent over arbitrary rings. A priori, however, it is a much stronger result if we know that two quivers are strongly stably equivalent, since this means that the quivers have the same homotopy theories of abstract representations. We expand a bit on this in \S\ref{sec:applications}.

\section{Review of homotopy exact squares}
\label{sec:review-htpy}

In this section we review some results concerning the calculus of homotopy exact squares. This calculus is arguably the most important technical tool in the theory of derivators and it allows us to establish many useful manipulation rules for Kan extensions in derivators. For more details see for example \cite{ayoub:I-II,maltsiniotis:htpy-exact} and \cite{groth:ptstab,gps:mayer,gst:basic}.

To begin with let us consider a natural transformation $\alpha\colon up\to vq$ living in a square of small categories
\begin{equation}
  \vcenter{\xymatrix{
      D\ar[r]^p\ar[d]_q \drtwocell\omit{\alpha} &
      A\ar[d]^u\\
      B\ar[r]_v &
      C.
    }}
\label{eq:htpyexact}
\end{equation}
The square~\eqref{eq:htpyexact} is \textbf{homotopy exact} if one of the canonical mates
\begin{gather}
  q_! p^* \to q_! p^* u^* u_! \xto{\alpha^*} q_! q^* v^* v_! \to v^* u_!  \mathrlap{\qquad\text{and}}\label{eq:hoexmate1'}\\
  u^* v_* \to p_* p^* u^* v_* \xto{\alpha^*} p_* q^* v^* v_* \to p_* q^*\label{eq:hoexmate2'}
\end{gather}
is a natural isomorphism. It turns out that \eqref{eq:hoexmate1'} is an isomorphism if and only if \eqref{eq:hoexmate2'} is an isomorphism.

Using this terminology, note that axiom (Der4) from~\autoref{defn:derivator} precisely says that slice squares \eqref{eq:Der4} are homotopy exact.
Although it may seem from the definition that the notion of homotopy exactness depends on the theory of derivators, this is only seemingly the case. Homotopy exact squares can be characterized by means of the classical homotopy theory of (diagrams of) topological spaces. In fact, a square is homotopy exact if and only if the canonical mate is an isomorphism for the homotopy derivator of topological spaces, and this even admits a combinatorial reformulation; see \cite[\S3]{gps:mayer}.

For later reference, we collect a few additional examples of homotopy exact squares and make explicit what they tell us about Kan extensions.

\begin{egs}\label{egs:htpy-exact} 
~
\begin{enumerate}
\item \emph{Kan extensions along fully faithful functors are fully faithful.} If $u\colon A\to B$ is fully faithful, then the square
\begin{equation}
\vcenter{
\xymatrix{
A\ar[r]^-\id\ar[d]_-\id&A\ar[d]^-u\\
A\ar[r]_-u&B
}}
\end{equation}
is homotopy exact, which is to say that the unit $\eta\colon\id\to u^\ast u_!$ and the counit $\epsilon\colon u^\ast u_\ast\to\id$ are isomorphisms  (\cite[Proposition~1.20]{groth:ptstab}). Thus, $u_!,u_\ast\colon\D(A)\to\D(B)$ are fully faithful.
\item \emph{Kan extensions and restrictions in unrelated variables commute.} Given functors $u\colon A\to B$ and $v\colon C\to D$ then the commutative square
\begin{equation}
\vcenter{
\xymatrix{
A\times C\ar[r]^-{u\times\id}\ar[d]_-{\id\times v}&B\times C\ar[d]^-{\id\times v}\\
A\times D\ar[r]_-{u\times\id}&B\times D
}}
\end{equation}
is homotopy exact (\cite[Proposition~2.5]{groth:ptstab}). Thus, the canonical mate transformation
$(\id\times v)_!(u\times \id)^\ast\to(u\times \id)^\ast(1\times v)_!$ is an isomorphism and similarly for right Kan extensions.
\item \emph{Right adjoint functors are \textbf{homotopy final}.} If $u\colon A\to B$ is a right adjoint, then the square
\[
\xymatrix{
A\ar[r]^-u\ar[d]_-{\pi_A}\drtwocell\omit{\id}&B\ar[d]^-{\pi_B}\\
\bbone\ar[r]_-\id&\bbone
}
\]
is homotopy exact, i.e., the canonical mate $\mathrm{colim}_Au^\ast\to\mathrm{colim}_B$ is an isomorphism (\cite[Proposition~1.18]{groth:ptstab}). In particular, if $b\in B$ is a terminal object, then there is a canonical isomorphism $b^\ast\cong\colim_B$.
\item \emph{Homotopy exact squares are compatible with pasting.} Since the passage to the canonical mates \eqref{eq:hoexmate1'} and \eqref{eq:hoexmate2'} is functorial with respect to horizontal and vertical pasting, such pastings of homotopy exact squares are again homotopy exact (\cite[Lemma~1.14]{groth:ptstab}).
\end{enumerate}
\end{egs}

It follows from \autoref{egs:htpy-exact}(ii) that there are Kan extension morphisms of derivators. In fact, given a derivator~\D and a functor $u\colon A\to B$, there are adjunctions of derivators given by \emph{parametrized Kan extensions},
\begin{equation}
(u_!,u^\ast)\colon\D^A\rightleftarrows\D^B\qquad\text{and}\qquad
(u^\ast,u_\ast)\colon\D^B\rightleftarrows\D^A.\label{eq:Kan-adjunction}
\end{equation}
If $u$ is fully faithful, then $u_!,u_\ast\colon\D^A\to\D^B$ are fully faithful morphisms of derivators and as such they induce equivalences onto their respective essential images. In particular, these essential images are again derivators (\cite[\S3]{gst:basic}).

The point of the following lemma is that to check whether an object $X \in \D^B$ is in the essential image of $u_!$, it suffices to test objects in $B-u(A)$ only.

\begin{lem}[{\cite[Lemma~1.21]{groth:ptstab}}]\label{lem:ff-essim}
Let \D be a derivator and $u\colon A\to B$ a fully faithful functor between small categories. A coherent diagram $X\in\D^B$ lies in the essential image of $u_!\colon\D^A\to\D^B$ if and only if $\epsilon_b\colon u_! u^\ast(X)_b\to X_b$ is an isomorphism for all $b\in B-u(A)$.
\end{lem}

This lemma takes a particular simple form for certain Kan extensions in pointed derivators. Recall that a fully faithful functor $u\colon A\to B$ is a \textbf{sieve} if for every morphism $b\to u(a')$ in $B$ with target in the image of $u$ it follows that $b=u(a)$ for some $a\in A$. There is the dual notion of a \textbf{cosieve}.

\begin{prop}[{\cite[Prop.~3.6]{groth:ptstab}}]\label{prop:ext-by-zero}
Let \D be a pointed derivator and $u\colon A\to B$ a sieve. The morphism $u_\ast\colon\D^A\to\D^B$ is fully faithful and $X\in\D^B$ lies in the essential image of $u_\ast$ if and only if $u_b\cong 0$ for all $b\in B-u(A)$.
\end{prop}

We refer to right Kan extension morphisms along sieves as \textbf{right extensions by zero}. Dually, left Kan extensions along cosieves are \textbf{left extensions by zero}.

\begin{rmk}\label{rmk:ext-by-zero-unpointed}
If \D is not pointed, then \autoref{prop:ext-by-zero} yields right extensions by terminal objects and left extensions by initial objects in the obvious sense (\cite[Prop.~1.23]{groth:ptstab}).
\end{rmk}

By \autoref{egs:htpy-exact} there is an easy criterion guaranteeing that Kan extensions are fully faithful. The case of restrictions is more subtle. Inspired by the notion of a \emph{homological epimorphism} introduced by Geigle and Lenzing \cite[\S4]{GL91} there is the following definition (see \cite[\S6]{gst:tree} and, in particular, Remark~6.4 in~\emph{loc.\!~cit.}).

\begin{defn}\label{defn:htpy-epi}
A functor $u\colon A\to B$ is a \textbf{homotopical epimorphism} if for every derivator \D the restriction functor $u^\ast\colon\D(B)\to\D(A)$ is fully faithful.
\end{defn}

If $u$ is a homotopical epimorphism then $u^\ast\colon\D^B\to\D^A$ induces an equivalence onto its essential image. Basic examples and closure properties are collected in~\cite[\S\S6-7]{gst:tree}. Here it suffices to note that $u\colon A\to B$ is a homotopical epimorphism if and only if the square
\[
\xymatrix{
A\ar[r]^-u\ar[d]_-u&B\ar[d]^-\id\\
B\ar[r]_-\id&B
}
\]
is homotopy exact. We will get back to this in \S\ref{sec:glue-epi} and \S\ref{sec:detect}.

\section{A pictorial guide to general reflection morphisms}
\label{sec:guide}

In this section we describe the strategy behind the construction of the general reflection morphisms as carried out in~\S\ref{sec:reflection-sep} and~\S\ref{sec:reflection}. While some main steps follow the lines of the construction in~\cite[\S5]{gst:tree}, they have to be adapted significantly to cover the more general class of examples we consider in this paper.

Let $C\in\cCat$ and let $C^-$ be the category obtained from~$C$ by freely attaching a new object~$v$ together with $n$ morphisms from~$v$ to objects in~$C$; see~\autoref{fig:adjoining}. Performing a similar construction but this time adding morphisms pointing to~$v$ we obtain the category~$C^+$. Thus, the categories $C^-,C^+$ are obtained from~$C$ by attaching a source and sink, respectively, to the same objects in~$C$, and the picture to have in mind is as in \autoref{fig:adjoining}.

One of our main goals is to show that for \emph{every} small category~$C$ the categories $C^-$ and $C^+$ are strongly stably equivalent, i.e., that for every stable derivator~\D there is an equivalence $\D^{C^-}\simeq\D^{C^+}$ which is pseudo-natural with respect to exact morphisms (\autoref{defn:sse}). Mimicking the classical construction of reflection functors~\cite{BGP:Coxeter}, we obtain reflection morphisms $s^-\colon\D^{C^-}\to\D^{C^+}$ and $s^+\colon\D^{C^+}\to\D^{C^-}$, which we show to define such a strong stable equivalence. As a first approximation, the rough strategy behind the construction of $s^-$ and $s^+$ is as follows (see~\autoref{fig:rough-strategy}).

\begin{figure}
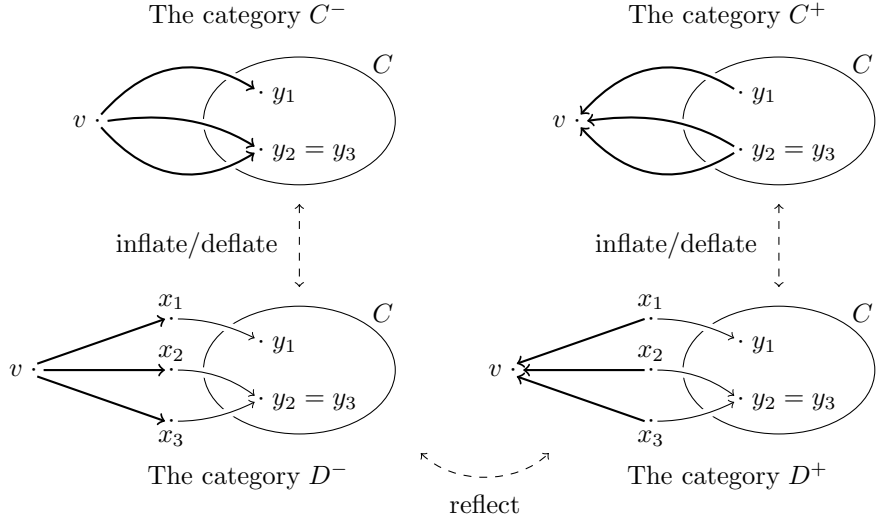

\centering
\TikzFigGlobalStrategy
\caption{Rough strategy behind construction of reflection functors.}
\label{fig:rough-strategy}
\end{figure}

\begin{enumerate} \label{page:rough-strategy}
\item Take a representation of~$C^-$ and separate the morphisms adjacent to the new source by inserting new morphisms. One point being that the shape $D^-$ of this new representation contains an isomorphic copy of the source of valence~$n$. Moreover, we know precisely which representations of~$D^-$ arise this way, namely those which populate the new morphisms by isomorphisms.
\item Show that the reflection morphisms for sources and sinks of valence~$n$ as constructed in \cite{gst:tree} yield similar reflection morphisms in this more general situation. 
Thus, if $D^+$ is the category obtained from~$D^-$ by turning the source into a sink, then we construct certain morphisms of derivators $\D^{D^-}\to\D^{D^+}$, which restrict to suitable equivalences. We expand on this step further below. 
\item Finally, it suffices to show that we can restrict representations of $D^+$ to representations of~$C^+$, thereby possibly identifying some of the sources of morphisms adjacent to the new sink. If we only consider representations of~$D^+$ satisfying certain exactness properties, then this step induces an equivalence of derivators. Note that the situation in this step differs from the one in step~(i) since here the arrows point in different directions. It turns out that this step is more involved than the similarly looking first step.
\end{enumerate}

The first and third steps are taken care of in \S\S\ref{sec:detect}-\ref{sec:reflection}, while the second step is addressed in \S\S\ref{sec:glue}-\ref{sec:reflection-sep}. We now expand on this second step, which performs the actual reflection and is motivated by the classical reflection functors from representation theory (see~\cite{gabriel:unzerlegbare,BGP:Coxeter,happel:dynkin} and also the discussion in~\cite[\S5]{gst:tree}). Let $v\to x_i,i=1,\ldots,n,$ be the morphisms in~$D^-$ which are adjacent to the source~$v$. Given an abstract representation~$X\in\D^{D^-}$, we consider the morphism $X_v\to\bigoplus_{i=1}^n X_{x_i}$ induced by the structure maps and pass to its cofiber. However, in order to obtain a representation of the reflected category~$D^+$, we have to take some care in setting up coherent biproduct diagrams appropriately.

To begin with, we recall from \cite[\S4\text{ and }\S7]{gst:tree} that finite biproduct objects in stable derivators can be modeled by $n$-cubes of length two. In more detail, let us consider the following diagram in~$\cCat$,
\begin{equation}
\vcenter{
\xymatrix{
n\cdot\bbone=[1]^n_{=n-1}\ar[r]^-{i_1}&[1]^n_{\geq n-1}\ar[r]^-{i_2}&[1]^n\ar[r]^-{i_3}& I\ar[r]^-{i_4}& [2]^n\ar[r]^-q&R^n,
}
}
\label{eq:biproducts}
\end{equation}
in which we ignore the functor $q\colon[2]^n\to R^n$ for now. The functors $i_1,i_2$ are the obvious fully faithful inclusion functors, and the composition $i_4i_3\colon[1]^n\to[2]^n$ is the inclusion as the $n$-cube $[1,2]^n,$ i.e., the convex hull of $(1,\ldots,1),(2,\ldots,2)\in[2]^n$. Let $I\subseteq[2]^n$ be the full subcategory spanned by $[1,2]^n$ and the corners
\[
(0,2,\ldots,2),\quad(2,0,2,\ldots,2),\quad\ldots,\quad(2,\ldots,2,0),
\]
and let $i_3\colon[1]^n\to I$ and $i_4\colon I\to [2]^n$ be the corresponding factorization of $i_4i_3$. The associated Kan extension morphisms
\begin{equation}
\D^{n\cdot\bbone}=\D^{[1]^n_{=n-1}}\stackrel{(i_1)_\ast}{\to}\D^{[1]^n_{\geq n-1}}\stackrel{(i_2)_\ast}{\to}
\D^{[1]^n}\stackrel{(i_3)_!}{\to}\D^I\stackrel{(i_4)_\ast}{\to}\D^{[2]^n}
\label{eq:biproducts-II}
\end{equation}
are fully faithful and the essential image is in the stable case as follows. For every stable derivator \D we denote by $\D^{[2]^n,\mathrm{ex}}\subseteq\D^{[2]^n}$ the full subderivator spanned by the diagrams such that
\begin{enumerate}
\item all subcubes are strongly bicartesian,
\item the values at all corners are trivial, and
\item the maps $(i_1,\ldots,i_{k-1},0,i_{k+1},\ldots,i_n)\to (i_1,\ldots,i_{k-1},2,i_{k+1},\ldots,i_n)$ are sent to isomorphisms for all $i_1,\ldots,i_{k-1},i_{k+1},\ldots,i_n$ and $k$.
\end{enumerate}
We note that (iii) is a consequence of (i) and (ii) together with isomorphisms being stable under base change (\cite[Prop.~3.12]{groth:ptstab}), but it is included here for emphasis. As discussed in \cite[\S4]{gst:tree} such diagrams model coherent finite biproduct diagrams together with all the inclusion and projection morphisms.
The following result justifies that we refer to $\D^{[2]^n,\exx}$ as a derivator.

\begin{prop}[{\cite[Proposition~4.9]{gst:tree}}]\label{prop:biproducts}
Let \D be a stable derivator and $n\geq 2$. The morphisms \eqref{eq:biproducts-II} are fully faithful and induce an equivalence $\D^{n\cdot\bbone}\simeq\D^{[2]^n,\exx}$, which is pseudo-natural with respect to exact morphisms. The derivator $\D^{[2]^n,\exx}$ is the \textbf{derivator of biproduct $n$-cubes}.
\end{prop}

Note that property (iii) of the characterization of biproduct $n$-cubes suggests that such diagrams arise via restriction from a `larger shape where the length two morphisms are invertible'. This turns out to be true and will be taken care of by the remaining functor in~\eqref{eq:biproducts}.

\begin{figure}
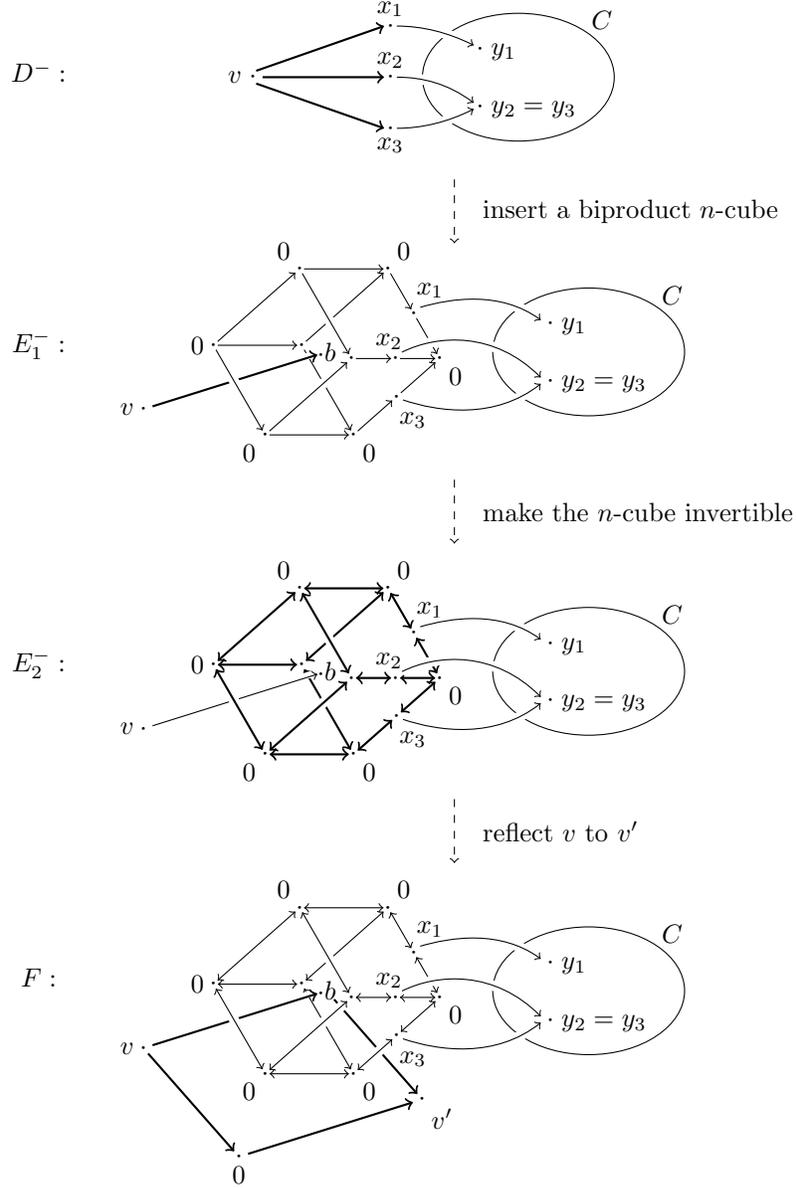

\centering
\TikzFigReflectionUnraveled
\caption{Intermediate steps in construction of reflection functors. Changes from step to step are drawn in bold.}
\label{fig:detailed-strategy}
\end{figure}

In fact, let $p\colon[2]\to R$ be the localization functor inverting the length two morphism $0\to 2$ in $[2]$, so that $R$ corepresents pairs of composable morphisms such that the composition is an isomorphism; see~\cite[\S7]{gst:tree} for a precise description of~$R$. We know that $p$ is a homotopical epimorphism~\cite[Prop.~7.3]{gst:tree}, and it is completely formal to see that the same is true for the $n$-fold product $q\colon[2]^n\to R^n$.

\begin{cor}[{\cite[Corollary~7.4]{gst:tree}}]\label{cor:inv-hyper}
Let \D be a derivator and $n\geq 1$. The functor $q\colon [2]^n\to R^n$ is a homotopical epimorphism and $q^\ast\colon \D^{R^n}\to \D^{[2]^n}$ induces an equivalence onto the full subderivator of $\D^{[2]^n}$ spanned by all diagrams $X$ such that
\[
X_{i_1,\ldots,i_{k-1},0,i_{k+1},\ldots,i_n}\to X_{i_1,\ldots,i_{k-1},2,i_{k+1},\ldots,i_n}
\]
is an isomorphism for all $i_1,\ldots,i_{k-1},i_{k+1},\ldots,i_n$ and $k$.
\end{cor}

Thus, in the stable case, there is the following result concerning the morphisms
\begin{equation}
\D^{n\cdot\bbone}=\D^{[1]^n_{=n-1}}\stackrel{(i_1)_\ast}{\to}\D^{[1]^n_{\geq n-1}}\stackrel{(i_2)_\ast}{\to}
\D^{[1]^n}\stackrel{(i_3)_!}{\to}\D^I\stackrel{(i_4)_\ast}{\to}\D^{[2]^n}\stackrel{q^\ast}{\to}\D^{R^n}.
\label{eq:biproducts-III}
\end{equation}
Let $\D^{R^n,\exx}\subseteq\D^{R^n}$ be the full subderivator spanned by all $X\in\D^{R^n}$ such that $q^\ast X$ is a biproduct $n$-cube, i.e., such that $q^\ast X\in\D^{[2]^n,\exx}$.

\begin{cor}[{\cite[Corollary~7.5]{gst:tree}}]\label{cor:inv-biprod}
Let~\D be a  stable derivator and $n\geq 2$. The morphisms \eqref{eq:biproducts-III} are fully faithful and induce an equivalence $\D^{n\cdot\bbone}\simeq\D^{R^n,\mathrm{ex}}$, which is pseudo-natural with respect to exact morphisms. The derivator $\D^{R^n,\exx}$ is the \textbf{derivator of invertible biproduct $n$-cubes}.
\end{cor}

With this preparation we now describe in more detail the second step in the above strategy behind the construction of general reflection morphisms (see \autoref{fig:detailed-strategy}). The above-mentioned morphism $\D^{D^-}\to\D^{D^+}$ is roughly obtained as follows.
\begin{enumerate}
\item Starting with an abstract representation $X\in\D^{D^-}$, we glue in a coherent biproduct $n$-cube centered at $\bigoplus_{i=1}^n X_{x_i}$. The corresponding morphism $\D^{D^-}\to\D^{E^-_1}$ is obtained by adapting the respective morphisms in \eqref{eq:biproducts-II}, and this step relies on the discussion of `free oriented gluing constructions' in~\S\ref{sec:glue}.
\item Next, using a variant of the functor $q\colon [2]^n\to R^n$, we invert the biproduct $n$-cubes, thereby constructing a restriction morphism $\D^{E^-_2}\to\D^{E^-_1}$. To understand this morphism, we study the compatibility of homotopical epimorphisms with `free oriented gluing constructions'; see~\S\ref{sec:glue-epi}.
\item As a next step, given a representation $X\in\D^{E^-_2}$, we extend it by passing from $X_v\to\bigoplus_{i=1}^n X_{x_i}$ to the corresponding cofiber square. To get our hands on the resulting morphism of derivators $\D^{E^-_2}\to\D^F$ we again apply results from~\S\ref{sec:glue}.
\item The steps so far yield a morphism of derivators $\D^{D^-}\to\D^F$. One observes that the category $F$ also comes with a functor $D^+\to F$. Dualizing the steps so far, we show that there is a similar morphism of derivators $\D^{D^+}\to\D^F$, and that the span $\D^{D^-}\to\D^F\ot\D^{D^+}$ restricts to the desired equivalence.
\end{enumerate}

These steps will be carried out in detail in \S\ref{sec:reflection-sep}, and combined with the above inflation and deflation steps, they are shown in \S\ref{sec:reflection} to yield the intended general reflection morphisms $\D^{C^-}\to\D^{C^+}$ and $\D^{C^+}\to\D^{C^-}$, showing that the categories $C^-$ and $C^+$ are strongly stably equivalent; see~\autoref{thm:reflect}. In the following two sections we first develop some of the necessary techniques.

\section{Free oriented gluing constructions}
\label{sec:glue}

In this section we study in more detail the gluing construction alluded to in~\S\ref{sec:guide}. In particular, we will see that these gluing constructions behave well with Kan extension morphisms. The results of this section and \S\ref{sec:glue-epi} are central to the construction of the reflection morphisms in \S\ref{sec:reflection-sep}.

To begin with let us consider the following construction. 

\begin{con}\label{con:glue}
Let $A_1,A_2\in\cCat$ be small categories, let $n\in\lN$, let $s_1,\ldots,s_n\in A_1$, and $t_1,\ldots, t_n\in A_2$. Moreover, let $[1]$ again be the poset $(0<1)$ considered as a category. The category $[1]$ comes with a functor $(0,1)\colon\bbone\sqcup\bbone\to [1]$ classifying the objects $0$ and $1$. Using this notation, we define the category $A$ to be the following pushout
\begin{equation}\label{eq:glue}
\vcenter{
\xymatrix{
\coprod_{i=1,\ldots,n}\bbone\sqcup\bbone\ar[r]^-{s\sqcup t}\ar[d]&A_1\sqcup A_2\ar[d]^-{(i_1,i_2)}\\
\coprod_{i=1,\ldots,n}[1]\ar[r]_-{\beta}&A,\pushoutcorner
}
}
\end{equation}
and call it the \textbf{free oriented gluing construction} associated to $(A_1,A_2,s,t)$. Given $k \in \{1, \dots, n\}$, we denote the image of the morphism $0 \to 1$ in the $k$-th copy of $[1]$ by $\beta_k\colon i_1(s_k) \to i_2(t_k)$. 
\end{con}

This construction clearly enjoys the following properties.

\begin{lem}\label{lem:glue}
In the situation of \eqref{eq:glue} the following properties are satisfied.
\begin{enumerate}
\item The functors $i_1\colon A_1\to A$ and $i_2\colon A_2\to A$ are fully faithful with disjoint images.
\item Every object in $A$ lies either in $i_1(A_1)$ or in $i_2(A_2)$.
\item There are no morphisms in $A$ from an object in $i_2(A_2)$ to an object in $i_1(A_1)$.
\item For every morphism $f\colon i_1(a_1)\to i_2(a_2)$ there is a unique $k\in\{1,\ldots,n\}$ and a unique factorization of $f$ as
\begin{equation}\label{eq:standard}
f\colon i_1(a_1)\stackrel{i_1(f')}{\longrightarrow}i_1(s_k)\stackrel{\beta_k}{\longrightarrow} i_2(t_k)\stackrel{i_2(f'')}{\longrightarrow} i_2(a_2).
\end{equation}
\end{enumerate}
\end{lem}
\begin{proof}
This is immediate from the construction of the pushout category in~\eqref{eq:glue} (see also \autoref{lem:normal-forms}).
\end{proof}

\begin{defn}\label{defn:glue-type}
We refer to the factorizations in \autoref{lem:glue}(iv) as \textbf{standard factorizations} and call the unique number $k\in\{1,\ldots,n\}$ the \textbf{type} of~$f$.
\end{defn}

\begin{eg}\label{eg:separated-glue}
Let $C\in\cCat$ and let $y_1,\ldots,y_n\in  C$ be a list of objects (possibly with repetition), let $t=y\colon n\cdot\bbone\to C$ be the corresponding functor. Moreover, note that $[1]^n_{\leq 1}$ is the source of valence $n$ which comes with the functor $s\colon n\cdot\bbone\to[1]^n_{\leq 1}$ classifying the objects different from the source. The pushout square
\[
\xymatrix{
\coprod_{i=1,\ldots,n}\bbone\sqcup\bbone\ar[r]^-{s\sqcup t}\ar[d]&[1]^n_{\leq 1}\sqcup C\ar[d]\\
\coprod_{i=1,\ldots,n}[1]\ar[r]&D^-,\pushoutcorner
}
\]
exhibits the category $D^-$ showing up in the outline of the strategy of the construction of general reflection morphisms (see \autoref{fig:rough-strategy}) as an instance of a free oriented gluing construction. There is a similar description of the category $D^+$ in \autoref{fig:rough-strategy}.
\end{eg}

\begin{eg}\label{eg:glue-one-point}
As a special case of \autoref{con:glue} we recover the one-point extensions of \cite[\S8]{gst:tree}. In fact, this is the case for the free oriented gluing construction associated to $(A_1,A_2,s,t)$ in the case where $n=1$ and $A_1$ or $A_2$ is the terminal category $\bbone$.
\end{eg}

\begin{con}\label{con:glue-comp}
We now consider two free oriented gluing constructions $A$ and $A'$ which are associated to $(A_1,A_2,s,t)$ and $(A'_1,A'_2,s',t')$, respectively. Let us assume that the second summands $A_2=A_2'$ as well as the targets $t=t'$ agree while there is a functor $u_1\colon A_1\to A_1'$ such that $s'= u_1\circ s$. This situation may be summarized by the following commutative diagram
\begin{equation}\label{eq:comp-glue}
\vcenter{
\xymatrix@-1.5pc{
\coprod_{i=1,\ldots,n}\bbone\sqcup\bbone \ar[rr]^-{s\sqcup t} \ar@/_0.8pc/[drrr]_-{s'\sqcup t}\ar[dd] &&
A_1\sqcup A_2 \ar[dr]^-{u_1\sqcup\id} \ar'[d][dd] \\
&&&
A_1'\sqcup A_2 \ar[dd]^-{(i'_1,i'_2)}\\
\coprod_{i=1,\ldots,n}[1] \ar@/_0.8pc/[drrr]_-{\beta'} \ar[rr]_-\beta &&
A \ar[dr]^-u\pushoutcorner\\
&&& A'.
}
}
\end{equation}
Here, both the front and the back face are the pushout squares defining the respective gluing constructions and $u\colon A\to A'$ is induced by the universal property of the back pushout square. We refer to the situation described in \eqref{eq:comp-glue} as \textbf{two compatible (free oriented) gluing constructions} (see \autoref{fig:compatible-gluings} for an illustration).
\end{con}

\begin{figure}
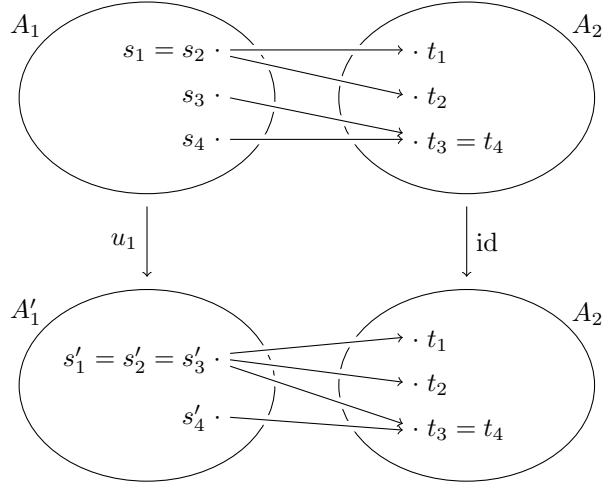

\centering
\TikzFigCompatibleGlueings
\caption{Two compatible (free oriented) gluing constructions.}
\label{fig:compatible-gluings}
\end{figure}

Combining the face on the right in \eqref{eq:comp-glue} with the inclusions of the respective first summands we obtain a commutative square of small categories, which we consider in two ways as a square populated by the identity transformation,
\begin{equation}\label{eq:glue-exact}
\vcenter{
\xymatrix{
A_1\ar[d]_-{u_1}\ar[r]^-{i_1}&A\ar[d]^-u&&A_1\ar[d]_-{u_1}\ar[r]^-{i_1}\drtwocell\omit{\id}&A\ar[d]^-u\\
A_1'\ar[r]_-{i_1'}&A',\ultwocell\omit{\id}&&A_1'\ar[r]_-{i_1'}&A'.
}
}
\end{equation}
The following proposition guarantees that Kan extensions along $u$ and Kan extensions along $u_1$ interact as expected.

\begin{prop}\label{prop:comp-glue}
If \eqref{eq:comp-glue} are two compatible gluing constructions, then both squares in~\eqref{eq:glue-exact} are homotopy exact, i.e., in every derivator the canonical mates
\[
(i_1')^\ast u_\ast\to (u_1)_\ast (i_1)^\ast\qquad\text{and}\qquad (u_1)_!(i_1)^\ast\to (i_1')^\ast u_!
\]
are isomorphisms.
\end{prop}
\begin{proof}
We first show that the square on the left in \eqref{eq:glue-exact} is homotopy exact, and show that the canonical mate $(i_1)_! u_1^\ast\to u^\ast (i_1')_!$ is an isomorphism. Since the functors $i_1\colon A_1\to A$ and $i_2\colon A_2\to A$ are jointly surjective, it suffices by (Der2) to show that the restrictions of the canonical mate with $i_1^\ast,i_2^\ast$ are isomorphisms. For the first case we consider the pastings
\[
\xymatrix{
A_1\ar[r]^-\id\ar[d]_-\id\drtwocell\omit{\id}&A_1\ar[r]^-{u_1}\ar[d]_-{i_1}\drtwocell\omit{\id}&A_1'\ar[d]^-{i_1'}
\ar@{}[rrd]|{=}&&
A_1\ar[r]^-{u_1}\ar[d]_-\id\drtwocell\omit{\id}&A_1'\ar[r]^-\id\ar[d]_-\id\drtwocell\omit{\id}&A_1'\ar[d]^-{i_1'}
\\
A_1\ar[r]_-{i_1}&A\ar[r]_-u&A'&&
A_1\ar[r]_-{u_1}&A_1'\ar[r]_-{i_1'}&A'.
}
\]
The fully faithfulness of $i_1,i_1'$ imply that the square to the very left and the square to the very right are homotopy exact (\autoref{egs:htpy-exact}). Moreover, the second square from the right is constant and hence homotopy exact. The functoriality of mates with respect to pasting implies that the restricted canonical mate $i_1^\ast(i_1)_! u_1^\ast\to i_1^\ast u^\ast (i_1')_!$ is an isomorphism.

Now, given an object $i_2(a_2)\in A$ we consider the pasting  
\[
\xymatrix{
\coprod_{k}A_2(t_k,a_2)\ar[r]^-r\ar[d]\drtwocell\omit{\id}&(i_1/i_2a_2)\ar[r]^-p\ar[d]\drtwocell\omit{}&A_1\ar[r]^-{u_1}\ar[d]_-{i_1}\drtwocell\omit{\id}&A_1'\ar[d]^-{i_1'}\\
\bbone\ar[r]_-\id&\bbone\ar[r]_-{i_2a_2}&A\ar[r]_-u&A'
}
\]
in which the square in the middle is a slice square. The functor $r$ sends a morphism $t_k\to a_2$ to the pair $(s_k, i_1s_k\to i_2t_k\to i_2a_2)\in(i_1/i_2a_2)$. Using \autoref{lem:glue} the reader easily checks that this functor is a right adjoint so that the above square on the left is homotopy exact by the homotopy finality of right adjoints (\autoref{egs:htpy-exact}). Note that the above pasting agrees with the pasting
\[
\xymatrix{
\coprod_{k}A_2(t_k,a_2)\ar[r]^-{r'}\ar[d]\drtwocell\omit{\id}&(i_1'/i_2'a_2)\ar[r]^-p\ar[d]\drtwocell\omit{}&A_1'\ar[d]^-{i_1'}\\
\bbone\ar[r]_-\id&\bbone\ar[r]_-{i_2'a_2}&A'
}
\]
given by a slice square and a similarly defined right adjoint functor $r'$. The functoriality of mates with pasting hence implies that $(i_1)_! u_1^\ast\to u^\ast (i_1')^\ast$ is an isomorphism at $i_2a_2$.

We now turn to the second claim and show that the canonical mate $u^\ast (i_1')_\ast\to (i_1)_\ast u_1^\ast$ is an isomorphism. Using again that $i_1,i_2$ are jointly surjective, it suffices to show that the corresponding restrictions of the canonical mate are invertible. Since $i_1,i_1'$ are sieves, both right Kan extensions are right extensions by terminal objects (\autoref{rmk:ext-by-zero-unpointed}), and the above canonical mate is hence automatically an isomorphism on objects of the form $i_2a_2$. It remains to show that its restriction along $i_1^\ast$ is an isomorphism and for that purpose we consider the diagram
\[
\xymatrix{
A_1\ar[r]^-\id\ar[d]_-\id\drtwocell\omit{\id}&A_1\ar[d]^-{i_1}&&
A_1\ar[r]^-\id\ar[d]_-{u_1}\drtwocell\omit{\id}&A_1\ar[d]^-{u_1}
\\
A_1\ar[r]_-{i_1}\ar[d]_-{u_1}\drtwocell\omit{\id}&A\ar[d]^-u\ar@{}[rr]|{=}&&
A_1'\ar[r]_-\id\ar[d]_-\id\drtwocell\omit{\id}&A_1'\ar[d]^-{i_1'}\\
A_1'\ar[r]_-{i_1'}&A'&&
A_1'\ar[r]_-{i_1'}&A'.
}
\]
Using the same arguments as in the first part of the proof we conclude that $i_1^\ast u^\ast (i_1')_\ast\to i_1^\ast(i_1)_\ast u_1^\ast$ is an isomorphism, concluding the proof.
\end{proof}

In the case that $u_1$ and, hence, $u$ is fully faithful there is the following convenient result.

\begin{cor} \label{cor:glue-img-Kan}
Let \eqref{eq:comp-glue} be two compatible gluing constructions such that $u_1$ and, hence, $u$ are fully faithful, and let \D be a derivator.
\begin{enumerate}
\item The right Kan extension morphism $u_\ast\colon\D^A\to\D^{A'}$ is fully faithful with essential image given by those $X$ such that $(i_1')^\ast X$ lies in the essential image of~$(u_1)_\ast\colon\D^{A_1}\to\D^{A_1'}$.
\item The left Kan extension morphism $u_!\colon\D^A\to\D^{A'}$ is fully faithful with essential image given by those $X$ such that $(i_1')^\ast X$ lies in the essential image of~$(u_1)_!\colon\D^{A_1}\to\D^{A_1'}$.
\end{enumerate}
\end{cor}
\begin{proof}
We give a proof of (i), the case of (ii) is dual. Since both $u_1$ and $u$ are fully faithful, the respective right Kan extension morphisms are fully faithful (\autoref{egs:htpy-exact}). Thus, the corresponding essential images consist precisely of those diagrams on which the respective units $\eta_1\colon\id\to (u_1)_\ast u_1^\ast$ and $\eta\colon\id\to u_\ast u^\ast$ are  isomorphisms. To express this differently we consider the following pastings
\[
\xymatrix{
A_1\ar[r]^-{i_1}\ar[d]_-{u_1}&A\ar[r]^-u\ar[d]^-u&A'\ar[d]^-=\ar@{}[rrd]|{=}&&
A_1\ar[r]^-{u_1}\ar[d]_-{u_1}&A_1'\ar[r]^-{i_1'}\ar[d]^-=&A'\ar[d]^-=\\
A_1'\ar[r]_-{i_1'}&A'\ar[r]_-=\ultwocell\omit{\id}&A'\ultwocell\omit{\id}&&
A_1'\ar[r]_-=&A_1'\ar[r]_-{i_1'}\ultwocell\omit{\id}&A'.\ultwocell\omit{\id}
}
\]
By \autoref{lem:ff-essim}, $X\in\D^{A'}$ lies in the essential image of~$u_\ast$ if and only if $(i_1')^\ast\eta$ is an isomorphism on~$X$. Using the compatibility of mates with pasting and the homotopy exactness of the square to the very left (\autoref{prop:comp-glue}), this is the case if and only if the canonical mate associated to the pasting on the left is an isomorphism on~$X$. But since the above two pastings agree, this is the case if and only if the canonical mate of the pasting on the right is an isomorphism on~$X$. As the square on the right is constant and hence homotopy exact, this is to say that $\eta_1$ is an isomorphism on~$(i_1')^\ast X$, i.e., that $(i_1')^\ast X$ is in the essential image of $(u_1)_\ast$ (by an additional application of \autoref{lem:ff-essim}).
\end{proof}

As we shall see in~\S\ref{sec:reflection-sep}, the results of this section allow us to add the desired biproduct $n$-cubes and (co)fiber squares needed for the reflection morphisms.
To also be able to pass to the invertible $n$-cube we include the following section.

\section{Gluing constructions and homotopical epimorphisms}
\label{sec:glue-epi}

In this section we continue the study of free oriented gluing constructions as defined in~\S\ref{sec:glue} and show that they are compatible with homotopical epimorphisms (\autoref{defn:htpy-epi}). The goal is to establish \autoref{thm:glue-epi} showing that given a pair of compatible gluing construction~\eqref{eq:comp-glue} such that $u_1$ is a homotopical epimorphism then so is $u$. Moreover, the essential images of the corresponding restriction morphisms $u_1^\ast$ and $u^\ast$ are related as desired. 

In the situation of two compatible gluing constructions~\eqref{eq:comp-glue}, the respective inclusions of the second summands induce the following commutative square, which we consider as being populated by the identity transformation as indicated in

\begin{equation}\label{eq:glue-exact-II}
\vcenter{
\xymatrix{
A_2\ar[d]_-=\ar[r]^-{i_2}&A\ar[d]^-u\\
A_2\ar[r]_-{i_2'}&A'.\ultwocell\omit{\id}
}
}
\end{equation}

\begin{prop}\label{prop:comp-glue-II}
Given two compatible oriented gluing constructions as in \eqref{eq:comp-glue} the commutative square \eqref{eq:glue-exact-II} is homotopy exact.
\end{prop}
\begin{proof}
To reformulate the claimed homotopy exactness of the square~\eqref{eq:glue-exact-II}, we consider the pasting on the left in
\[
\xymatrix{
\bbone\ar[r]^-{a_2}\ar[d]_=&A_2\ar[d]^-=\ar[r]^-{i_2}&A\ar[d]^-u\ar@{}[rd]|{=}&
\bbone\ar[r]^-{i_2a_2}\ar[d]_-=&A\ar[d]^-u\ar@{}[rd]|{=}&
\bbone\ar[r]^-{(i_2a_2,\id)}\ar[d]_=&(i_2'a_2/u)\ar[d]^-\pi\ar[r]^-q&A\ar[d]^-u\\
\bbone\ar[r]_-{a_2}&A_2\ar[r]_-{i_2'}\ultwocell\omit{\id}&A'\ultwocell\omit{\id}&
\bbone\ar[r]_-{i_2'a_2}&A'\ultwocell\omit{\id}&
\bbone\ar[r]_-=&\bbone\ultwocell\omit{\id}\ar[r]_-{i_2'a_2}&A',\ultwocell\omit{\id}
}
\]
in which the left square is constant and hence homotopy exact. Using (Der2) and the compatibility of mates with pasting we conclude that \eqref{eq:glue-exact-II} is homotopy exact if and only if the above pasting is homotopy exact for every $a_2\in A_2$. Note that this pasting is simply the above commutative square in the middle which in turn can be written as the above pasting on the right. In that pasting, the square on the right is a slice square and hence homotopy exact. The square on the left is given by the functor classifying the initial object $(i_2a_2,\id\colon i_2'a_2\to ui_2a_2)$ in the slice category $(i_2'a_2/u)$, and that square is hence homotopy exact by the homotopy initiality of left adjoint functors (\autoref{egs:htpy-exact}). The compatibility of homotopy exact squares with pasting concludes the proof.
\end{proof}

We again consider two compatible gluing constructions as in \eqref{eq:comp-glue}. In that notation, by \autoref{prop:comp-glue} there is a homotopy exact square
\begin{equation}\label{eq:glue-epi}
\vcenter{
\xymatrix{
A_1\ar[d]_-{u_1}\ar[r]^-{i_1}&A\ar[d]^-u\\
A_1'\ar[r]_-{i_1'}&A'
}
}
\end{equation}
of small categories.

\begin{prop}\label{prop:glue-epi}
Given two compatible gluing constructions as in \eqref{eq:comp-glue} such that $u_1\colon A_1\to A_1'$ is a homotopical epimorphism, then also $u\colon A\to A'$ is a homotopical epimorphism.
\end{prop}
\begin{proof}
By assumption, $u_1\colon A_1\to A'_1$ is a homotopical epimorphism, i.e., the unit $\eta_1\colon \id\to (u_1)_\ast u_1^\ast$ is an isomorphism. We have to show that so is also the unit
$\eta\colon\id\to u_\ast u^\ast$. Using that the inclusions $i_1'\colon A_1'\to A'$ and $i_2'\colon A_2'\to A'$ are jointly surjective, (Der2) implies that it is enough to show that $(i_1')^\ast\eta$ and $(i_2')^\ast\eta$ are isomorphisms. As for the first restriction, let us consider the pasting on the left in
\[
\xymatrix{
A_1\ar[r]^-{i_1}\ar[d]_-{u_1}&A\ar[r]^-u\ar[d]^-u&A'\ar[d]^-=\ar@{}[rrd]|{=}&&
A_1\ar[r]^-{u_1}\ar[d]_-{u_1}&A_1'\ar[r]^-{i_1'}\ar[d]^-=&A'\ar[d]^-=\\
A_1'\ar[r]_-{i_1'}&A'\ar[r]_-=\ultwocell\omit{\id}&A'\ultwocell\omit{\id}&&
A_1'\ar[r]_-=&A_1'\ar[r]_-{i_1'}\ultwocell\omit{\id}&A'.\ultwocell\omit{\id}
}
\]
The square to the left is homotopy exact by \autoref{prop:comp-glue}, and the compatibility of homotopy exact squares with pasting implies that $(i'_1)^\ast\eta$ is an isomorphism if and only if the pasting on the left is homotopy exact. Note that this pasting agrees with the pasting on the right in which the square to the right is constant and hence homotopy exact. Moreover, the homotopy exactness of square on the left is equivalent to $u_1$ being a homotopical epimorphism, showing that $(i_1')^\ast\eta$ is an isomorphism. 

In order to show that also the restriction $(i_2')^\ast\eta$ is an isomorphism, let us consider the pasting on the left in
\[
\xymatrix{
A_2\ar[r]^-{i_2}\ar[d]_-=&A\ar[r]^-u\ar[d]^-u&A'\ar[d]^-=\ar@{}[rrd]|{=}&&
A_2\ar[r]^-{i_2'}\ar[d]_-=&A'\ar[d]^-=\\
A_2'\ar[r]_-{i_2'}&A'\ar[r]_-=\ultwocell\omit{\id}&A'\ultwocell\omit{\id}&&
A_1'\ar[r]_-{i_2'}&A'.\ultwocell\omit{\id}
}
\]
Using similar arguments as in the previous case together with the homotopy exactness of the square to the very left (\autoref{prop:comp-glue-II}), we deduce that $(i_2')^\ast\eta$ is an isomorphism if and only if the pasting on the left is homotopy exact. Since this pasting agrees with the constant square on the very right, we conclude by the homotopy exactness of constant squares.
\end{proof}

In the situation of \autoref{prop:glue-epi} both restriction morphisms $u^\ast\colon\D^{A'}\to\D^A$ and $u_1^\ast\colon\D^{A_1'}\to\D^{A_1}$ are fully faithful for every derivator~\D. To show that the essential images are related as desired (see \autoref{thm:glue-epi}) we establish the following result. 

\begin{lem}\label{lem:glue-epi}
Let \eqref{eq:comp-glue} be two compatible gluing constructions such that $u_1\colon A_1\to A_1'$ is a homotopical epimorphism and let \D be a derivator. A diagram $X\in\D^A$ lies in the essential image of $u^\ast\colon\D^{A'}\to\D^A$ if and only if $i_1^\ast\epsilon\colon i_1^\ast u^\ast u_\ast X\to i_1^\ast X$ is an isomorphism.
\end{lem}
\begin{proof}
By \autoref{prop:glue-epi} the functor $u\colon A\to A'$ is also a homotopical epimorphism and $u^\ast\colon\D^{A'}\to\D^A$ is hence a fully faithful morphism of derivators. A diagram $X\in\D^A$ lies in the essential image of $u^\ast$ if and only if the counit $\epsilon\colon u^\ast u_\ast X\to X$ is an isomorphism. Using the joint surjectivity of $i_1\colon A_1\to A$ and $i_2\colon A_2\to A$, by (Der2) this is the case if and only if the restricted counits $i_1^\ast\epsilon,i_2^\ast\epsilon$ are isomorphisms on $X$. Hence, to conclude the proof it suffices to show that $i_2^\ast\epsilon$ is always an isomorphism, and to this end we consider the pasting on the left in 
\[
\xymatrix{
A_2\ar[r]^-{i_2}\ar[d]_-=&A\ar[r]^-=\ar[d]^-=&A\ar[d]^-u\ar@{}[rrd]|{=}&&
A_2\ar[r]^-{i_2}\ar[d]_-=&A\ar[d]^-u\\
A_2\ar[r]_-{i_2}&A\ar[r]_-u\ultwocell\omit{\id}&A'\ultwocell\omit{\id}&&
A_2\ar[r]_-{i_2'}&A'.\ultwocell\omit{\id}
}
\]
The homotopy exactness of constant squares and the compatibility of canonical mates with pasting implies that $i_2^\ast\epsilon$ is always an isomorphism if and only if the pasting on the left is homotopy exact. However, this pasting agrees with the square on the right, which is homotopy exact by \autoref{prop:comp-glue-II}.
\end{proof}

\begin{thm}\label{thm:glue-epi}
Given two compatible gluing constructions as in \eqref{eq:comp-glue} such that $u_1\colon A_1\to A_1'$ is a homotopical epimorphism, then also $u\colon A\to A'$ is a homotopical epimorphism. Moreover, $X\in\D^A$ lies in the essential image of $u^\ast\colon\D^{A'}\to\D^A$ if and only if $i_1^\ast X\in\D^{A_1}$ lies in the essential image of $u_1^\ast\colon\D^{A_1'}\to\D^{A_1}$.
\end{thm}
\begin{proof}
By \autoref{prop:glue-epi} the functor $u\colon A\to A'$ is a homotopical epimorphism and $u^\ast\colon\D^{A'}\to\D^A$, as a fully faithful morphism of derivators, induces an equivalence onto its essential image. A coherent diagram $X\in\D^A$ lies by \autoref{lem:glue-epi} in this essential image if and only if $i_1^\ast\epsilon\colon i_1^\ast u^\ast u_\ast X\to i_1^\ast X$ is an isomorphism. But, using the homotopy exactness of constant squares, this is the case if and only if the canonical mate associated to the pasting on the left in
\[
\xymatrix{
A_1\ar[r]^-{i_1}\ar[d]_-=&A\ar[r]^-=\ar[d]^-=&A\ar[d]^-u\ar@{}[rrd]|{=}&&
A_1\ar[r]^-=\ar[d]_-=&A_1\ar[r]^-{i_1}\ar[d]^-{u_1}&A\ar[d]^-u\\
A_1\ar[r]_-{i_1}&A\ar[r]_-u\ultwocell\omit{\id}&A'\ultwocell\omit{\id}&&
A_1\ar[r]_-{u_1}&A_1'\ar[r]_-{i_1'}\ultwocell\omit{\id}&A'\ultwocell\omit{\id}
}
\]
is an isomorphism on~$X$. Since the above two pastings agree, the compatibility of mates with respect to pasting together with the homotopy exactness of the square to the very right (\autoref{prop:comp-glue}) implies that $X\in\D^A$ lies in the essential image of $u^\ast$ if and only if the canonical mate $\epsilon_1 i_1^\ast\colon u_1^\ast (u_1)_\ast i_1^\ast\to i_1^\ast$ is an isomorphism on~$X$. Since $u_1^\ast\colon\D^{A_1'}\to\D^{A_1}$ is fully faithful the counit $\epsilon_1$ is an isomorphism on $i_1^\ast X$ if and only if $i_1^\ast X$ lies in the essential image of $u_1^\ast$.
\end{proof}

In the construction of reflection morphisms in~\S\ref{sec:reflection-sep} we will see that the results of this section allow us to pass from biproduct $n$-cubes to invertible biproduct $n$-cubes (compare again with the strategy outlined in~\S\ref{sec:guide}).

\section{Reflection morphisms: the separated case}
\label{sec:reflection-sep}

In this section we construct the reflection morphisms in abstract stable derivators and show them to be strong stable equivalences. The strategy behind the construction is described in \S\ref{sec:guide}. Here we deal only with the part of the construction depicted in the lower half of \autoref{fig:rough-strategy} and which is described in more detail in \autoref{fig:detailed-strategy}. Thus, we shall assume that the source/sink is ``separated'' from the category $C$ by freely added morphisms. The inflation/deflation steps indicated by the vertical dashed arrows in \autoref{fig:rough-strategy} are postponed to \S\ref{sec:reflection}.

More precisely, the goal is the following. Let $C\in\cCat$, and let $y_1,\ldots,y_n\in C$ be objects (not necessarily distinct). We can view this data as a functor $y\colon n\cdot\bbone\to C$. We obtain two new categories $D^-$ and $D^+$ by attaching a source of valence~$n$ and a sink of valence~$n$, respectively, to~$C$ by means of the free oriented gluing construction in the sense of~\S\ref{sec:glue} (see the first line of \autoref{fig:detailed-strategy}). Formally, we consider the two pushout diagrams in $\cCat$
\begin{equation}\label{eq:step-2-II}
\vcenter{
\xymatrix{
\coprod_{i=1}^{n}\bbone\sqcup\bbone\ar[r]^-{\mathsf{inc}\sqcup y}\ar[d]&([1]^n_{=n-1})^\lhd\sqcup C\ar[d] &
\coprod_{i=1}^{n}\bbone\sqcup\bbone\ar[r]^-{\mathsf{inc}\sqcup y}\ar[d]&([1]^n_{=n-1})^\rhd\sqcup C\ar[d]
\\
\coprod_{i=1}^{n}[1]\ar[r]_-k&D^-,\pushoutcorner &
\coprod_{i=1}^{n}[1]\ar[r]_-k&D^+,\pushoutcorner
}
}
\end{equation}
where $\mathsf{inc}$ stands for the obvious inclusions $n\cdot\bbone\to\source{n}=([1]^n_{=n-1})^\lhd$ and $n\cdot\bbone\to\sink{n}=([1]^n_{=n-1})^\rhd$.

Here we carry out the individual steps of the construction of a strong stable equivalence of $D^-$ and $D^+$; see~\autoref{fig:detailed-strategy}. Starting with a representation $X\in\D^{D^-}$ in a stable derivator~\D, this roughly amounts to the following.
\begin{enumerate}
\item Glue in a biproduct $n$-cube centered at $\bigoplus_{i=1}^nX_{x_i}$.
\item Pass to the invertible biproduct $n$-cube.
\item Add a cofiber square to the resulting morphism $X_v\to\bigoplus_{i=1}^nX_{x_i}$.
\end{enumerate}
At the level of shapes this corresponds to considering the first three functors in
\begin{equation}\label{eq:strategy}
D^-\to E^-_1\to E^-_2\to F\ot E^+_2\ot E^+_1\ot D^+,
\end{equation}
precise definitions of which are given below. 

As we discuss further below, the category~$F$ is symmetric in the following sense. If we begin with a representation $X\in\D^{D^+}$ and perform similar steps then we end up with a representation of the same category~$F\in\cCat$. At the level of shapes this amounts to considering the remaining three functors in \eqref{eq:strategy}.

We now turn to the first step which essentially amounts to gluing an $n$-cube $[2]^n$ to $D^-$, yielding the functor $D^-\to E^-_1$ in ~\eqref{eq:strategy}; see again \autoref{fig:detailed-strategy}. To define this functor, we consider the following diagram of small categories
\begin{equation}\label{eq:step-2}
\vcenter{
\xymatrix{
[1]^n_{=n-1}\ar[r]\ar[d]&[1]^n_{\geq n-1}\ar[r]\ar[d]&[1]^n\ar[r]\ar[d]&I\ar[r]\ar[d]&[2]^n\ar[d]\\
([1]^n_{=n-1})^\lhd\ar[r]_-{i_1}&([1]^n_{\geq n-1})^\lhd\ar[r]_-{i_2}&([1]^n)^\lhd\ar[r]_-{i_3} &I_1\pushoutcorner\ar[r]_-{i_4}&I_2\pushoutcorner,
}
}
\end{equation}
in which the two pushout squares to the right define the categories $I_1,I_2$, in which the top row is as in~\eqref{eq:biproducts}, and in which the two squares to the left are naturality squares. The functor $D^-\to E^-_1$ is obtained by an application of the free oriented gluing construction to the bottom row in \eqref{eq:step-2}. Thus, we consider the following diagram consisting of pushout squares
\begin{equation}\label{eq:step-2-III}
\vcenter{
\xymatrix@C=1.5em{
([1]^n_{=n-1})^\lhd\sqcup C\ar[d]\ar[r]\ar[d]&
([1]^n_{\geq n-1})^\lhd\sqcup C\ar[r]\ar[d]&([1]^n)^\lhd\sqcup C\ar[r]\ar[d] &I_1\sqcup C\ar[r]\ar[d]&I_2\sqcup C\ar[d]\\
D^-\ar[r]_-{j_1}&A_1\pushoutcorner\ar[r]_-{j_2}&A_2\pushoutcorner\ar[r]_-{j_3} &A_3\pushoutcorner\ar[r]_-{j_4}&E^-_1.\pushoutcorner
}
}
\end{equation}
Associated to the bottom row in this diagram there are the following fully faithful Kan extension morphisms
\begin{equation}\label{eq:step-2-IV}
\D^{D^-}\stackrel{(j_1)_\ast}{\to}\D^{A_1}\stackrel{(j_2)_\ast}{\to}\D^{A_2}\stackrel{(j_3)_!}{\to}\D^{A_3}\stackrel{(j_4)_\ast}{\to}\D^{E^-_1}.
\end{equation}
We note that the category $E^-_1$ comes by definition with a functor
\[
l\colon[2]^n\to I_2\to E^-_1
\]
(see \eqref{eq:step-2} and \eqref{eq:step-2-III}). For every stable derivator~\D we denote by $\D^{E^-_1,\exx}\subseteq\D^{E^-_1}$ the full subderivator spanned by all $X\in\D^{E^-_1}$ for which the $n$-cube $l^\ast X\in\D^{[2]^n}$ is a biproduct $n$-cube (see \autoref{prop:biproducts}). The following proposition implies that $\D^{E^-_1,\exx}$ indeed is a derivator.

\begin{prop}\label{prop:step-2}
Let \D be a stable derivator. The morphisms in \eqref{eq:step-2-IV} are fully faithful and induce an equivalence $\D^{D^-}\simeq\D^{E^-_1,\exx}$. 
This equivalence is pseudo-natural with respect to exact morphisms.
\end{prop}
\begin{proof}
The first part of this proof is very similar to the proof of \autoref{prop:biproducts} (see \cite[Prop.~4.9]{gst:tree}). We begin by considering the functors in the bottom row of \eqref{eq:step-2}. Since these functors are fully faithful, the associated
Kan extension morphisms
\begin{equation}\label{eq:step-2-V}
\D^{([1]^n_{=n-1})^\lhd}\stackrel{(i_1)_\ast}{\to}\D^{([1]^n_{\geq n-1})^\lhd}\stackrel{(i_2)_\ast}{\to} \D^{([1]^n)^\lhd}\stackrel{(i_3)_!}{\to}\D^{I_1}\stackrel{(i_4)_\ast}{\to}\D^{I_2}
\end{equation}
are also fully faithful. We now describe the essential images of the respective morphisms, and show that they induce the following pseudo-natural equivalences.
\begin{enumerate}
\item Since $i_1$ is a sieve, the morphism $(i_1)_\ast$ is right extension by zero and hence induces an equivalence onto the full subderivator of $\D^{([1]^n_{\geq n-1})^\lhd}$ defined by this vanishing condition.
\item One easily checks that $(i_2)_\ast$ precisely amounts to adding a strongly cartesian $n$-cube, hence induces a corresponding equivalence of derivators.
\item The functor $i_3$ is a cosieve and $(i_3)_!$ is hence left extension by zero, yielding an equivalence onto the full subderivator of $\D^{I_1}$ defined by this vanishing condition.
\item The morphism $(i_4)_\ast$ precisely amounts to adding strongly cartesian $n$-cubes. In fact, this follows as in the case of \autoref{prop:biproducts}; see~\cite[\S4]{gst:tree} for details.
\end{enumerate}

Now, recall that the functors in the bottom row of~\eqref{eq:step-2-III} are obtained from the corresponding functors in the bottom row of \eqref{eq:step-2} by the free oriented gluing construction. Hence, by \autoref{cor:glue-img-Kan} we can describe the respective essential images of the Kan extension morphisms in \eqref{eq:step-2-IV} in terms of the essential images of the corresponding morphisms in \eqref{eq:step-2-V}. The above explicit description of these latter essential images concludes the proof of the first statement. The pseudo-naturality with respect to exact morphisms follows since exact morphisms preserve right and left extensions by zero as well as strongly cartesian and strongly cocartesian $n$-cubes.
\end{proof}

The next step in this construction consists of inverting the biproduct $n$-cube $[2]^n$ in $E^-_1$, yielding the functor $E^-_1\to E^-_2$ in~\eqref{eq:strategy}; see again~\autoref{fig:detailed-strategy}. To give a precise definition of this functor, we begin by observing that the category $E^-_1$ is obtained from $[2]^n$ by two iterated free gluing constructions in the sense of \S\ref{sec:glue}. In fact, let $E_1\in\cCat$ be defined as the free oriented gluing construction on the left in
\begin{equation}\label{eq:step-3-I}
\vcenter{
\xymatrix{
\coprod_{i=1,\ldots,n}\bbone\sqcup\bbone\ar[r]\ar[d]&[2]^n\sqcup C\ar[d]&&
\bbone\sqcup\bbone\ar[r]\ar[d]&\bbone\sqcup E_1\ar[d]\\
\coprod_{i=1,\ldots,n}[1]\ar[r]&E_1,\pushoutcorner&&
[1]\ar[r]&E^-_1,\pushoutcorner
}
}
\end{equation}
obtained from $n\cdot\bbone\cong [1]^n_{=n-1}\to[1]^n\stackrel{[1,2]^n}{\to} [2]^n$ and $(y_1,\ldots,y_n)\colon n\cdot\bbone\to C$. Note that the category $E^-_1$ is simply the free oriented gluing construction associated to the functors $\id\colon\bbone\to\bbone$, and $(1,\ldots,1)\colon\bbone\to[2]^n\to E_1$, as depicted in the pushout square on the right in \eqref{eq:step-3-I}. In order to obtain the category $E^-_2$ we now simply replace the $n$-cube $[2]^n$ by the invertible $n$-cube~$R^n$, as defined prior to \autoref{cor:inv-hyper}. In detail, we define $E^-_2$ as the corresponding two-step free oriented gluing construction described via the pushout squares
\begin{equation}\label{eq:step-3-II}
\vcenter{
\xymatrix{
\coprod_{i=1,\ldots,n}\bbone\sqcup\bbone\ar[r]\ar[d]&R^n\sqcup C\ar[d]&&
\bbone\sqcup\bbone\ar[r]\ar[d]&\bbone\sqcup E_2\ar[d]\\
\coprod_{i=1,\ldots,n}[1]\ar[r]&E_2,\pushoutcorner&&
[1]\ar[r]&E^-_2.\pushoutcorner
}
}
\end{equation}

Finally, the functor $r\colon E^-_1\to E^-_2$ is obtained by tracing the homotopical epimorphism $q\colon[2]^n\to R^n$ (\autoref{cor:inv-hyper}) through the above constructions, thereby first obtaining a functor $E_1\to E_2$ and then $r\colon E^-_1\to E^-_2$ (\eqref{eq:step-3-I} and \eqref{eq:step-3-II} yield two pairs of compatible oriented gluing constructions in the sense of \S\ref{sec:glue}).

To perform the next step of the construction of reflection functors we now consider the commutative square
\begin{equation}\label{eq:step-3-III}
\vcenter{
\xymatrix{
[2]^n\ar[d]_-q\ar[r]^-i&E^-_1\ar[d]^-r\\
R^n\ar[r]_-j&E^-_2
}
}
\end{equation}
to which we apply our results from \S\ref{sec:glue-epi}.

\begin{prop}\label{prop:step-3}
The functor $r\colon E^-_1\to E^-_2$ is a homotopical epimorphism. Moreover, for every derivator~\D, a diagram $X\in\D^{E^-_1}$ lies in the essential image of $r^\ast\colon\D^{E^-_2}\to\D^{E^-_1}$ if and only if $i^\ast X\in\D^{[2]^n}$ lies in the essential image of $q^\ast\colon\D^{R^n}\to\D^{[2]^n}$.
\end{prop}
\begin{proof}
The following diagram expresses that $r\colon E^-_1\to E^-_2$ is obtained in two steps as a free oriented gluing construction starting with $q\colon[2]^n\to R^n$,
\[
\xymatrix{
[2]^n\ar[r]\ar[d]_-q&E_1\ar[r]\ar[d]&E^-_1\ar[d]^-r\\
R^n\ar[r]&E_2\ar[r]&E^-_2.
}
\]
Since $q$ is a homotopical epimorphism and we have a description of the essential image of $q^\ast\colon\D^{[2]^n}\to\D^{R^n}$ (\autoref{cor:inv-hyper}), the result follows from two applications of \autoref{thm:glue-epi}.
\end{proof}

The morphism $r^\ast$ induces an equivalence onto its essential image defined by invertibility conditions (\autoref{cor:inv-hyper}). We are interested in the following restriction of this equivalence. Note that the category~$E^-_2$ comes by construction with a functor
$j\colon R^n\to E^-_2$ (see~\eqref{eq:step-3-II}). For every stable derivator~\D, we denote by $\D^{E^-_2,\exx}\subseteq\D^{E^-_2}$ the full subderivator spanned by all diagrams $X\in\D^{E^-_2}$ for which the $n$-cube $j^\ast X\in\D^{R^n}$ is an invertible biproduct $n$-cube in the sense of \autoref{cor:inv-biprod}.
Recall also the definition of the derivator~$\D^{E^-_1,\exx}$ as considered in~\autoref{prop:step-2}.

\begin{cor}\label{cor:step-3}
Let~\D be a stable derivator. The morphism $r^\ast\colon\D^{E^-_2}\to\D^{E^-_1}$ induces an equivalence of derivators $\D^{E^-_2,\exx}\simeq\D^{E^-_1,\exx}$ which is pseudo-natural with respect to exact morphisms of derivators.
\end{cor}
\begin{proof}
This is immediate from \autoref{cor:inv-hyper} and \autoref{prop:step-3}.
\end{proof}

The third step in the construction of reflection morphisms amounts to extending the morphisms $X_v\to\bigoplus_{i=1}^nX_{x_i}$ in abstract representations to cofiber squares, as will be made precise by the functor $E^-_2\to F$ in \eqref{eq:strategy}; see again \autoref{fig:detailed-strategy}. We recall that cofiber squares in pointed derivators are constructed as follows (see~\cite[\S3.3]{groth:ptstab}). Let the functor $[1]\to\square=[1]\times[1]$ classify the top horizontal morphism $(0,0)\to(1,0)$ and let $[1]\stackrel{i}{\to}\ulcorner\stackrel{j}{\to}\square$ be the obvious factorization of it. For every pointed derivator~\D the corresponding Kan extension morphisms
\begin{equation}\label{eq:cof-square}
\D^{[1]}\stackrel{i_\ast}\to\D^\ulcorner\stackrel{j_!}{\to}\D^\square
\end{equation}
are fully faithful. Since $i$ is a sieve, $i_\ast$ is right extension by zero (\autoref{prop:ext-by-zero}). It follows that \eqref{eq:cof-square} induces an equivalence of derivators $\D^{[1]}\simeq\D^{\square,\exx}$, where $\D^{\square,\exx}\subseteq\D^\square$ is the full subderivator spanned by the \textbf{cofiber squares}, i.e., those coherent squares $X\in\D^\square$ having the following properties.
\begin{enumerate}
\item The square vanishes at the lower left corner, $X_{0,1}\cong 0$.
\item The square is cocartesian.
\end{enumerate}
This construction is clearly pseudo-natural with respect to right exact morphisms. Given a coherent morphism $X=(f\colon x\to y)\in\D^{[1]}$ the corresponding cofiber square looks like
\[
\xymatrix{
x\ar[r]^-f \ar[d]&y\ar[d]^-{\cof(f)}\\
0\ar[r]&z.\pushoutcorner
}
\]

To prepare the corresponding relative construction, we consider the following diagram of small categories
\begin{equation}\label{eq:step-4}
\vcenter{
\xymatrix{
\bbone\ar[r]^-1\ar[d]_-{(1,\ldots,1)}&[1]\ar[r]^-i\ar[d]&\ulcorner\ar[r]^-j\ar[d]^-{l_1}&\square\ar[d]^-{l_2}\\
R^n\ar[r]&B_1\pushoutcorner\ar[r]_-{i_1}&B_2\pushoutcorner\ar[r]_-{i_2}&B,\pushoutcorner
}
}
\end{equation}
consisting of pushout squares. The square to the left exhibits $B_1$ as a one-point extension of $R^n$ (\autoref{eg:glue-one-point}). And the category~$B$ is obtained from the invertible $n$-cube~$R^n$ by attaching a new morphism with target the center $(1,\ldots,1)\in R^n$ and a square containing this morphism as top horizontal morphism. (The category~$F$ as well as $E^-_2\to F$ in \eqref{eq:strategy} will be obtained from \eqref{eq:step-4} by a free oriented gluing construction.) We begin by considering a pointed derivator~\D and the Kan extension morphisms
\begin{equation}\label{eq:step-4-II}
\D^{B_1}\stackrel{(i_1)_\ast}{\to}\D^{B_2}\stackrel{(i_2)_!}{\to}\D^B.
\end{equation}
Let $\D^{B_2,\exx}\subseteq\D^{B_2}$ be the full subderivator spanned by all $X\in\D^{B_2}$ such that $l_1^\ast X$ vanishes at $(0,1)$. Similarly, let $\D^{B,\exx}\subseteq\D^B$ be the full subderivator spanned by those diagrams $X\in\D^B$ such that $l_2^\ast X$ is a cofiber square.

\begin{lem}\label{lem:step-4}
Let \D be a pointed derivator.
\begin{enumerate}
\item The morphism $(i_1)_\ast$ is fully faithful and induces $\D^{B_1}\simeq\D^{B_2,\exx}$.
\item The morphism $(i_2)_!$ is fully faithful with essential image the full subderivator of $\D^B$ spanned by all $X$ such that $l_2^\ast X$ is cocartesian.
\item The morphisms in \eqref{eq:step-4-II} induce an equivalence $\D^{B_1}\simeq\D^{B,\exx}$.
\end{enumerate}
These equivalences are pseudo-natural with respect to right exact morphisms.
\end{lem}
\begin{proof}
We leave it to the reader to work out the necessary homotopy (co)finality arguments and apply~\cite[Prop.~3.10]{groth:ptstab}.
\end{proof}

We note that the category~$E^-_2$ can be obtained as a free oriented gluing construction from $B_1$. In fact, associated to the functor
\[
n\cdot\bbone=[1]^n_{=n-1}\to[1]^n\stackrel{[1,2]^n}{\to}[2]^n\stackrel{q}{\to}R^n\to B_1
\]
and $y=(y_1,\ldots,y_n)\colon n\cdot\bbone\to C$ there is the free oriented gluing construction given by the pushout square on the left in
\begin{equation}\label{eq:step-4-III}
\vcenter{
\xymatrix{
\coprod_{i=1,\ldots,n}\bbone\sqcup\bbone\ar[r]\ar[d]&B_1\sqcup C\ar[d]\ar[r]&B_2\sqcup C\ar[r]\ar[d]&B\sqcup C\ar[d]\\
\coprod_{i=1,\ldots,n}[1]\ar[r]&E^-_2\pushoutcorner\ar[r]_-{j_1}&F_1\pushoutcorner\ar[r]_-{j_2}&F.\pushoutcorner
}
}
\end{equation}
The remaining two pushout squares are induced by the bottom row in \eqref{eq:step-4}. Thus, in the terminology of \S\ref{sec:glue} we have two pairs of compatible oriented gluing constructions. For every derivator~\D the Kan extension morphisms
\begin{equation}\label{eq:step-4-IV}
\D^{E^-_2}\stackrel{(j_1)_\ast}{\to}\D^{F_1}\stackrel{(j_2)_!}{\to}\D^F.
\end{equation}
are fully faithful. Note that the category~$F$ comes with a functor $l\colon\square\to B\to F$; see~\eqref{eq:step-4} and \eqref{eq:step-4-III}.

\begin{prop}\label{prop:step-4}
Let \D be a pointed derivator. The morphisms \eqref{eq:step-4-IV} are fully faithful and induce an equivalence onto the full subderivator of $\D^F$ spanned by all $X\in\D^F$ such that $l^\ast X\in\D^\square$ is a cofiber square. This equivalence is pseudo-natural with respect to right exact morphisms.
\end{prop}
\begin{proof}
Since we are in the context of two pairs of free oriented gluing constructions, this is immediate from two applications of \autoref{cor:glue-img-Kan} to \autoref{lem:step-4}.
\end{proof}

We are interested in the following induced equivalence. Note that associated to the category~$F$ there are functors
\[
l\colon\square\to F,\quad \text{and}\quad m\colon R^n\to F;
\]
see~\eqref{eq:step-4} and \eqref{eq:step-4-III}. Given a stable derivator \D, we denote by $\D^{F,\exx}\subseteq\D^F$ the full subderivator spanned by all $X\in\D^F$ satisfying the following properties.
\begin{enumerate}
\item The square $l^\ast X\in\D^\square$ is a cofiber square.
\item The $n$-cube $m^\ast X\in\D^{R^n}$ is an invertible biproduct $n$-cube.
\end{enumerate}
Recall also the definition of the derivator~$\D^{E^-_2,\exx}$ as considered in~\autoref{cor:step-3}.

\begin{cor}\label{cor:step-4}
Let \D be a stable derivator. The morphisms~\eqref{eq:step-4-IV} induce an equivalence of derivators~$\D^{E^-_2,\exx}\simeq\D^{F,\exx}$ which is pseudo-natural with respect to exact morphisms.
\end{cor}
\begin{proof}
This is immediate from \autoref{prop:step-4} and the defining exactness and vanishing conditions of $\D^{E^-_2,\exx}$ and $\D^{F,\exx}$.
\end{proof}

It now suffices to assemble the above individual steps in order to settle the reflection morphisms in the separated case.

\begin{thm}\label{thm:reflect-sep}
Let $C\in\cCat$, let $y_1,\ldots,y_n\in C$ (not necessarily distinct), and let $D^-,D^+\in\cCat$ be as in \eqref{eq:step-2-II}. The categories $D^-$ and~$D^+$ are strongly stably equivalent.
\end{thm}
\begin{proof}
As discussed at the beginning of this section, the functors in~\eqref{eq:strategy} correspond to the respective steps in the construction of the strong stable equivalence. The results \autoref{prop:step-2}, \autoref{cor:step-3}, and \autoref{cor:step-4} take care of the first three steps. In fact, they show that for every stable derivator \D there are equivalences of derivators
\[
\D^{D^-}\simeq\D^{E_1^-,\exx}\simeq\D^{E_2^-,\exx}\simeq\D^{F,\exx},
\]
which are pseudo-natural with respect to exact morphisms.

If we start with an abstract representation of $D^+$ instead, then, as indicated by the remaining three functors in \eqref{eq:strategy}, we can perform similar constructions to again obtain an abstract representation of $F$. We leave it to the reader to verify that this way we in fact construct a category isomorphic to $F$. (The arguments for this are essentially the same as in the case of \cite[Lemma~9.15]{gst:tree}.) At the level of derivators of representations, this amounts to additional pseudo-natural equivalences
\[
\D^{D^+}\simeq\D^{E_1^+,\exx}\simeq\D^{E_2^+,\exx}\simeq\D^{F,\exx},
\]
which are similar to \autoref{prop:step-2}, \autoref{cor:step-3}, and \autoref{cor:step-4}. These steps respectively amount to gluing in a biproduct $n$-cube, inverting the $n$-cube, and adding a fibre square. Since cofiber squares and fiber squares agree in stable derivators, it follows that the essential image of these three steps is again given by the derivator~$\D^{F,\exx}$ as described prior to~\autoref{cor:step-4}. Putting these pseudo-natural equivalences together, 
\[
\D^{D^-}\simeq\D^{E_1^-,\exx}\simeq\D^{E_2^-,\exx}\simeq\D^{F,\exx}\simeq\D^{E_2^+,\exx}\simeq\D^{E_1^+,\exx}\simeq\D^{D^+},
\]
we obtain the desired strong stable equivalence $\D^{D^-}\simeq\D^{D^+}$.
\end{proof}

\section{Detection criteria for homotopical epimorphisms}
\label{sec:detect}

The aim of this section is to establish two simple detection results for homotopical epimorphisms. These will be used in \S\ref{sec:reflection} to construct reflection morphisms in the general case and thereby to complete the plan from \S\ref{sec:guide}. 

The first criterion is completely straightforward; we show that (co)reflective (co)localizations are homotopical epimorphisms (compare to~\cite[Prop.~6.5]{gst:tree}).

\begin{prop}\label{prop:refl-loc-htpical-epi}
Let $(l,r)\colon A\rightleftarrows B$ be an adjunction of small categories with unit $\eta\colon\id\to rl$ and counit $\ep\colon lr\to\id$.
\begin{enumerate}
\item For every prederivator \D there is an adjunction 
\[
(r^\ast,l^\ast,\eta^\ast\colon\id\to l^\ast r^\ast,\ep^\ast\colon r^\ast l^\ast\to\id)\colon\D^A\rightleftarrows\D^B.
\]
\item If $l$ is a reflective localization, i.e., $r$ is fully faithful, then $l$ is a homotopical epimorphism. Moreover, $X\in\D^A$ lies in the essential image of $l^\ast$ if and only if $X_{\eta_a}\colon X_a\to X_{rla}$ is an isomorphism for all $a\in A-r(B)$.
\item If $r$ is a coreflective colocalization, i.e., $l$ is fully faithful, then $r$ is a homotopical epimorphism. Moreover, $Y\in\D^B$ lies in the essential image of $r^\ast$ if and only if $Y_{\ep_b}\colon Y_{lrb}\to Y_b$ is an isomorphism for all $b\in B-l(A)$.
\end{enumerate} 
\end{prop}
\begin{proof}
The first statement is immediate from the fact that every prederivator \D defines a $2$-functor
\[
\D^{(-)}\colon\cCat\op\to\cPDER\colon A\mapsto \D^A
\]
and since $2$-functors preserve adjunctions. By duality it suffices to establish the second statement. Since $r$ is fully faithful, the counit $\ep\colon lr\to\id$ is an isomorphism, hence so is the counit $\ep^\ast\colon r^\ast l^\ast\to\id$. But this means that $l^\ast\colon\D^B\to\D^A$ is fully faithful, i.e., that $l\colon A\to B$ is a homotopical epimorphism. The essential image of $l^\ast$ consists precisely of those $X\in\D^A$ such that the unit $\eta^\ast\colon X\to l^\ast r^\ast X$ is an isomorphism. By (Der2) this is the case if and only if $\eta^\ast_a$ is an isomorphism for every $a\in A$. Now, the triangular identity
\[
\id=\ep^\ast r^\ast \circ r^\ast\eta^\ast\colon r^\ast\to r^\ast l^\ast r^\ast\to r^\ast
\]
and the invertibility of $\ep^\ast$ implies that $r^\ast\eta^\ast$ is an isomorphism. Hence to characterize the essential image of $l^\ast$ it suffices to check $\eta^\ast$ at all objects $a\in A-r(B)$.
\end{proof}

This first criterion is already enough for one of the inflation and deflation steps in \S\ref{sec:reflection}. For the remaining one we establish the following additional criterion, which will be applied to more general localization functors. While these functors do not necessarily admit adjoints, they are still essentially surjective, thereby making the first condition in the coming proposition automatic.

\begin{prop}\label{prop:detect}
Let $u\colon A\to B$ be essentially surjective, let \D be a derivator, and let $u^\ast\colon \D^B\to \D^A$ be the restriction morphism. Let us assume further that $\E\subseteq\D^A$ is a full subprederivator such that
\begin{enumerate}
\item the essential image $\im(u^\ast)$ lies in $\E$, i.e., $\im(u^\ast)\subseteq \E\subseteq \D^A$, and
\item the unit $\eta\colon X\to u^\ast u_! X$ is an isomorphism for all $X\in\E$.
\end{enumerate}
Then $u^\ast\colon\D^B\to\D^A$ is fully faithful and $\im (u^\ast)=\E$. In particular, $\E$ is a derivator.
\end{prop}
\begin{proof}
To prove that $u^\ast$ is fully faithful it suffices to show that $\ep\colon u_! u^\ast \to \id$ is a natural isomorphism. The assumptions imply that $\eta u^\ast$ is a natural isomorphism. Hence, by the triangular identity
\[ 
\xymatrix@1{
\id=u^\ast\ep\circ\eta u^\ast\colon u^\ast \ar[r]^-{\eta u^\ast} & u^\ast u_! u^\ast \ar[r]^-{u^\ast \ep} & u^\ast 
}
\]
it follows that also $u^\ast \ep$ is an isomorphism. In order to conclude that $\ep$ is a natural isomorphism, it suffices by (Der2) to show that $b^\ast \ep$ is an isomorphism for every $b\in B$. This follows immediately from the essential surjectivity of $u$ and the fact that $u^\ast\ep$ is invertible. 

Using the fully faithfulness of $u^\ast$, its essential image consists precisely of those $X\in\D^A$ such that the unit $\eta\colon X\to u^\ast u_! X$ is an isomorphism. The assumptions (i) and (ii) immediately imply that this is the case if and only if $X\in\E$. Finally, $\E$ is also a derivator by the invariance of derivators under equivalences.
\end{proof}

Thus, once we make an \emph{e}ducated guess \E satisfying the above assumptions we get an equivalence onto \E. The relation to homotopical epimorphisms is as follows.

\begin{rmk}
In our later applications the subprederivator $\E\subseteq\D^A$ is a full subprederivator $\D^{A,\exx}$ determined by some exactness conditions. Recall from \cite[\S 3]{gst:basic} that such exactness conditions are formalized by certain (co)cones in $A$ to be populated by (co)limiting (co)cones. As a special case this includes the assumption that certain morphisms are populated by isomorphisms.

In such a situation we hence start with a full subprederivator $\D^{A,\exx}\subseteq\D^A$ for \emph{every} derivator \D. If the assumptions of  \autoref{prop:detect} are satisfied, then this implies first that $u\colon A\to B$ is a homotopical epimorphism and second that the essential image of $u^\ast$ is $\im(u^\ast)=\D^{A,\exx}$.
\end{rmk}

To be able to apply \autoref{prop:detect} in specific situations, it is useful to have better control over the adjunction unit $\eta\colon\id\to u^\ast u_!$.

\begin{con}\label{con:colim-mor}
Let \D be a derivator, $A\in\cCat$ and let $a\in A$. Associated to the square
\[
\xymatrix{
\bbone\ar[r]^-a\ar[d]&A\ar[d]^-{\pi_A}\\
\bbone\ar[r]&\bbone
}
\]
there is the canonical mate 
\begin{equation}\label{eq:can-colim-map}
a^\ast\to\colim_A.
\end{equation}
As a special case relevant in later applications, given a functor $u\colon A\to B$ and $a\in A$ there is the functor $p\colon (u/ua)\to A$. Whiskering the mate \eqref{eq:can-colim-map} in the case of $(a,\id\colon ua\to ua)\in(u/ua)$ with $p^\ast$ we obtain a canonical map
\begin{equation}\label{eq:lkan-colim-map}
a^\ast=(a,\id_{ua})^\ast p^\ast\to \colim_{(u/ua)}p^\ast.
\end{equation} 
\end{con}

\begin{lem}\label{lem:detect-Kan-to-colim}
Let \D be a derivator, $u\colon A\to B$, and $a\in A$. The component of the unit $a^\ast\eta\colon a^\ast\to a^\ast u^\ast u_!$ is isomorphic to $a^\ast\to\colim_{(u/ua)}p^\ast$ \eqref{eq:lkan-colim-map}. In particular, $\eta_a$ is an isomorphism if and only if this is the case for \eqref{eq:lkan-colim-map}.
\end{lem}
\begin{proof}
To reformulate that the adjunction unit $\eta_a$ is an isomorphism we consider the pasting on the left in 
\[
\xymatrix{
\bbone\ar[r]^-a\ar[d]&A\ar[r]\ar[d]&A\ar[d]^-u&&\bbone\ar[r]^-{(a,\id_{ua})}\ar[d]&(u/ua)\ar[r]^-p\ar[d]\drtwocell\omit{}&A\ar[d]^-u\\
\bbone\ar[r]_-a&A\ar[r]_-u&B,&&\bbone\ar[r]&\bbone\ar[r]_-{ua}&B,
}
\]
in which the square to the left is constant and hence homotopy exact. Note that this pasting agrees with the pasting on the right in which the square to the right is a slice square and hence also homotopy exact. The functoriality of canonical mates with pasting concludes the proof.
\end{proof}

We will later apply the previous lemma in situations in which the slice category admits homotopy final functors from certain simpler shapes. For this purpose we collect the following result.

\begin{lem} \label{lem:detect-final}
Let $u\colon A\to B$ be a homotopy final functor and let $a\in A$.
\begin{enumerate}
\item The map $u(a)^\ast\to\colim_B$ \eqref{eq:can-colim-map} is naturally isomorphic to $a^\ast u^\ast\to\colim_A u^\ast$, the whiskering of an instance of \eqref{eq:can-colim-map} with $u^\ast$.
\item If $A$ admits a terminal object $\infty$, then the map $a^\ast\to\colim_A$ \eqref{eq:can-colim-map} is naturally isomorphic to $a^\ast\to\infty^\ast$.
\end{enumerate}
\end{lem}
\begin{proof}
Using the functoriality of canonical mates, for the first statement it suffices to observe that the following two pasting agree
\[
\xymatrix{
\bbone \ar[d] \ar[r]^-a &A \ar[d]\ar[r]^-u & B \ar[d] &&
\bbone \ar[d] \ar[r]^-{ua} & B \ar[d]\\
\bbone \ar[r] & \bbone\ar[r] & \bbone, &&
\bbone \ar[r] & \bbone,
}
\]
and that the square in the middle is homotopy exact by assumption on $u$. For the second statement it suffices to unravel the definition of \eqref{eq:can-colim-map} using $\infty^\ast$ as a model for $\colim_A$.
\end{proof}

We finish the section with another lemma related to \autoref{con:colim-mor} which will be useful when dealing with a more complicated instance of \autoref{prop:detect} in the next section.

\begin{lem} \label{lem:detect-ff}
Let \D be a derivator, let $u\colon A \to B$ be fully faithful, and let $a\in A$. The map $a^\ast\to\colim_A$ \eqref{eq:can-colim-map} at $X\in\D^A$ is isomorphic to $u(a)^\ast \to \colim_B$ \eqref{eq:can-colim-map} at $u_!X$.
\end{lem}
\begin{proof}
Considering the pasting on the left in the diagram
\[
\vcenter{
\xymatrix{
&A\ar[r]\ar[d]&A\ar[d]^-u\\
\bbone\ar[r]^-a\ar[d]&A\ar[r]^-u\ar[d]&B\ar[d]\\
\bbone\ar[r]&\bbone\ar[r]&\bbone,
}
}
\qquad\qquad
\vcenter{
\xymatrix{
a^\ast\ar[r]^-\eta_-\cong\ar[d]&a^\ast u^\ast u_!\ar[d]\\
\colim_A\ar[r]_-\cong&\colim_B u_!,\\
}
}
\]
it is immediate from the functoriality of mates with pasting that the square on the right commutes.
\end{proof}

\section{General reflection morphisms}
\label{sec:reflection}

In this section we implement the remaining steps of the strategy outlined in \S\ref{sec:guide}, namely the inflation and deflation steps from \autoref{fig:rough-strategy}. This will allow us to finish the construction of a strong stable equivalence between the categories $C^+$ and $C^-$ depicted in \autoref{fig:adjoining} (see \autoref{thm:reflect}).

We start by formalizing the construction of the categories $C^-$ and $C^+$. Let $C\in\cCat$, and let $y_1,\ldots,y_n\in C$ be objects. We denote by $y\colon n\cdot\bbone\to C$ the resulting functor. For all preparatory results before \autoref{cor:contract-sep-mor}, we adopt the following hypothesis which will allow us to apply results from \autoref{sec:amalgamation}.

\begin{hyp} \label{hyp:y's}
The functor $y\colon n\cdot\bbone \to C$ is injective on objects. Equivalently, $y_1, y_2, \ldots, y_n$ are pairwise distinct objects of $C$.
\end{hyp}

We obtain $C^-$ and $C^+$ by attaching a source of valence~$n$ and a sink of valence~$n$ to~$C$, respectively. More precisely, the source of valence~$n$ is the cone~$\source{n}$ obtained from~$n\cdot\bbone$ by adjoining an initial object, and dually for the sink~$\sink{n}$. Using the obvious inclusion functors $n\cdot\bbone\to\source{n}$ and $n\cdot\bbone\to\sink{n}$ we define $C^-$ and $C^+$ as the respective pushouts in
\begin{equation}\label{eq:C+C-}
\vcenter{
\xymatrix{
n\cdot\bbone\ar[r]^-y\ar[d]&C\ar[d]&&n\cdot\bbone\ar[r]^-y\ar[d]&C\ar[d]\\
\source{n}\ar[r]&C^-\pushoutcorner,&&\sink{n}\ar[r]&C^+.\pushoutcorner
}
}
\end{equation}
Assuming \autoref{hyp:y's}, note that $C \to C^+$ and $C \to C^-$ are fully faithful by \autoref{prop:amalg-full}, and we view these functors as inclusions.

As already mentioned in \S\ref{sec:guide}, the two inflation and deflation steps are not dual to each other. Starting with a representation of $C^-$ we separate the morphisms adjacent to the source by adding morphisms pointing in the \emph{same} direction, while in the other case we add morphisms pointing in the \emph{opposite} direction. 

\begin{figure}
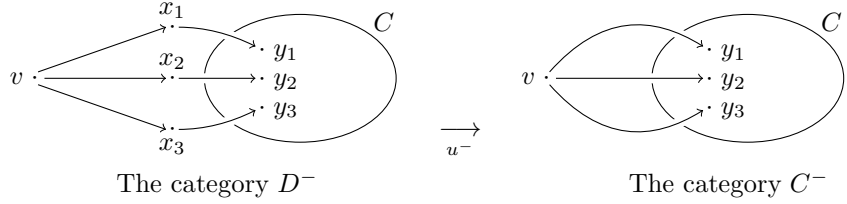

\centering
\TikzFigSeparateSource
\caption{The functor $u^-\colon D^- \to C^-$ which contracts the edges $x_i \to y_i$. It is used to separate the source of $C^-$.}
\label{fig:separate-C^-}
\end{figure}

Let us start with the easier case and consider the functor $u^-\colon D^-\to C^-$ as shown in \autoref{fig:separate-C^-}. Formally we can construct the functor by means of the following pushout squares in $\cCat$, where we use the inclusion of the target object $1\colon\bbone\to[1]$ and the collapse functor $\pi\colon[1]\to\bbone$ in the upper line.
\begin{equation}\label{eq:step-1-minus}
\vcenter{
\xymatrix{
n\cdot\bbone\ar[r]\ar[d]_-y&(n\cdot[1])^\lhd\ar[r]\ar[d]&\source{n}\ar[d]\\
C\ar[r]_-{j^-}&D^-\pushoutcorner\ar[r]_-{u^-}&C^-\pushoutcorner
}
}
\end{equation}
The functor $j^-$ is fully faithful by \autoref{prop:amalg-full}, and for every derivator~\D, the restriction morphism $(u^-)^\ast\colon\D^{C^-}\to\D^{D^-}$ separates the objects adjacent to the source. We denote by $\D^{D^-,\exx}\subseteq\D^{D^-}$ the full subderivator spanned by all diagrams $X\in\D^{D^-}$ such that $k^\ast X\in\D^{n\cdot[1]}$ consists of isomorphisms, where $k\colon n\cdot[1]\to (n\cdot[1])^\lhd\to D^-$ is the obvious functor.

\begin{prop}\label{prop:step-1}
The functor $u^-\colon D^-\to C^-$ is a homotopical epimorphism. Moreover, for every derivator~\D the essential image of $(u^-)^\ast\colon\D^{C^-}\to\D^{D^-}$ is  $\D^{D^-,\exx}$ and the resulting equivalence $(u^-)^\ast\colon\D^{C^-}\simeq\D^{D^-,\exx}$ is pseudo-natural with respect to arbitrary morphisms of derivators.
\end{prop}
\begin{proof}
This is an immediate application of \autoref{prop:refl-loc-htpical-epi}. In fact, the functor $u^-\colon D^- \to C^-$ is a reflective localization, a fully faithful right adjoint being given by the obvious functor $r\colon C^-\to D^-$ which sends $v$ to $v$ and which is the identity on $C$. Let us denote the resulting adjunction by
\[
(u^-,r,\eta\colon\id\to r\circ u^-,\ep=\id\colon u^-\circ r\to\id).
\]
The only non-identity components of the adjunction unit $\eta$ are those at $x_i\in D^-$ for $i=1,\ldots,n$ in which case they are given by 
\[
\eta_{x_i}\colon x_i\to y_i,\qquad i=1,\ldots,n.
\]
By \autoref{prop:refl-loc-htpical-epi} we conclude that $u^-$ is a homotopical epimorphism and that $X\in\D^{D^-}$ lies in the essential image of $(u^-)^\ast$ if and only if $X_{x_i}\to X_{y_i}$ is an isomorphism, which is to say that $X\in\D^{D^-,\exx}$.
\end{proof}

\begin{figure}
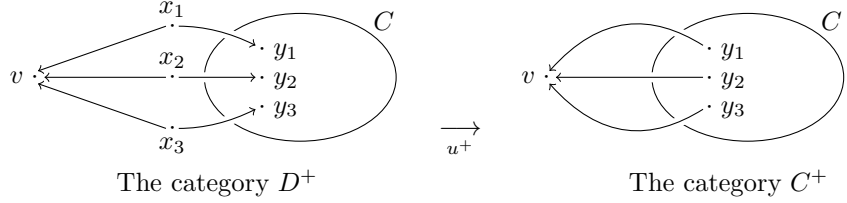

\centering
\TikzFigSeparateSink
\caption{The functor $u^+\colon D^+ \to C^+$ which contracts the edges $x_i \to y_i$. It is used to separate the sink of $C^+$.}
\label{fig:separate-C^+}
\end{figure}

The other inflation and deflation step turns out to be a bit more involved, and the situation is shown in \autoref{fig:separate-C^+}. We again have defining pushouts squares
\begin{equation}\label{eq:step-1-plus}
\vcenter{
\xymatrix{
n\cdot\bbone\ar[r]\ar[d]_-y&Z_n\ar[r]^-{q}\ar[d]^-{t}&\sink{n}\ar[d]^-{y^\rhd}\\
C\ar[r]_-{j^+}&D^+\pushoutcorner\ar[r]_-{u^+}&C^+,\pushoutcorner
}
}
\end{equation}
where $Z_n$ is the free category generated by the quiver
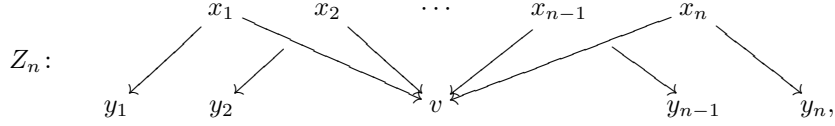
\begin{equation} \label{eq:Z_n}
\vcenter{
\hbox{$Z_n\colon$ \quad}
}
\vcenter{
\xymatrix{
& x_1 \ar[rrd] \ar[ld] & x_2 \ar[rd] \ar[ld]|(.34)\hole & \cdots & x_{n-1} \ar[ld] \ar[rd]|(.34)\hole & x_n \ar[lld] \ar[rd] \\
y_1 & y_2 && v && y_{n-1} & y_n,
}
}
\end{equation}
where $n\cdot\bbone\to Z_n$ classifies $y_1,\ldots,y_n$, and where $q\colon Z_n \to \sink{n}$ sends each $x_i$ and $y_i$ to the $i$-th copy of $\bbone$ and $v$ to the terminal object $\infty$. Assuming \autoref{hyp:y's}, both $j^+$ and $u^+j^+$ are fully faithful, and we again view $u^+j^+$ as an inclusion. As it will be important in further computations, we spell out what morphisms in $D^+$ and $C^+$ look like.	

\begin{lem} \label{lem:morphisms-C+D+}
~
\begin{enumerate}
\item Every non-identity morphism in the category $C^+$ has a unique expression of one of the forms $\gamma$, $\omega$, $\omega\gamma$, where $\gamma$ stands for a non-identity morphism of $C$ and $\omega$ stand for a non-identity morphism of $\sink{n}$.
\item Every non-identity morphism in the category $D^+$ has a unique expression of one of the forms $\gamma$, $\omega$, $\gamma\omega$, where $\gamma$ stands for a non-identity morphism of $C$ and $\omega$ stand for a non-identity morphism of $Z_n$.
\end{enumerate}
\end{lem}

\begin{proof}
This is an immediate consequence of \autoref{lem:normal-forms}.
\end{proof}

For every derivator \D we denote by $\D^{D^+,\exx}\subseteq\D^{D^+}$ the full subderivator formed by the coherent diagrams $X$ such that $X_{x_i} \to X_{y_i}$ is an isomorphism for every $i=1,\dots,n$.

\begin{prop} \label{prop:contract-sep-mor}
If $y\colon n\cdot\bbone \to C$ is injective on objects, then $u^+\colon D^+\to C^+$ \eqref{eq:step-1-plus} is a homotopical epimorphism. Moreover, for $\D\in\cDER$ the essential image of $(u^+)^\ast\colon\D^{C^+}\to\D^{D^+}$ is  $\D^{D^+,\exx}$ and the resulting equivalence $(u^+)^\ast\colon\D^{C^+}\simeq\D^{D^+,\exx}$ is pseudo-natural with respect to arbitrary morphisms of derivators.
\end{prop}
\begin{proof}
Let us fix a derivator $\D$ and let $\E=\D^{D^+,\exx}$. We show that \autoref{prop:detect} applies. Clearly $u^+$ is essentially surjective on objects and $\im((u^+)^\ast) \subseteq \E$. It remains to verify the assumption \autoref{prop:detect}(ii) and by (Der2) it suffices to check the invertibility of the unit $\eta$ at every $d\in D^+$. By \autoref{lem:detect-Kan-to-colim} this is the case if and only if the following instance of \eqref{eq:lkan-colim-map}
\begin{equation}\label{eq:contract-sep-mor}
(d,\id_{u^+d})^\ast p^\ast\to \colim_{(u^+/u^+d)}p^\ast 
\end{equation}
is invertible for every $d\in D^+$ and on $\D^{D^+,\exx}$. Here, $p\colon (u^+/u^+d)\to D^+$ is the canonical functor, and there are the following three cases. 

First, let $d = j^+(c),c\in C$, so that $u^+d=c$. Since $(j^+c,\id_c)\in(u^+/c)$ is a terminal object, by \autoref{lem:detect-final} the corresponding morphism \eqref{eq:contract-sep-mor} is an isomorphism on $\D^{D^+,\exx}$ if and only if $(j^+c,\id)^\ast p^\ast\to(j^+c,\id)^\ast p^\ast$ is an isomorphism on $\D^{D^+,\exx}$, and this is even true for all $X\in\D^{D^+}$. 

Suppose next that $d = x_i$ for some $i=1,\ldots, n$. In this case $u^+d = y_i \in C^+$ and it is easy to see that $(u^+/y_i)$ admits 
\[
(x_i, \id_{y_i})\to (y_i, \id_{y_i})
\]
as homotopy final subcategory, where the map is given by the freely attached map from $Z_n$. Two applications of \autoref{lem:detect-final} imply that we have to show that $x_i^\ast\to y_i^\ast$ is an isomorphism on $\D^{D^+,\exx}$ which is true by the defining exactness properties.

The remaining case is $d=v$. With the aid of \autoref{lem:morphisms-C+D+}, we divide the objects $w = (d', g\colon u^+(d') \to v)$ of $(u^+/v)$ into five disjoint classes, according to what $d'$ is and whether the structure morphism $g$ factors through a non-identity morphism in $C$. Each object $w \in (u^+/v)$ has exactly one of the following forms (where unlabeled arrows $y_i \to v$ always stand for the maps in $C^+$ coming from $\sink{n}$ in~\eqref{eq:C+C-})
\begin{enumerate}
\item $w = (v, \id_v)$,
\item $w = (x_i, y_i \to v)$ for some $i \in \{1,\ldots,n\}$,
\item $w = (y_i, y_i \to v)$ for some $i \in \{1,\ldots,n\}$,
\item $w = (j^+(c), c \overset{h}\to y_i \to v)$ for some $c \in C$ and non-identity map $h$ in $C$, or 
\item $w = (x_i, y_i \overset{h}\to y_j \to v)$ for $i,j \in \{1,\ldots,n\}$ and non-identity map $h$ in $C$.
\end{enumerate}
Let $H\subseteq(u^+/v)$ be the full subcategory spanned by the objects of type (i)--(iii). This category is a free category generated by the following quiver, where the object associated to which we wish to inspect the map \eqref{eq:contract-sep-mor} is in the box (for brevity we denote the objects only by the corresponding object of $D^+$), 
\begin{equation} \label{eq:htpy-final-v-case}
\vcenter{
\hbox{$H\colon$ \quad}
}
\vcenter{
\xymatrix{
& x_1 \ar[rrd] \ar[ld] & x_2 \ar[rd] \ar[ld]|(.34)\hole & \cdots & x_{n-1} \ar[ld] \ar[rd]|(.34)\hole & x_n \ar[lld] \ar[rd] \\
y_1 & y_2 && *+[F]{v} && y_{n-1} & y_n.
}
}
\end{equation}
Another short computation reveals that every object of type (v) admits a unique map in $(u^+/v)$ to the object of type (iv) with $c=y_i$ and the same morphism $h$ in $C$, and that every object of type (iv) admits a unique map in $(u^+/v)$ to an object of type (iii) obtained by stripping off $h$ from the structure morphism. In particular, the inclusion $H\to (u^+/v)$ is a right adjoint and hence homotopy final, so that \autoref{lem:detect-final} applies. As an upshot so far, the decoration of the objects in \eqref{eq:htpy-final-v-case} defines a functor $i\colon H\to D^+$ and it remains to show that the map
\[
v^\ast i^\ast(X)\to \colim_Hi^\ast X
\]
which is an instance of \eqref{eq:can-colim-map} is an isomorphism for all $X\in\D^{D^+,\exx}$.

To this end, let $j\colon H'\to H$ be the full subcategory of $H$ obtained by removing $y_i,i=1,\ldots,n$. It is straightforward to show that $j_!\colon\D^{H'}\to\D^H$ is fully faithful with essential image precisely those $Y\in\D^H$ such that $Y_{x_i}\to Y_{y_i}$ is invertible (compare to \cite[Prop.~3.12.(1)]{groth:ptstab}). In particular, for $X\in\D^{D^+,\exx}$ the restriction $i^\ast X$ belongs to this essential image, and \autoref{lem:detect-ff} reduces hence our task to show that $v^\ast\to\colim_{H'}$ \eqref{eq:can-colim-map} is an isomorphism on $j^\ast i^\ast X\in\D^{H'}$. By \autoref{lem:detect-final} this is even the case for every diagram in $\D^{H'}$ since $v\in H'$ is a terminal object.

To summarize, all assumptions of \autoref{prop:detect} are satisfied and $u^+$ is hence a homotopical epimorphism with essential image $\D^{D^+,\exx}$.
\end{proof}

Now we shall revoke \autoref{hyp:y's}.

\begin{cor} \label{cor:contract-sep-mor}
Let $C\in\cCat$, let $y_1,\ldots,y_n\in C$ (not necessarily distinct), and consider the functors $u^-\colon D^- \to C^-$ and $u^+\colon D^+ \to C^+$ constructed again by the pushouts \eqref{eq:step-1-minus} and \eqref{eq:step-1-plus}, respectively. Then $u^-$ and $u^+$ are still homotopical epimorphisms and the essential images are $\D^{D^-,\exx}$ and $\D^{D^+,\exx}$ defined by the same exactness conditions as in \autoref{prop:step-1} and \autoref{prop:contract-sep-mor}, respectively.
\end{cor}

\begin{proof}
We will discuss only $u^+$, the case of $u^-$ is similar. Suppose $y\colon n\cdot\bbone \to C$ is any functor. Thanks to \autoref{lem:separate-obj}(i) there is a factorization $y = p \tilde{y}$ such that $p\colon \widetilde C \to C$ is an equivalence of categories and $\tilde{y_1}, \dots, \tilde{y_n}$ are pairwise distinct objects in $\widetilde C$. Replacing $y$ by $\tilde y$ in~\eqref{eq:step-1-plus}, we obtain \autoref{prop:amalg-basic} and \autoref{lem:separate-obj}(ii) a diagram whose lower row changes only up to equivalence.
\end{proof}

Finally, we can establish the main result of this paper.

\begin{thm}\label{thm:reflect}
Let $C\in\cCat$, let $y_1,\ldots,y_n\in C$ (not necessarily distinct), and let $C^-,C^+\in\cCat$ be as in \eqref{eq:C+C-}. The categories $C^-$ and~$C^+$ are strongly stably equivalent.
\end{thm}
\begin{proof}
In \autoref{thm:reflect-sep} we constructed a pseudo-natural equivalence $\D^{D^-} \simeq \D^{D^+}$. It is direct from the construction of this equivalence that it restricts to a pseudo-natural equivalence $\D^{D^-,\exx} \simeq \D^{D^+,\exx}$. Invoking \autoref{cor:contract-sep-mor}, we obtain a chain
\[
\D^{C^-}\simeq\D^{D^-,\exx}\simeq\D^{D^+,\exx}\simeq\D^{C^+}
\]
of pseudo-natural equivalences. Putting them together, we obtain the strong stable equivalence
\begin{equation}\label{eq:sse-reflect}
(s^-,s^+)\colon\D^{C^-}\simeq\D^{C^+},
\end{equation}
concluding the proof.
\end{proof}

\begin{defn}
Let \D be a stable derivator, let $C\in\cCat$, let $y_1,\ldots,y_n\in C$ (not necessarily distinct), and let $C^-,C^+\in\cCat$ be as in~\eqref{eq:C+C-}. The components $s^-,s^+$ of the strong stable equivalence~in \eqref{eq:sse-reflect}, witnessing that $C^-\sse C^+$, are \textbf{(general) reflection morphisms}. 
\end{defn}

\section{Applications to abstract representations theory}
\label{sec:applications}

In this section we draw some consequence of the main theorem in this paper (\autoref{thm:reflect}). Since the categories $C^+$ and $C^-$ are strongly stably equivalent, we obtain abstract tilting results for various contexts. To begin with, let us specialize to representations over a ring.

\begin{eg}\label{eq:derived-cat-alg}
Let $R$ be a (possibly non-commutative) ring, let $C\in\cCat$, let $y_1,\ldots,y_n\in C$ (not necessarily distinct), and let $C^-,C^+\in\cCat$ be as in~\eqref{eq:C+C-}. 
\begin{enumerate}
\item There is an exact equivalence of categories $\D_R^{C^-}(\bbone)\stackrel{\Delta}{\simeq}\D_R^{C^+}(\bbone).$
\item If~$C$ has only finitely many objects, then the category algebras $RC^-$ and $RC^+$ are derived equivalent over $R$,
\begin{equation}\label{eq:der-eq-cat-alg}
D(RC^-)\stackrel{\Delta}{\simeq} D(RC^+).
\end{equation}
\end{enumerate}
\end{eg}

In fact, the first statement is \cite[Prop.~4.18]{groth:ptstab} while the second statement follows from \autoref{eg:grothendieck-shift}. However, having a strong stable equivalence is a stronger result in the following three senses.

\begin{enumerate}
\item Simply by choosing specific stable derivators, this yields exact equivalences of derived or homotopy categories of representations over rings or schemes, of differential graded representations, of spectral representations and of other types of representations (see \cite[\S5]{gst:basic}).
\item There are equivalences of derivators of representations as opposed to having mere equivalences of underlying categories. For example, in the case of homotopy derivators of combinatorial, stable model categories~$\cM$ it is a formal consequence of having an equivalence $\ho_\cM^Q\sim\ho_\cM^{Q'}$ and \cite{renaudin} that the corresponding model categories of representations $\cM^Q,\cM^{Q'}$ are related by a zigzag of Quillen equivalences.
\item The equivalences are pseudo-natural with respect to exact morphisms, hence commute with various types of morphisms like restriction of scalars, induction and coinduction of scalars, derived tensor and hom functors, localizations and colocalizations.
\end{enumerate}

With this added generality in mind, for the rest of the section we mostly focus on the shapes $C^-,C^+$. As a first instance, we recover the main result of \cite{gst:tree}.

\begin{thm}[{\cite[Corollary 9.23]{gst:tree}}]
Let $T$ be a finite oriented tree and let $T'$ be a reorientation of $T$. The trees $T$ and $T'$ are strongly stably equivalent.
\end{thm}
\begin{proof}
By an inductive argument, it suffices to show that if $T$ is as above and $t_0\in T$ is a source, then the reflected tree $T'=\sigma_{t_0}T$ and $T$ are strongly stably equivalent. But obviously $T=C^-$ and $T'=C^+$ for the full subcategory $C\subseteq T$ of $T$ obtained by removing $t_0$, hence \autoref{thm:reflect} concludes the proof.
\end{proof}

Increasing the class of shapes, we obtain the following.

\begin{thm}\label{thm:stable-acyclic}
Let $Q$ be a finite acyclic quiver and let $q_0\in Q$ be a source or a sink and let $Q'=\sigma_{q_0}Q$ be the reflected quiver. The two quivers $Q$ and $Q'$ are strongly stably equivalent.
\end{thm}
\begin{proof}
Assuming without loss of generality that $q_0$ is a source, we observe that $Q=C^-$ for the full subcategory $C\subseteq Q$ obtained by removing $q_0$. In this case one notes that $Q'=C^+$ and \autoref{thm:reflect} applies.
\end{proof}

\begin{rmk}
Specializing to the derivator $\D_k$ of a field $k$, \autoref{thm:stable-acyclic} yields exact equivalences of derived categories
\[
D(kQ)\stackrel{\Delta}{\simeq} D(kQ').
\]
The classical representation theory is more concerned with bounded derived categories of finite dimensional representations. However, as Rickard showed in~\cite[Corollary 8.3]{rickard:morita} (and its proof), any exact equivalence between the unbounded derived categories restricts to an exact equivalence of the corresponding bounded derived categories,
\[
D^b(kQ)\stackrel{\Delta}{\simeq} D^b(kQ')
\]
Hence the reflection functors yield such an equivalence and we recover a theorem of Happel (see~\cite[\S1.7]{happel:fd-algebra} and also the references therein).
\end{rmk}

In contrast to the case of trees, already for acyclic quivers it is not true that such quivers can be reoriented arbitrarily without affecting the abstract representation theory. If $Q,Q'$ are finite and without oriented cycles, then $Q,Q'$ being strongly stably equivalent still implies that $Q$ and $Q'$ have the same underlying graph (\cite[Proposition 5.3]{gst:basic}), but this condition is no longer sufficient. Let us consider the simplest case, where $Q$ is an orientation of an $n$-cycle:
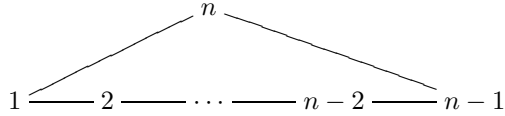
\[
\xymatrix{
&& n \\
1 \ar@{-}[r] \ar@{-}[urr] & 2 \ar@{-}[r] & \cdots \ar@{-}[r] & n-2 \ar@{-}[r] & n-1 \ar@{-}[ull]. }
\]
In representation theory one says that $Q$ is an \textbf{Euclidean} (or \textbf{extended Dynkin}) quiver of type $\widetilde{A}_{n-1}$, \cite{ringel:tame-alg,simson-skowronski:vol2}. Given such $Q$, put $c(Q) = \{p,q\}$, where $p$ is the number of arrows oriented clockwise and $q$ is the number of arrows oriented anti-clockwise. Then one obtains the following.

\begin{prop}
Let $Q,Q'$ be two orientations of an $n$-cycle, $n \ge 1$. Then $Q \sse Q'$ if and only if $c(Q) = c(Q')$.
\end{prop}

\begin{proof}
The sufficiency of the `clock condition' $c(Q) = c(Q')$ is easy. One quickly convinces oneself that given $Q$ with $c(Q) = \{p,q\}$, $p \le q$, after finitely many reflections at sinks or sources one gets a quiver isomorphic to
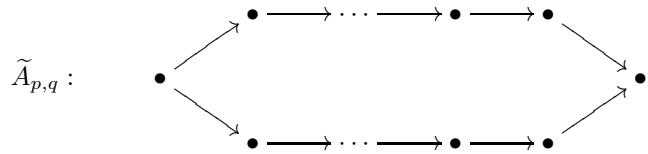
\begin{equation}
\xymatrix@R=1em{
&& \bullet \ar[r] & \cdots \ar[r] & \bullet \ar[r] & \bullet \ar[rd] \\
\widetilde{A}_{p,q}: & \bullet \ar[ur] \ar[dr] &&&&& \bullet  \\
&& \bullet \ar[r] & \cdots \ar[r] & \bullet \ar[r] & \bullet \ar[ru] 
}
\end{equation}
with $p$ arrows above and $q$ arrows below. Hence if $c(Q) = C(Q')$, one gets for any stable derivator \D a strong stable equivalence $\D^Q \simeq \D^{\widetilde{A}_{p,q}} \simeq \D^{Q'}$ by composing finitely many general reflection morphisms (\autoref{thm:reflect}).  

To prove the necessity, let $k$ be a field, $\D = \D_k$ be the derivator of $k$, and suppose that $\D_k^Q \simeq \D_k^{Q'}$. We shall appeal to results from representation theory and show that then $c(Q) = c(Q')$. The equivalence of derivators gives an equivalence of the underlying categories which in turn gives an equivalence of the subcategories of compact objects. In our case this means that the bounded derived categories of finitely generated modules of the corresponding path algebras are equivalent,
\[ D^b(kQ) \simeq D^b(kQ'). \] 

Now $kQ$ is a finite dimensional algebra over $k$ if and only if not all arrows have the same orientation if and only if $c(Q) \ne \{0,n\}$ if and only if all objects of $D^b(\modr kQ)$ have finite dimensional endomorphism rings. Thus $c(Q) = \{0,n\}$ if and only if $c(Q') = \{0,n\}$.

Suppose now that $c(Q), c(Q') \ne \{0,n\}$. Then $kQ$ is finite dimensional and we can construct a so-called Auslander--Reiten quiver of $D^b(kQ)$. This is an infinite quiver which is a useful combinatorial invariant of $D^b(kQ)$ and its general shape is described in~\cite[Corollary 4.5(ii)]{happel:fd-algebra}. A more precise description can be extracted from~\cite[Theorem~3.6.5, p.~158]{ringel:tame-alg} or~\cite[Proposition~XII.2.8]{simson-skowronski:vol2}. In particular, the numbers $p,q$, where $c(Q) = \{p,q\}$, can be read off the Auslander--Reiten quiver since it contains so-called tubes of ranks precisely $1,p,$ and $q$. Of course one can do the same for $Q'$ and hence $c(Q) = c(Q')$.
\end{proof}

\begin{rmk}
The existence of reflection equivalences in \autoref{thm:reflect} applies to more general shapes than finite, acyclic quivers.
\begin{enumerate}
\item First, neither the finiteness nor the acyclicity is needed. In fact, given an arbitrary quiver $Q$ with a source or a sink, \autoref{thm:reflect} yields a strong stable equivalence between $Q$ and the reflected quiver $Q'$. In particular, if $Q$ has finitely many objects only, the infinite-dimensional path algebras $kQ$ and $kQ'$ are derived equivalent for arbitrary fields $k$, and there are variants if we use rings as coefficients instead.
\item More generally, as noted in \autoref{eq:derived-cat-alg}, \autoref{thm:reflect} yields strong stable equivalences for shapes which are more general than quivers. To the best of the knowledge of the authors, even in the case that $R=k$ is a field, the result that the category algebras $kC^-$ and $kC^+$ are derived equivalent does not appear in the published literature.
\end{enumerate}
\end{rmk}

\appendix

\section{Amalgamation of categories}
\label{sec:amalgamation}

As is illustrated by the construction of abstract reflection functors, performing more complicated constructions in derivators often means that we need to ``glue together'' various small categories or diagram shapes. Formally we are speaking of pushouts of categories, which is a fairly complicated construction. As we need to understand some of these pushouts rather explicitly (for example, in order to be able to compute slice categories), here we discuss some basic properties of pushouts and amalgamations of small categories. We fix the following notation for the rest of the appendix.
\begin{equation} \label{eq:pushout-cat}
\vcenter{
\xymatrix{
W \ar[d]_-{f_X} \ar[r]^-{f_Y} & Y \ar[d]^-{g_Y} \\
X \ar[r]_-{g_X} & Z\pushoutcorner \\
}
}
\end{equation}

Often one is only interested in categories up to equivalences, but pushouts of small categories are, in general, not well behaved with pushouts. To address this issue, we include the following lemma.

\begin{lem} \label{lem:separate-obj}
Let $f_X\colon W \to X$ be a functor in $\cCat$.
\begin{enumerate}
\item There exists a factorization $f_X = p\circ f_{\widetilde{X}}$ such that $f_{\widetilde{X}}\colon W \to \widetilde{X}$ is injective on objects and $p\colon \widetilde{X} \to X$ is surjective on objects and an equivalence.
\item If $f_X$ is injective on objects and $f_Y\colon W \to Y$ in~\eqref{eq:pushout-cat} is an equivalence, then also $g_X\colon X \to Z$ is an equivalence.
\end{enumerate}
\end{lem}

\begin{proof}
Both are easy consequences of the existence of a (in fact unique) model structure on $\cCat$ with weak equivalences being the equivalences. This is a special case of a more general result in \cite{joyal-tierney:stacks}, and (i) is simply a factorization of $f_X$ into a cofibration followed by a trivial fibration. (ii) means that this model structure is left proper, which follows from the fact that every small category is cofibrant, \cite[Corollary 13.1.3]{hirschhorn:model}.
\end{proof}

For the rest of the section we adopt the following assumption and convention.

\begin{hyp} \label{hyp:embedding}
Assume that $f_X$ and $f_Y$ are honest inclusions of categories, that is, injective on objects and faithful. We will view $f_X$ and $f_Y$ as (not necessarily full) inclusions $W \subseteq X$ and $W \subseteq Y$, respectively.
\end{hyp}

\begin{defn} \label{defn:amalg}
The pushout~\eqref{eq:pushout-cat} is called an \textbf{amalgamation} if also $g_X$ and $g_Y$ are injective on objects and faithful. In this case we also view $g_X$ and $g_Y$ as inclusions $X \subseteq Z$ and $Y \subseteq Z$, respectively.
\end{defn}

\begin{rmk} \label{rem:amalg}
In the usual terminology of model theory, an amalgamation of the span $X \overset{f_X}\ot W \overset{f_Y}\to Y$ would in fact mean \emph{any} commutative square like~\eqref{eq:pushout-cat} (i.e., not necessarily a pushout) for which $g_X$ and $g_Y$ are inclusions. But if such a square exists, the pushout square is also an amalgamation in this sense. 
\end{rmk}

As shown in~\cite[Example 4.4]{macdonald-scull:amalgamations}, not every pushout of inclusions is an amalgamation. On the other hand, a sufficient condition for the existence of amalgamations is given in the same paper.

\begin{defn} \label{defn:3-for-2-amalg}
A functor $f\colon W \to Y$ has the \textbf{3-for-2 property} if, whenever $\alpha$ and $\beta$ are two composable morphisms in $Y$ and two of $\alpha, \beta, \beta\alpha$ belong to the honest (not just essential) image of $f$, then so does the third.
\end{defn}

\begin{prop}[{\cite[Theorem 3.3]{macdonald-scull:amalgamations}}] \label{prop:amalg-basic}
Suppose that $f_X\colon W \to X$ and $f_Y\colon W \to Y$ are functors in $\cCat$ which are injective on objects, faithful, and have the 3-for-2 property. Then their pushout~\eqref{eq:pushout-cat} is an amalgamation.
\end{prop}

\begin{rmk} \label{rmk:3-for-2}
The result is rather subtle in that it is \emph{not} enough to assume that only one of $f_X$ and $f_Y$ has the 3-for-2 property; see~\cite[Example 4.4]{macdonald-scull:amalgamations} again. Note that $f\colon W \to Y$ has the 3-for-2 property for example if $f$ is fully faithful or if $W$ is a groupoid (so in particular if $W$ is a discrete category as in \S\S\ref{sec:glue}-\ref{sec:glue-epi} and \S\ref{sec:reflection}).
\end{rmk}

For practical purposes it will be convenient to know that the 3-for-2 property transfers via amalgamations, i.e., that also the functors $g_X$ and $g_Y$ have it. Once we know this, we can iterate the amalgamation process. Here we need refine the argument in~\cite{macdonald-scull:amalgamations}.

We first recall details about the construction of a pushout in $\cCat$. At the level of objects, we simply construct the pushout of sets.
The morphisms in the pushout are more interesting, see \cite[\S2]{macdonald-scull:amalgamations} for details. To this end, we denote by $\overline{Z}$ the pushout of the sets of morphisms of $X$ and $Y$ over the set of morphisms of $W$. In particular an element of $\overline{Z}$ which comes from both $X$ and $Y$ comes already from $W$ by our standing assumption. Every morphism in $Z$ is represented by a finite sequence
\[ (\alpha_1,\alpha_2, \dots, \alpha_n) \]
of length $n \ge 1$ in $\overline{Z}$, subject to the condition that codomain of $\alpha_{i+1}$ always agrees with the domain of $\alpha_i$. The composition of morphisms is simply given by concatenation. Of course we must identify some of these sequences. To do so, we first define a partial order on the set of allowable sequences of elements of $\overline{Z}$ which is generated by the \emph{elementary reductions}%
\begin{equation} \label{eq:el-red}
(\alpha_1,\dots,\alpha_i,\alpha_{i+1}, \dots, \alpha_n) > (\alpha_1,\dots,\alpha_i\alpha_{i+1}, \dots, \alpha_n),
\end{equation} 
where either both $\alpha_i$ and $\alpha_{i+1}$ are morphisms from $X$ and the composition on the right takes place in $X$, or symmetrically $\alpha_i$ and $\alpha_{i+1}$ are from $Y$ and we compose them in $Y$. This \emph{reduction order} is of course a binary relation and by taking its symmetric and transitive closure, we obtain an equivalence relation. The morphisms in $Z$ are then precisely the equivalence classes of allowable sequences in $\overline{Z}$.

For convenience, we introduce the following notation. Given an allowable sequence $\gamma = (\alpha_1,\alpha_2,\dots, \alpha_n)$, we denote the equivalence class of $\gamma$ by $[\alpha_1,\alpha_2,\dots, \alpha_n]$, and we view this equivalence class as a partially ordered set with the restriction of the reduction order above. The following is a key observation.

\begin{lem} \label{lem:Z-reducible}
Suppose that $\gamma = (\alpha_1)$ consist of a single element of $\overline{Z}$. Then $\gamma$ is the unique minimal element of $[\alpha_1]$ with respect to the reduction order.
\end{lem}

\begin{proof}
This is exactly what the first paragraph of the proof of \cite[Theorem 3.3]{macdonald-scull:amalgamations} asserts. For a very detailed proof we refer to the rest of the proof of Theorem 3.3 and to \S5 in op.\ cit.
\end{proof}

Now we can complement \autoref{prop:amalg-basic} with the promised result which will allow for iterated amalgamations.

\begin{prop} \label{prop:amalg-iterate}
Suppose that $f_X\colon W \to X$ and $f_Y\colon W \to Y$ are injective on objects and faithful functors with the 3-for-2 property. Then, in their pushout amalgamation~\eqref{eq:pushout-cat}, also $g_X$ and $g_Y$ have the 3-for-2 property. 
\end{prop}

\begin{proof}
By symmetry we only need to treat $g_X$. Suppose that $\alpha_1,\beta$ are composable morphisms in $Z$ and that $\alpha_1$ and $\alpha_1\beta$ both belong to $X$. We must show that $\beta$ belongs there as well.

To this end, $\beta$ can be represented by a suitable sequence $\gamma = (\alpha_2, \dots, \alpha_n)$ of elements of $\overline{Z}$. Then $\alpha_1\beta$ is represented by $\delta = (\alpha_1, \alpha_2, \dots, \alpha_n)$ and, by \autoref{lem:Z-reducible}, $[\alpha_1, \alpha_2, \dots, \alpha_n]$ has the unique minimal element $(\alpha_1\beta)$ with respect to the reduction order. We shall prove by induction on $n$ that $\beta$ is in $X$. 

Suppose first $n=2$. In this case $\beta = \alpha_2$ belongs either to $X$ or $Y$.
If $\beta$ is in $X$, we are done. If $\beta$ is in $Y$, we know by the above that $(\alpha_1,\beta) > (\alpha_1\beta)$ in the reduction order on $[\alpha_1\beta]$. By definition of the reduction order, the latter must be an elementary reduction and hence all $\alpha_1,\beta,\alpha_1\beta$ belong to $X$ or all three belong to $Y$. In the first case we are done and in the second case we know that $\alpha_1, \alpha_1\beta \in X \cap Y = W$. Hence $\beta \in W\subseteq X$ by the 3-for-2 property of $f_Y\colon W \overset{\subseteq}\to Y$.

If now $n>2$, there is an elementary reduction
\[
(\alpha_1, \alpha_2, \dots, \alpha_i, \alpha_{i+1}, \dots, \alpha_n) > (\alpha_1, \alpha_2, \dots, \alpha_i\alpha_{i+1}, \dots, \alpha_n)
\]
Let us choose such a reduction with maximal possible $i$. Two situations may occur. If $i>1$, then by the very definition of elementary reductions we have that $(\alpha_2, \dots, \alpha_n) > (\alpha_2, \dots, \alpha_i\alpha_{i+1}, \dots, \alpha_n)$ and also that $\beta$ is in $X$ by the induction hypothesis.

Suppose on the other hand that $i=1$. We claim that in such a case $\alpha_2$ is in $X$. To this end, assume by way of contradiction that $\alpha_2 \in Y \setminus W$. Then $\alpha_1 \in W$ since we have the reduction $(\alpha_1,\alpha_2,\alpha_3,\ldots,\alpha_n)>(\alpha_1\alpha_2,\alpha_3,\ldots,\alpha_n)$. Consequently $\alpha_1\alpha_2 \in Y \setminus W$ since otherwise $\alpha_1,\alpha_1\alpha_2 \in W$ would imply $\alpha_2\in W$. Finally, since the sequence $(\alpha_1\alpha_2, \alpha_3,\dots,\alpha_n)$ must reduce further, the maximality of $i=1$ implies
\[ (\alpha_1\alpha_2, \alpha_3, \dots, \alpha_n) > (\alpha_1\alpha_2\alpha_3, \dots, \alpha_n). \]
Now $\alpha_1\alpha_2 \in Y \setminus W$, so $\alpha_3 \in Y$ in order for the reduction to be defined. However, then we also have an elementary reduction
\[ (\alpha_1, \alpha_2, \alpha_3, \dots, \alpha_n) > (\alpha_1, \alpha_2\alpha_3, \dots, \alpha_n). \]
contradicting the maximality of $i$. This proves the claim.

To summarize, we have $\alpha_1,\alpha_2 \in X$. Now let $\alpha' = \alpha_1\alpha_2 \in X$ and $\beta'$ be the equivalence class $[\alpha_3,\dots,\alpha_n]$. Then $\alpha', \alpha'\beta' \in X$ and we infer by the inductive hypothesis that $\beta' \in X$. Then clearly $\beta = \alpha_2\beta'$ is in $X$, which finishes the induction.

The case when $\alpha,\beta$ are composable in $Z$ and $\beta, \alpha\beta$ are in $X$ is similar.
\end{proof}

As pointed out in \cite{macdonald-scull:amalgamations}, a special case when a functor has the 3-for-2 property is when it is fully faithful. Under our usual assumptions, it turns out that also full faithfulness is compatible with amalgamations. This has been observed already in~\cite{trnkova:sums-catgs}, and we include a short proof for the convenience of the reader.

\begin{prop} \label{prop:amalg-full}
Suppose that $f_X\colon W \to X$ and $f_Y\colon W \to Y$ are injective on objects. If $f_X$ is fully faithful and $f_Y$ is faithful and has the 3-for-2 property, then, in the pushout amalgamation~\eqref{eq:pushout-cat}, $g_Y\colon Y \to Z$ is fully faithful.
\end{prop}

\begin{proof}
We only need to prove that $g_Y$ is full. Suppose that we are given a morphism in $Z$, represented by a sequence $(\alpha_1, \alpha_2, \dots, \alpha_n)$ in $\overline{Z}$ such that the domain of $\alpha_n$ and the codomain of $\alpha_1$ belong to $Y$. By possibly reducing this sequence, we may assume that $\alpha_i$ belongs to $Y$ for $i$ odd and to $X$ for $i$ even. If $i$ is even, the domain and the codomain of $\alpha_i$ must be objects in $X \cap Y = W$. Since $f_X$ is full, $\alpha_i$ is a morphism in $W$, and hence also in $Y$. Thus all the $\alpha_i$ in fact belong to $Y$ and so does their composition.
\end{proof}

Finally, we consider the case where $W$ is a discrete category (recall \autoref{rmk:3-for-2}). The main advantage is that, analogous to the situation with free products of monoids, all morphisms of a pushout amalgamation of two categories over a discrete category have unique reduced factorizations to morphisms of the original categories (see \autoref{lem:glue}(iv) for an illustration). To state this precisely, we call an allowable sequence $(\alpha_1,\dots, \alpha_n)$ of elements of $\overline{Z}$ \emph{reduced} if it is minimal with respect to the reduction order. For $W$ discrete, the following stronger version of \autoref{lem:Z-reducible} holds.

\begin{lem} \label{lem:normal-forms}
Suppose that $W$ is a discrete category and $f_X\colon W \to X$ and $f_Y\colon W \to Y$ are injective on objects. Given any morphism in $Z$ represented by a sequence $\gamma = (\alpha_1,\dots, \alpha_n)$ in $\overline{Z}$, the equivalence class $[\alpha_1,\dots, \alpha_n]$ has a unique minimal element in the reduction order.

In other words, each non-identity morphism $\beta$ in $Z$ uniquely factors as $\beta = \alpha_1\cdots\alpha_n$, where each $\alpha_i$ belongs to $X$ or $Y$, but no composition $\alpha_i\alpha_{i+1}$ belongs to $X$ or $Y$.
\end{lem}

\begin{proof}
Suppose that we have two elementary reductions of out sequence $\gamma$,
\begin{equation} \label{eq:two-el-red}
(\alpha_1,\dots, \alpha_i\alpha_{i+1}, \dots, \alpha_n) < \gamma > (\alpha_1,\dots, \alpha_j\alpha_{j+1}, \dots, \alpha_n),
\end{equation}
where $i\le j$ without loss of generality. We claim that there is a common predecessor. This is clear if $i=j$ and easy if $j-i \ge 2$ as then both the reductions further reduce to $(\alpha_1,\dots, \alpha_i\alpha_{i+1},\dots, \alpha_j\alpha_{j+1}, \dots, \alpha_n)$. If $j = i+1$, there are two cases. First, all of $\alpha_i,\alpha_{i+1},\alpha_{i+2}$ may belong to one of $X$ or $Y$. Then $(\alpha_1,\dots, \alpha_i\alpha_{i+1}\alpha_{i+2}, \dots, \alpha_n)$ is the common predecessor which we look for. Second, two of $\alpha_i,\alpha_{i+1},\alpha_{i+2}$ may belong to $X$ and one to $Y$, or vice versa. Then, since both the reductions from \eqref{eq:two-el-red} were possible, it is easy to check that in all possible distributions of $\alpha_i,\alpha_{i+1},\alpha_{i+2}$ among $X$ and $Y$, we always get that one of $\alpha_i,\alpha_{i+1},\alpha_{i+2}$ belongs to $W = X \cap Y$, so it is the identity morphism.  If $\alpha_{i+1} = \id$, the original reductions are equal, and in the remaining cases $(\alpha_1,\dots, \alpha_i\alpha_{i+1}\alpha_{i+2}, \dots, \alpha_n)$ is a common predecessor of the two. This proves the claim.

An easy induction argument shows now that $\big([\alpha_1,\ldots,\alpha_n],<\big)$ is a downwards directed poset. Together with the obvious fact that the reduction order satisfies the descending chain condition, it follows that $\big([\alpha_1,\ldots,\alpha_n],<\big)$ has a unique minimal (=reduced) element.
\end{proof}

\bibliographystyle{alpha}
\bibliography{tilting4}

\end{document}